\documentclass[11pt]{amsart}

\usepackage{amsmath}
\usepackage{amsfonts}
\usepackage{amssymb}
\usepackage{amsthm}
\usepackage{graphicx}
\usepackage{hyperref}
\usepackage{pb-diagram}
\usepackage{epstopdf}
\usepackage{amscd}
\usepackage{pdfpages}
\usepackage{bbm}
\usepackage{multirow}
\usepackage[mathscr]{eucal}
\usepackage{epsfig,epsf,epic}
\usepackage{pdfsync}
\usepackage{enumerate}
\usepackage{etex}
\usepackage[all]{xy}
\usepackage{fancyref}
\usepackage{comment}

\newtheorem{theorem}{Theorem}[section]
\newtheorem{prop}[theorem]{Proposition}

\newtheorem{definition}[theorem]{Definition}

\newtheorem{lemma}[theorem]{Lemma}
\newtheorem{remark}[theorem]{Remark}
\newtheorem{assum}[theorem]{Assumption}
\newtheorem{notation}[theorem]{Notations}
\numberwithin{equation}{section}

\newcommand{\real}{\mathbb{R}}
\newcommand{\comp}{\mathbb{C}}

\newcommand{\inte}{\mathbb{Z}}

\newcommand{\pd}{\partial}
\newcommand{\pdb}{\bar{\partial}}
\newcommand{\dd}[1]{\frac{\partial}{\partial #1}}

\newcommand{\half}{\frac{1}{2}}
\newcommand{\ddd}[2]{\frac{\partial#1}{\partial{#2}}}

\newcommand{\pdpd}[3]{\frac{\partial^2 #1}{\partial #2\partial #3}}

\newcommand{\ppd}[1]{\frac{\partial^2}{\partial #1^2}}
\newcommand{\Hess}{\nabla^2}

\newcommand{\bmc}{z}

\newcommand{\facs}{\varphi}
\newcommand{\ang}{\theta}
\newcommand{\incoming}{\Pi}
\newcommand{\onewall}{\Pi}
\newcommand{\eincoming}{\breve{\Pi}}
\newcommand{\emc}{\varPsi}
\newcommand{\uperp}{u_{m,\perp}}
\newcommand{\rhoperp}{\rho_{m,\perp}}
\newcommand{\tree}[1]{\mathbb{T}_{#1}}
\newcommand{\rtree}[1]{\mathbb{R}\mathbb{T}_{#1}}
\newcommand{\tr}{T}
\newcommand{\rtr}{\mathbf{T}}
\newcommand{\ltr}{\mathsf{T}}
\newcommand{\lrtr}{\mathcal{T}}
\newcommand{\ltree}[1]{\mathtt{L}\mathtt{T}^{#1}}
\newcommand{\lrtree}[1]{\mathtt{L}\mathtt{R}\mathtt{T}^{#1}}

\newcommand{\hp}{\hslash}  
\newcommand{\he}{\hslash}  

\newcommand{\filt}{\mathcal{W}}

\begin{document}

\title[Scattering diagrams from Maurer-Cartan equations]{Scattering diagrams from asymptotic analysis\\ on Maurer-Cartan equations}
\author[Chan]{Kwokwai Chan}
\address{Department of Mathematics\\ The Chinese University of Hong Kong\\ Shatin\\ Hong Kong}
\email{kwchan@math.cuhk.edu.hk}
\author[Leung]{Naichung Conan Leung}
\address{The Institute of Mathematical Sciences and Department of Mathematics\\ The Chinese University of Hong Kong\\ Shatin \\ Hong Kong}
\email{leung@math.cuhk.edu.hk}
\author[Ma]{Ziming Nikolas Ma}
\address{The Institute of Mathematical Sciences and Department of Mathematics\\ The Chinese University of Hong Kong\\ Shatin \\ Hong Kong}
\email{zmma@ims.cuhk.edu.hk}

\begin{abstract}
Let $\check{X}_0$ be a semi-flat Calabi-Yau manifold equipped with a Lagrangian torus fibration $\check{p}:\check{X}_0 \rightarrow B_0$.
We investigate the asymptotic behavior of Maurer-Cartan solutions of the Kodaira-Spencer deformation theory on $\check{X}_0$
by expanding them into Fourier series along fibres of $\check{p}$ over a contractible open subset $U\subset B_0$, following a program set forth by Fukaya \cite{fukaya05} in 2005.
We prove that semi-classical limits (i.e. leading order terms in asymptotic expansions) of the Fourier modes of a specific class of Maurer-Cartan solutions
naturally give rise to consistent scattering diagrams, which are tropical combinatorial objects that have played a crucial role in works of Kontsevich-Soibelman \cite{kontsevich-soibelman04} and Gross-Siebert \cite{gross2011real} on the reconstruction problem in mirror symmetry.
\end{abstract}

\maketitle

\section{Introduction}

\subsection{Background}

The celebrated Strominger-Yau-Zaslow (SYZ) conjecture \cite{syz96} asserts that mirror symmetry is a {\em T-duality}, meaning that a mirror pair of Calabi-Yau manifolds should admit fibre-wise dual (special) Lagrangian torus fibrations to the same base. This immediately suggests a construction of the mirror (as a complex manifold):
Given a Calabi-Yau manifold $X$, one first looks for a Lagrangian torus fibration $p: X \to \check{B}$.
The base $\check{B}$ is then an integral affine manifold with singularities.
Letting $\check{B}_0 \subset \check{B}$ be the smooth locus and setting
$$\check{X}_0 := T\check{B}_0 / \Lambda_{\check{B}_0},$$
where $\Lambda_{\check{B}_0} \subset T\check{B}_0$ denotes the natural lattice locally generated by affine coordinate vector fields, yields a torus bundle $\check{p}: \check{X}_0 \to \check{B}_0$ which admits a natural complex structure $\check{J}_0$, called the {\em semi-flat complex structure}. This would not produce the correct mirror in general,\footnote{Except in the semi-flat case when $\check{B} = \check{B}_0$ where there are no singular fibres; see \cite{Leung05}.} simply because $\check{J}_0$ cannot be extended across the singular points $\check{B}^{\text{sing}}$. But the SYZ proposal suggests that the mirror is given by deforming $\check{J}_0$ using {\em quantum corrections} coming from holomorphic disks in $X$ with boundary on the Lagrangian torus fibres of $p$.

The precise mechanism of such a mirror construction was first depicted by Kontsevich-Soibelman \cite{kontsevich00} using rigid analytic geometry and then by Fukaya \cite{fukaya05} using asymptotic analysis.
In Fukaya's proposal, he described how instanton corrections would arise near the large volume limit given by scaling of the symplectic structure on $X$ by $\hp \in \real_{>0}$, which is mirrored to scaling of the complex structure $\check{J}_0$ on $\check{X}_0$. It was conjectured that the desired deformations of $\check{J}_0$ were given by a specific class of solutions to the Maurer-Cartan equation of the Kodaira-Spencer deformation theory of complex structures on $\check{X}_0$, whose expansions into Fourier modes along torus fibres of $\check{p}$ would have semi-classical limits (i.e. leading order terms in asymptotic expansions as $\hp \to 0$) concentrated along gradient flow trees of a canonically defined multi-valued Morse function on $\check{B}_0$ \cite[Conjecture 5.3]{fukaya05}. On the mirror side, holomorphic disks in $X$ with boundary on fibres of $p$ were conjectured to collapse to gradient flow trees emanating from the singular points $\check{B}^{\text{sing}} \subset \check{B}$ \cite[Conjecture 3.2]{fukaya05}.
From this one sees directly how the mirror complex structure is determined by quantum corrections.
Unfortunately, the arguments in \cite{fukaya05} were only heuristical and the analysis involved to make them precise seemed intractable at that time.

These ideas were later exploited by Kontsevich-Soibelman \cite{kontsevich-soibelman04} (for dimension 2) and Gross-Siebert \cite{gross2011real} (for general dimensions) to construct families of rigid analytic spaces and formal schemes respectively from integral affine manifolds with singularities, thereby solving the very important {\em reconstruction problem} in SYZ mirror symmetry. They cleverly got around the analytical difficulties, and instead of solving the Maurer-Cartan equation, used gradient flow trees in $\check{B}_0$ \cite{kontsevich-soibelman04} or tropical trees in the Legendre dual $B_0$ \cite{gross2011real} to encode the modified gluing maps between charts in constructing the mirror family. A key notion in their constructions is that of {\em scattering diagrams}, which are combinatorial structures encoding possibly very complicated gluing data. It has also been understood (by works of these authors and their collaborators, notably \cite{gross2010tropical}) that these scattering diagrams encode Gromov-Witten data as well.

In this paper, we revisit Fukaya's original ideas and apply asymptotic analysis motivated by Witten-Morse theory \cite{witten82}. Our primary goal is to connect consistent scattering diagrams to the asymptotic behavior of a specific class of solutions of the Maurer-Cartan equation. In particular we prove a modified version of (the ``scattering part'' of) Fukaya's original conjecture in \cite{fukaya05}. As pointed out by Fukaya himself, understanding scattering phenomenon is vital to a general understanding of quantum corrections in mirror symmetry.

We start with a Calabi-Yau manifold $X$ (regarded as a symplectic manifold) equipped with a Lagrangian torus fibration which admits a Lagrangian section $\mathtt{s}$
$$
\xymatrix@1{ (X,\omega,J) \ar[rr]_{p} &  &\check{B} \ar@(ul,ur)[ll]^{s} }
$$
and whose discriminant locus is given by $\check{B}^{\text{sing}} \subset \check{B}$, over which the integral affine structure develops singularities.
Restricting $p$ to the smooth locus $\check{B}_0 = \check{B} \setminus \check{B}^{\text{sing}}$, we obtain a semi-flat symplectic Calabi-Yau manifold
$X_0 \hookrightarrow X,$
which, by Duistermaat's action-angle coordinates \cite{Duistermaat80}, can be identified as a quotient of the cotangent bundle of the base
$ X_0 \cong T^*\check{B}_0 / \Lambda^\vee_{\check{B}_0},$
where $\Lambda^\vee_{\check{B}_0} \subset T^*\check{B}_0$ is the natural lattice (dual to $\Lambda_{\check{B}_0}$) locally generated by affine coordinate 1-forms. We then have a pair of fibre-wise dual torus bundles over the same base:
$$
\xymatrix{
X_0 = T^*\check{B}_0 / \Lambda^\vee_{\check{B}_0} \ar[dr]_{p} & & \check{X}_0 = T\check{B}_0 / \Lambda_{\check{B}_0} \ar[dl]^{\check{p}}\\
& \check{B}_0 &}
$$

We scale both the complex structure on $\check{X}_0$ and the symplectic structure on $X_0$ by introducing a $\real_{>0}$-valued parameter $\hp$ (so that $\hp \to 0$ give the respective large structure limits) and consider the family of spaces (as well as the associated dgLa's) parametrized by $\hp$.

As suggested by Fukaya \cite{fukaya05} (and motivated by the relation between Morse theory and de Rham theory \cite{witten82, HelSj4, klchan-leung-ma}), we consider the Fourier expansion (see Definition \ref{def:fourier_transform}) of the Kodaira-Spencer differential graded Lie algebra (dgLa)
$(KS_{\check{X}_0} = \Omega^{0,*}({\check{X}_0},T^{1,0}\check{X}_0), \bar{\partial}, [\cdot,\cdot])$
associated to $\check{X}_0$ along fibres of $\check{p}$, and try to solve the Maurer-Cartan (abbrev. MC) equation
\begin{equation}\label{eqn:MC_eqn_A-side}
\pdb \Phi + \half \left[ \Phi, \Phi \right] = 0.
\end{equation}

\begin{remark}
The idea that Fourier-type transforms should be responsible for the interchange between symplectic-geometric data on one side and complex-geometric data on the mirror side (i.e. T-duality) came from the original SYZ proposal \cite{syz96}. This has been applied successfully in the toric case: see
\cite{Hori-Vafa00, kontsevich00, cho05, Cho-Oh06, chanleung10, chanleung08, FOOO-toricI, FOOO-toricII, FOOO-toricIII, Abouzaid06, Abouzaid09, Fang08, FLTZ12}
for compact toric varieties
and
\cite{Leung-Vafa98, HIV00, Gross01, Gross-Inventiones, Auroux07, Auroux09, AAK12, cll12, CCLT13, Gross-Siebert_ICM, Lau14}
for toric Calabi-Yau varieties.
Nevertheless, no scattering phenomenon was involved in those examples.
\end{remark}

\subsection{Main results}

Before describing our main results, we first choose a Hessian type metric (see Definition \ref{Hessian_type_g}) on the affine manifold $\check{B}_0$ which allows us to apply the Legendre transform (see Section \ref{semiflat_kahler}) and work with the {\em Legendre dual} $B_0$.
This originates from an idea of Gross-Siebert \cite{gross2011real} who suggested that, while tropical trees on $B_0$ correspond to Morse gradient flow trees on $\check{B}_0$ under the Legendre transform, the former are easier to work with because of their linear nature.

We will also choose a convex open subset $U \subset B_0$, fix a codimension $2$ tropical affine subspace $Q \subset U$ and work locally around $Q$.\footnote{In the language of the Gross-Siebert program \cite{gross2011real}, we are working locally near a {\em joint} (i.e. a codimension 2 cell) in a polyhedral decomposition of the singular set $\text{Sing}(\mathscr{D})$ of a scattering diagram $\mathscr{D}$.}
In $U$, a scattering diagram can be viewed schematically as the process of how new walls are being created from the transversal intersection between non-parallel walls supported on tropical hyperplanes in $U$. The combinatorics of this process is governed by the algebra of the {\em tropical vertex group} \cite{gross2010tropical}, which will be reviewed in Section \ref{recallscattering}.

We work with dgLa's over the {\em formal} power series ring $R = \comp[[t]]$ where $t$ is a formal deformation variable.
Our goal is to investigate the relation between the scattering process and solutions of the MC equation of the Kodaira-Spencer dgLa $KS_{\check{X}_0}[[t]]$.
\footnote{There are other approaches to the scattering process or wall-crossing formulas such as \cite{BTL12, FGFS17, Reineke10}.}

To begin with, let $\mathbf{w} = (P, \Theta)$ be a single wall
supported on a tropical hyperplane $P \subset U$ containing $Q$ (although $Q$ does not play any role in this single wall case) and equipped with a wall-crossing factor $\Theta$ (as an element in the tropical vertex group). Our first aim is to see how $\Theta$ is related to solutions of the MC equation \eqref{eqn:MC_eqn_A-side}.

Recall that in Witten-Morse theory \cite{witten82, HelSj4, klchan-leung-ma}, the shrinking of a fibre-wise loop $m \in \pi_1(p^{-1}(x),s(x))$ towards a singular fibre indicates the presence of a critical point of the symplectic area function $f_m$ in the singular locus (in $B$), and the union of gradient flow lines emanating from the singular locus should be interpreted as a stable submanifold associated to that critical point. Furthermore, this codimension one stable submanifold should correspond to a bump differential $1$-form with support concentrated along $P$ (see \cite{klchan-leung-ma}).

Inspired by this, given a wall $\mathbf{w}$, we are going to write down an ansatz
$\incoming \in KS^1_{\check{X}_0}[[t]]$
solving \eqref{eqn:MC_eqn_A-side}; 
see Definition \ref{ansatz} for the precise formula. Since $\check{X}_0(U) := \check{X}_0\times_{B_0} U$ does not admit any non-trivial deformations, the MC solution $\incoming$ is gauge equivalent to $0$, i.e. there exists $\varphi \in KS^0_{\check{X}_0}[[t]]$ such that $e^{\varphi} * 0 = \incoming;$
we further use a gauge fixing condition ($\hat{P}\varphi = 0$) to uniquely determine the gauge $\varphi$.

In Proposition \ref{prop:MC_sol_one_wall}, we demonstrate how the semi-classical limit (as $\hp \rightarrow 0$) of $\facs$ determines the wall-crossing factor $\Theta$ (or more precisely, $\text{Log}(\Theta)$); see the introduction of Section \ref{onewall} for a more detailed description.
Moreover, the support of the bump-form-like MC solution $\incoming$ (see Figure \ref{fig:delta_function}) is more and more concentrated along $P$ as $\hp \rightarrow 0$. In Definition \ref{asypmtotic_support_def}, we make precise the key notion of having {\em asymptotic support on $P$} to describe such asymptotic behavior. We further show that {\em any} MC solution $\incoming$ with asymptotic support on $P$ would give rise to the {\em same} wall crossing factor $\Theta$ in Section \ref{sec:asy_support} (see Remark \ref{general_input}).


At this point we are ready to explain the main results of this paper.
From now on, unlike the case of a single wall, we will be solving the Maurer-Cartan equation {\em only up to error terms with exponential order in $\hp^{-1}$}, i.e. terms of the form $O(e^{-c/\hp})$.
This is sufficient for our purpose because those error terms tend to zero as one approaches the large volume/complex structure limits when $\hp \to 0$, and thus they do not contribute to the semi-classical limits of the MC solutions and the associated scattering diagrams.\footnote{This point was also anticipated by Fukaya in \cite{fukaya05}.}
To make this precise, we introduce in Section \ref{sec:MC_error} a dgLa $\widehat{\mathbf{g}^*/\mathcal{E}^*}(U)$ which is a quotient of a sub-dgLa of $KS_{\check{X}_0}(U)[[t]]$, and we will work with and construct MC solutions of $\widehat{\mathbf{g}^*/\mathcal{E}^*}(U)$.

Our first main result relates a specific class of MC solutions (satisfying the two assumptions described below) to {\em consistent} scattering diagrams (see Definition \ref{consistent_def} for the precise meaning of consistency).
Suppose that we have a countable collection $\{P_a\}_{a \in \mathbb{W}}$ of tropical half-hyperplanes (supports of the walls) sharing the codimension $2$ tropical affine subspace $Q$ as their common boundary, as shown in Figure \ref{fig:introduction_collection_of_walls}.
\begin{figure}[h]
\centering
\includegraphics[scale=0.26]{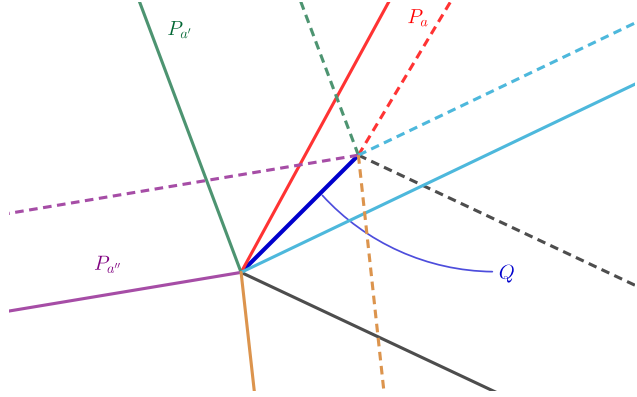}
\caption{A collection of walls sharing a common boundary $Q$}
\label{fig:introduction_collection_of_walls}
\end{figure}
We consider a Maurer-Cartan solution $\Phi \in \widehat{\mathbf{g}^*/\mathcal{E}^*}(U) \otimes_R \mathbf{m}$ which admits a Fourier decomposition
\begin{equation}\label{eqn:Phi_decomposition}
\Phi = \sum_{a \in \mathbb{W}} \Phi^{(a)},
\end{equation}
where the sum is finite modulo $\mathbf{m}^{N+1}$ for every $N \in \inte_{>0}$ (here $\mathbf{m}$ is the maximal ideal in $R = \comp[[t]]$).

\noindent {\bf Assumption I} (see Assumption \ref{asy_assumption_1} for the precise statement):
Each summand $\Phi^{(a)}$ has asymptotic support on the corresponding half-hyperplane $P_a$ (intuitively meaning that the support of $\Phi^{(a)}$ is more and more concentrated along $P_a$ as $\hp\rightarrow 0$) and has asymptotic expansion (as $\hp\rightarrow 0$) of the form
$$\Phi^{(a)} = \emc^{(a)} + \digamma^{(a)},$$
where $\emc^{(a)}$ is the leading order term consisting of terms with the leading $\hp$ order and $\digamma^{(a)}$ is the error term consisting of terms with higher $\hp$ orders.

From this assumption, we deduce that:
\begin{lemma}[=Lemma \ref{split_MC}]\label{theorem_support}
For each $a \in \mathbb{W}$, the summand $\Phi^{(a)}$ is a solution of the Maurer-Cartan equation \eqref{eqn:MC_eqn_A-side} over $U \setminus Q$.
\end{lemma}


Now we delete $Q$ from $U$ and work over $A := U \setminus Q$. We also choose an open set $\tilde{A}_0$ in the universal cover $\tilde{A}$ of $A$ and consider the covering map $\mathtt{p}: \tilde{A}_0 \rightarrow A$.

\noindent {\bf Assumption II} (see Assumption \ref{asy_assumption_2} for the precise statement):
Applying the homotopy operator $\hat{\mathcal{H}}$ (defined by integration over a homotopy $h : \real \times \tilde{A}_0 \rightarrow \tilde{A}_0$ contracting $\tilde{A}_0$ to a point in \eqref{homotopy_equation}) to the pullback of the leading order term $\emc^{(a)}$ by $\mathtt{p}$ gives a step function which jumps across the lift of $P_a$ in $\tilde{A}_0$ and whose restriction to the affine half space $\hat{\mathbb{H}}(P_a)\setminus P_a$ produces an element $Log(\Theta_a)$ in the tropical vertex Lie-algebra $\mathfrak{h}$ (defined in Definition \ref{trop_lie_algebra}). Figure \ref{fig:introduction_theorem1} illustrates the situation in a slice of a tubular neighborhood around $Q$.
\begin{figure}[h]
\centering
\includegraphics[scale=0.3]{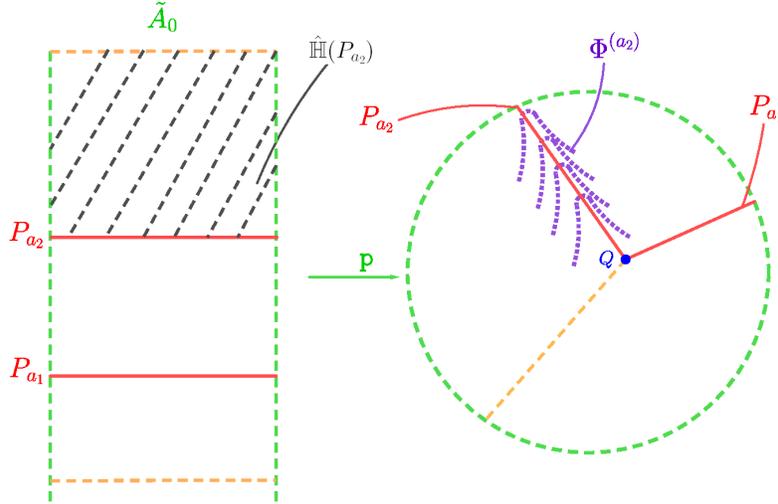}
\caption{A slice in a tubular neighborhood around $Q$}
\label{fig:introduction_theorem1}
\end{figure}

Since $\check{X}_0(U) \times_{A} \tilde{A}_0$ does not admit any non-trivial deformations, each summand $\Phi^{(a)}$ in \eqref{eqn:Phi_decomposition} is gauge equivalent to $0$, so there exists a unique solution $\varphi_a$ to
$e^{\varphi_a} * 0 = \Phi^{(a)}$
satisfying the gauge fixing condition $\hat{\mathcal{P}}\varphi_a = 0$. We carefully estimate the orders of the parameter $\hp$ in the asymptotic expansion of the gauge $\varphi_a$ as in the single wall case above, and obtain the following:
\begin{lemma}[=Lemma \ref{asymptoticexpansion}]\label{theorem_asymptotic_expansion}
The asymptotic expansion of the gauge $\varphi_a$ is of the form (see Notations \ref{O_loc} for the precise meaning of $O_{loc}(\hp^{1/2})$):
$$\varphi_a = \psi_a + O_{loc}(\hp^{1/2}),$$
where $\psi_a$, the semi-classical limit of $\varphi_a$ as $\hp \to 0$, is a step function which jumps across the half-hyperplane $P_a$ and is related to an element $\Theta_a$ of the tropical vertex group by the formula
$$\text{Log}(\Theta_a) =\psi_a|_{\hat{\mathbb{H}}(P_a) \setminus P_a};$$
here $\hat{\mathbb{H}}(P_a) \setminus P_a \subset \tilde{A}_0$ is the open half-space (defined in Notation \ref{order_N_neighborhood}) which contains the support of $\psi_a$.
\end{lemma}

Thus, each $\Phi^{(a)}$, or more precisely, the gauge $\varphi_a$, determines a wall
$\mathbf{w}_a = (P_a, \Theta_a)$
supported on a tropical half-hyperplane $P_a$ and equipped with a wall crossing factor $\Theta_a$.
Hence the Fourier decomposition \eqref{eqn:Phi_decomposition} of the Maurer-Cartan solution $\Phi$ defines a scattering diagram $\mathscr{D}(\Phi)$ consisting of the walls $\{\mathbf{w}_a\}_{a \in \mathbb{W}}$. Our first main result is the following:
\begin{theorem}[=Theorem \ref{scatteringtheorem1}]\label{theorem1}
If $\Phi$ is any solution to the Maurer-Cartan equation of $\widehat{\mathbf{g}^*/\mathcal{E}^*}(U)$ satisfying both Assumptions I and II (or more precisely Assumptions \ref{asy_assumption_1} and \ref{asy_assumption_2}), then the associated scattering diagram $\mathscr{D}(\Phi)$ is consistent, meaning that we have the following identity
$
\Theta_{\gamma, \mathscr{D}(\Phi)} = \text{Id},
$
where the left-hand side is the path ordered product (whose definition will be reviewed in Section \ref{analytic_continuation}) along any embedded loop $\gamma$ in $U \setminus \text{Sing}(\mathscr{D}(\Phi))$ intersecting $\mathscr{D}(\Phi)$ generically; here $\text{Sing}(\mathscr{D}(\Phi)) = Q$ is the singular set of the scattering diagram $\mathscr{D}(\Phi)$.
\end{theorem}

Our second main result studies how a scattering process starting with two non-parallel walls $\mathbf{w}_1 = (P_1, \Theta_1), \mathbf{w}_2 = (P_2, \Theta_2)$ intersecting transversally at $Q = P_1 \cap P_2$ gives rise to a MC solution of $\widehat{\mathbf{g}^*/\mathcal{E}^*}(U) $ satisfying both Assumptions I and II, thereby producing a consistent scattering diagram via Theorem \ref{theorem1}.\footnote{Indeed Assumptions I and II (or more precisely Assumptions \ref{asy_assumption_1} and \ref{asy_assumption_2}) are extracted from properties of the MC solutions we constructed.}

In this case, there are two solutions to the MC equation \eqref{eqn:MC_eqn_A-side} $\incoming_{\mathbf{w}_i} \in KS_{\check{X}_0}^1(U)[[t]]$, $i=1,2$ associated to the two initial walls $\mathbf{w}_1, \mathbf{w}_2$ respectively (e.g. those provided by our ansatz), but their sum
$\incoming := \incoming_{\mathbf{w}_1} + \incoming_{\mathbf{w}_2} \in KS_{\check{X}_0}(U)[[t]]$
does {\em not} solve \eqref{eqn:MC_eqn_A-side}, even up to error terms with exponential order in $\hp^{-1}$.
Nevertheless, a method of Kuranishi \cite{Kuranishi65} allows us to, after fixing the gauge using an explicit homotopy operator (introduced in Definition \ref{pathspacehomotopy}), write down a solution $\Phi = \incoming + \cdots$, as a sum over trees \eqref{eqn:MC_sol_Phi} with input $\incoming$, of the equation \eqref{eqn:MC_eqn_A-side} up to error terms with exponential order in $\hp^{-1}$, or more precisely, of the MC equation of the dgLa $\widehat{\mathbf{g}^*/\mathcal{E}^*}(U)$.

The MC solution $\Phi$ has a Fourier decomposition as in \eqref{eqn:Phi_decomposition} of the form
$
\Phi = \incoming + \sum_{a \in \mathbb{W}} \Phi^{(a)},
$
where the sum is over $a = (a_1, a_2) \in \mathbb{W} := \left(\inte_{>0}^2\right)_{\text{prim}}$ which parametrizes the tropical half-hyperplanes $P_a$'s containing $Q$ and lying in-between $P_1$ and $P_2$, as shown in Figure \ref{fig:high_dimensional_pic}.
\begin{figure}[h]
\centering
\includegraphics[scale=0.25]{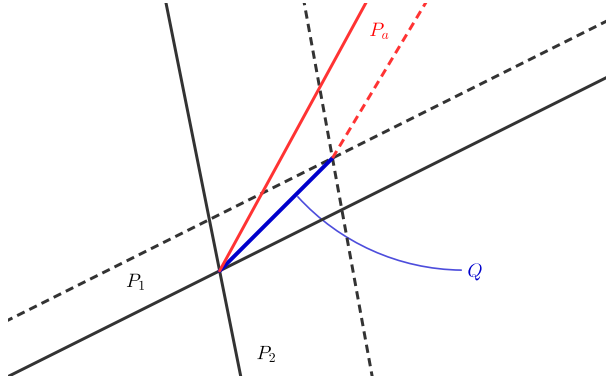}
\caption{Scattered walls $P_a$'s from two initial walls}
\label{fig:high_dimensional_pic}
\end{figure}
Our second main result is the following:
\begin{theorem}[=Theorem \ref{scatteringtheorem2}]\label{theorem2}
The Maurer-Cartan solution $\Phi$ satisfies both Assumptions I and II (or more precisely Assumptions \ref{asy_assumption_1} and \ref{asy_assumption_2}) in Theorem \ref{theorem1}, and hence the scattering diagram $\mathscr{D}(\Phi)$ associated to $\Phi$ is consistent, meaning that we have the following identity\footnote{Another common way to write this identity is as a formula for the commutator of two elements in the tropical vertex group: $\Theta_2^{-1}\Theta_1 \Theta_2 \Theta_1^{-1} = \prod^{\gamma}_{a \in \mathbb{W}} \Theta_a$.}
\begin{equation*}
\Theta_{\gamma, \mathscr{D}(\Phi)} = \Theta_1^{-1}\Theta_2\left(\prod^{\gamma}_{a \in \mathbb{W}} \Theta_a \right) \Theta_1 \Theta_2^{-1} = \text{Id},
\end{equation*}
along any embedded loop $\gamma$ in $U \setminus \text{Sing}(\mathscr{D}(\Phi))$ which intersects $\mathscr{D}(\Phi)$ generically; here $\text{Sing}(\mathscr{D}(\Phi)) = Q = P_1 \cap P_2$.
\end{theorem}


The proofs that $\Phi$ satisfies both Assumptions \ref{asy_assumption_1} and \ref{asy_assumption_2} occupy Sections \ref{sec:leading_order_MC} and \ref{sec:tropical_leading_order}; Assumption \ref{asy_assumption_1} will be handled in Theorem \ref{asy_support_theorem} in Section \ref{sec:leading_order_MC} while Assumption \ref{asy_assumption_2} will be handled in Lemma \ref{iteratedintegral} in Section \ref{sec:tropical_leading_order}.

\begin{remark}
Notice that the scattering diagram $\mathscr{D}(\Phi)$ is the unique (by passing to a minimal scattering diagram if necessary) consistent extension, determined by Kontsevich-Soibelman's Theorem \ref{KSscatteringtheorem}, of the scattering diagram consisting of two initial walls $\mathbf{w}_1$ and $\mathbf{w}_2$.
\end{remark}

\subsection{A reader's guide}

The rest of this paper is organized as follows.

In Section \ref{dgladeformation}, we review the Kodaira-Spencer dgLa $KS_{\check{X}_0}$ associated to the semi-flat Calabi-Yau manifold $\check{X}_0$, followed by a brief review of the Legendre and Fourier transforms.

In Section \ref{recallscattering}, we review the tropical vertex group and the theory of scattering diagrams (in particular a theorem due to Kontsevich-Soibelman) following the exposition in \cite{gross2010tropical}.

Section \ref{onewall} is about the single wall scenario.
In Section \ref{sec:ansatz_one_wall}, we write down an ansatz associated to a given single wall solving the MC equation.
In Section \ref{sec:asy_support}, we formulate the key notion of {\em asymptotic support on a tropical polyhedral subset} which allows us to define a filtration \eqref{filtration} to keep track of the $\hp$ orders. We also prove two key results, namely, Lemma \ref{support_product} (and its extension Lemma \ref{filtrationlemma}) and Lemma \ref{integral_lemma}, which form the basis for the subsequent asymptotic analysis. Applying them, we prove the main results Lemma \ref{leadingorderlemma} and Proposition \ref{prop:MC_sol_one_wall} for the single wall case.
Except Definition \ref{asypmtotic_support_def} and the statements of Lemmas \ref{support_product} and \ref{integral_lemma}, the reader may skip the rather technical Section \ref{sec:relating_wall_crossing_fac} at first reading.

Section \ref{twowalls} is the heart of this paper where we study the scattering process which starts with two initial walls.
In Section \ref{algebraicsolveMC}, Kuranishi's method of solving the MC equation of a dgLa is reviewed.
In Section \ref{sec:solve_MC_two_walls}, we introduce the dgLa $\widehat{\mathbf{g}^*/\mathcal{E}^*}(U)$ by which we make precise the meaning of {\em solving the MC equation of $KS_{\check{X}_0}(U)[[t]]$ up to error terms with exponential order in $\hp^{-1}$}. We then begin the asymptotic analysis of the MC solutions of $\widehat{\mathbf{g}^*/\mathcal{E}^*}(U)$; the key results here are Theorem \ref{asy_support_theorem} and Lemma \ref{lem:semi_classical_integral}.
In Section \ref{sec:key_lemmas}, we apply the results obtained in Section \ref{sec:solve_MC_two_walls} to prove Lemmas \ref{lem:loc_constant_coeff} and \ref{asymptoticexpansion} (which are parallel to Lemma \ref{leadingorderlemma} and Proposition \ref{prop:MC_sol_one_wall} in Section \ref{onewall}), from which we deduce Theorems \ref{theorem1} and \ref{theorem2}.

\section*{Acknowledgement}
We thank Si Li, Marco Manetti and Matt Young for various useful conversations when we were preparing the first draft of this paper. We are also heavily indebted to the anonymous referees for many critical yet very constructive comments and suggestions, which lead to a significant and substantial improvement in the exposition including the introduction of the notion of ``asymptotic support'' and an illuminating reorganization of many of the arguments in the proofs of our main results. Finally we would like to thank Mark Gross for his interest in our work.

The work of K. Chan was supported by a grant from the Hong Kong Research Grants Council (Project No. CUHK14302015) and direct grants from CUHK. The work of N. C. Leung was supported by grants from the Hong Kong Research Grants Council (Project No. CUHK402012 $\&$ CUHK14302215) and direct grants from CUHK. The work of Z. N. Ma was supported by a Professor Shing-Tung Yau Post-Doctoral Fellowship, Center of Mathematical Sciences and Applications at Harvard University, Department of Mathematics at National Taiwan University, Yau Mathematical Sciences Center at Tsinghua University, and Institute of Mathematical Sciences and Department of Mathematics at CUHK. 
\section{The Kodaira-Spencer dgLa in the semi-flat case}\label{dgladeformation}

In this section, we review the classical Kodaira-Spencer deformation theory of complex structures and the associated dgLa \cite{Morrow-Kodaira_book, manetti2005differential} in the semi-flat setting, as well as the Legendre and Fourier transforms \cite{Hitchin97, Leung05} which play important roles in semi-flat SYZ mirror symmetry.


\subsection{The semi-flat Calabi-Yau manifold $\check{X}_0$}\label{sec:semi_flat_complex_manifold}

We let $\text{Aff}(\real^n) = \real^n  \rtimes \text{GL}_n(\real)$ be the group of affine linear transformations of $\real^n$
and consider the subgroup $\text{Aff}_\real(\inte^n)_0 := \real^n \rtimes \text{SL}_n(\inte)$.
\begin{definition}[\cite{gross2010tropical}]
An $n$-dimensional smooth manifold $B$ is called {\em tropical affine} if it admits an atlas $\{(U_i, \psi_i)\}$ of coordinate charts $\psi_i : U_i \rightarrow \real^n$ such that $\psi_i \circ \psi_j^{-1} \in \text{Aff}_\real(\inte^n)_0$ for all $i, j$.
\end{definition}
Given a (possibly non-compact) tropical affine manifold $\check{B}_0$, we set
$$\check{X}_0 := T\check{B}_0 / \Lambda_{\check{B}_0},$$
where the lattice subbundle $\Lambda_{\check{B}_0} \subset T\check{B}_0$ is locally generated by the coordinate vector fields $\dd{x^1}, \ldots, \dd{x^n}$ for a given choice of local affine coordinates $\check{x} = (x^1,\dots ,x^n)$ in a contractible open subset $\check{U} \subset \check{B}_0$.
Then the natural projection map
$\check{p}: \check{X}_0 \rightarrow \check{B}_0$
is a torus fibration. We also let $y^j$'s be the canonical coordinates on the fibres of $\check{p}$ over $\check{U}$ with respect to the frame $\dd{x^1}, \ldots, \dd{x^n}$ of $T\check{B}_0$.

Choosing $\beta = \sum_{i, j = 1}^n \beta^j_i(\check{x}) dx^i \otimes \dd{x^j} \in \Omega^1(T\check{B}_0)$ satisfying
$\nabla \beta = 0 \in \Omega^2(T\check{B}_0)$, where $\nabla$ is the natural affine flat connection on $\check{B}_0$, we get a one-parameter family of complex structures parametrized by $\hp \in \real_{>0}$ defined by the family of matrices
\begin{equation}\label{eqn:complex_structure_J_beta}
\check{J}_\beta
= \left(
\begin{array}{cc}
- \hp\beta & \hp I   \\
- \hp^{-1}(I+\hp^2\beta^2)  & \hp\beta
\end{array} \right)
\end{equation}
with respect to the local frame $\dd{x^1}, \ldots, \dd{x^n}, \dd{y^1}, \ldots, \dd{y^n}$, where we write $\beta$ as a matrix with respect to the frame $\dd{x^1}, \ldots, \dd{x^n}$. Locally, the corresponding holomorphic volume form is given by
\begin{equation}
\check{\Omega}_{\beta} = \bigwedge_{j = 1}^n \left( ( dy^j -\sum_{k = 1}^n \beta^j_k dx^k) + i \hp^{-1}dx^j \right),
\end{equation}
and a holomorphic frame of $T^{1,0}\check{X}_0$ can be written as
\begin{equation}\label{local_holo_vector}
\check{\partial}_{j} := \dd{\log \bmc^j} = \frac{i}{4\pi} \left( \dd{y^j}- i \hp \left( \sum_{k = 1}^n \beta_j^k \dd{y^k} + \dd{x^j} \right) \right),
\end{equation}
for $j = 1,\dots,n$. So the local complex coordinates are given by
\begin{equation}\label{local_holo_coord}
\bmc^j = \exp\left( -2\pi i \left( y^j - \sum_k \beta^j_k x^k + i \hp^{-1} x^j \right) \right).
\end{equation}
The condition that $\sum_{k = 1}^n \beta^j_k(\check{x}) dx^k$ being closed for each $j = 1, \ldots, n$ is equivalent to integrability of the almost complex structure $\check{J}_\beta$.

\subsection{The Kodaira-Spencer dgLa}\label{sec:def_cpx_str}

For a complex manifold $\check{X}_0$, the Kodaira-Spencer complex is the space $KS_{\check{X}_0} := \Omega^{0,*}({\check{X}_0},T^{1,0}{\check{X}_0})$ of $T^{1,0}{\check{X}_0}$-valued $(0,*)$-forms, which is equipped with the Dolbeault differential $\pdb$ and a Lie bracket defined in local holomorphic coordinates $z_1, \ldots, z_n \in \check{X}_0$ by $[ \phi d\bar{z}^I , \psi d\bar{z}^J] = [\phi, \psi] d\bar{z}^I \wedge d\bar{z}^J$,
where $\phi, \psi \in \Gamma(T^{1,0}{\check{X}_0})$. The triple
$$(KS_{\check{X}_0}, \pdb, [\cdot, \cdot])$$
defines the \textit{Kodaira-Spencer differential graded Lie algebra (abbrev. dgLa)}, which governs the deformation theory of complex structures on $\check{X}_0$. Given an open subset $\check{U} \subset \check{X}_0$, we may also talk about the local Kodaira-Spencer complex $KS_{\check{X}_0}(\check{U})$.

\begin{notation}
We let $R = \comp[[t]]$ to be the ring of formal power series and $\mathbf{m} = (t)$ denote the maximal ideal generated by $t$, and consider dgLa's over $R$ to avoid convergence issues.
\end{notation}

An element $\varphi \in \Omega^{0,1}(\check{X}_0, T^{1,0}{\check{X}_0})\otimes \mathbf{m}$ defines a formal deformation of complex structures if and only if it is a solution to the Maurer-Cartan equation \eqref{eqn:MC_eqn_A-side}. The exponential group $KS_{\check{X}_0}^0\otimes \mathbf{m}$ acts on the set of Maurer-Cartan solutions $MC_{KS_{\check{X}_0}}(R)$ as automorphisms of the formal family of complex structures over $R$, and therefore one can define the space of deformations of $\check{X}_0$ over $R$ by $Def_{KS_{\check{X}_0}}(R) := MC_{KS_{\check{X}_0}}(R) / \exp(KS_{\check{X}_0}^0\otimes \mathbf{m})$ via the dgLa $KS_{\check{X}_0}$.

\subsection{The Legendre transform}\label{semiflat_kahler}

To define the Legendre dual $B_0$ of $\check{B}_0$ so that we can work in the tropical world, we need a metric $g$ on $\check{B}_0$ of Hessian type (see, e.g. \cite[Chapter 6]{dbrane}):
\begin{definition}\label{Hessian_type_g}
A Riemannian metric $g = (g_{ij})_{i,j}$ on $\check{B}_0$ is said to be {\em Hessian type} if it is locally given by $g = \sum_{i,j}\pdpd{\check{\phi}}{x^i}{x^j} dx^i \otimes dx^j$ in local affine coordinates $x^1, \ldots, x^n$ for some convex function $\check{\phi}$.
\end{definition}
To construct K\"ahler structures, we further need a compatibility condition between $g$ and $\beta$ in \eqref{eqn:complex_structure_J_beta}, namely, we assume that
\begin{equation}\label{Bfield_assumption}
\sum_{i,j,k} \beta^j_i g_{jk} dx^i\wedge dx^k = 0,
\end{equation}
when we write $\beta = \sum_{i,j}\beta^j_i(\check{x})dy_j \wedge dx^i $ in the local coordinates $x^1, \ldots, x^n, y_1, \ldots, y_n$.
Given such a Hessian type metric $g$, a K\"ahler form on $\check{X}_0$ is given by $\check{\omega}  = 2i \pd\pdb \check{\phi}= \sum_{j, k} g_{jk}  dy^j \wedge dx^k$.

We can now introduce the Legendre transform following Hitchin \cite{Hitchin97}; see also \cite[Chapter 6]{dbrane}. Given a strictly convex smooth function $\check{\phi} : \check{U} (\subset \check{B}_0) \rightarrow \real$, we trivialize $T^*\check{U} \cong \check{U} \times \real^n$ via affine frames and define the \textit{Legendre transform} $L_{\check{\phi}}: \check{U} \to \real^n$ by
$x = L_{\check{\phi}}(\check{x}) := d \check{\phi} (\check{x})  \in \real^n$,
or equivalently, by $x_j = \ddd{\check{\phi}}{x^j}$, where $x = (x_1,\dots,x_n) \in \real^n$ denote the dual coordinates.
The image $U := L_{\check{\phi}}(\check{U}) \subset \real^n$ is an open subset and $L_{\check{\phi}}$ is a diffeomorphism.
The Legendre dual $\phi : U \rightarrow \real$ of $\check{\phi}$ is defined by the equation $\phi(x) := \sum_{j = 1}^n x_j x^j - \check{\phi}(\check{x})$,
and the dual transform $L_\phi$ is inverse to $L_{\check{\phi}}$.

If $\check{\phi}$ is the semi-flat potential in Definition \ref{Hessian_type_g} which defines a Hessian type metric, then the dual coordinate charts
$U = L_{\check{\phi}}(\check{U})$
actually glue to give another tropical affine manifold $B_0$, which we call the {\em Legendre dual} of $\check{B}_0$, whose underlying smooth manifold is same as that of $\check{B}_0$ (see \cite[Chapter 6]{dbrane}.
The lattice bundles $\Lambda_{B_0} \cong \Lambda^\vee_{\check{B}_0}$ are interchanged in this process, so that we can write $\check{X}_0 = T^*B_0 / \Lambda^\vee_{B_0}$, and using the affine coordinates $(x_1,\dots,x_n)$ on $B_0$, we can write 
$$\check{\Omega}_\beta = \bigwedge_{k = 1}^n \left(dy^k - \sum_{j = 1}^n \left( \beta^{jk} - i\hp^{-1} g^{jk} \right) dx_j \right), \quad \check{\omega} = \sum_{k = 1}^n dy^k\wedge dx_k$$

\subsection{The Fourier transform}\label{sec:fourier_transform}
\begin{definition}\label{def:sheaf_of_affine_function}
The {\em sheaf of integral affine functions} $\text{Aff}_{\check{B}_0}^\inte$, as a sheaf over $B_0$ (which is the same as $\check{B}_0$ as a smooth manifold), is the subsheaf of the sheaf of smooth functions over $B_0$ whose local sections $\text{Aff}_{\check{B}_0}^\inte(U)$ over a contractible open set $U \subset B_0$ are defined to be affine linear functions of the form $m(\check{x}) = m_1 x^1 + \cdots + m_n x^n + b$ for some $m_i \in \inte$ and $b \in \real$, in local affine coordinates on $\check{B}_0$ ({\em caution: not $B_0$}).
This sheaf fits into the following exact sequence of sheaves over $B_0$
$$0 \rightarrow \underline{\real} \rightarrow \text{Aff}^\inte_{\check{B}_0} \rightarrow \Lambda_{B_0}\rightarrow 0.$$
\end{definition}

Since $\check{X}_0 = T\check{B}_0 / \Lambda_{\check{B}_0}$, exponentiation of complexification of local affine linear functions on $\check{B}_0$ give local holomorphic functions on $\check{X}_0$ as follows.
\begin{definition}
Given $m \in \text{Aff}_{\check{B}_0}^\inte(U)$, expressed locally as $m(x) = \sum_j m_j x^j + b$, we let
$\bmc^m := e^{\frac{2\pi b}{\hp}} (\bmc^1)^{m_1} \cdots (\bmc^n)^{m_n} \in \mathcal{O}_{\check{X}_0}(\check{p}^{-1}(U)),$
where $\bmc^j$ is given in equation \eqref{local_holo_coord}. This defines an embedding
$\text{Aff}_{\check{B}_0}^\inte(U)\hookrightarrow \mathcal{O}_{\check{X}_0}(\check{p}^{-1}(U))$, and we denote the image subsheaf by $\mathcal{O}^{\text{aff}}$, as a sheaf over $B_0$.
\end{definition}

We can embed the lattice bundle $\Lambda_{B_0}^\vee \hookrightarrow \check{p}_* T^{1,0}\check{X}_0 $ into the push forward of the sheaf of holomorphic vector fields; in local coordinates $U$, it is given by (cf. equation \eqref{local_holo_vector})
\begin{equation}\label{holo_vector_field_2}
n = (n^j) \mapsto \check{\partial}_n  :=  \sum_j n^j \dd{\log \bmc^j} = \frac{i}{4\pi} \sum_j n^j \left( \dd{y^j}- i\hp \sum_k \left( \beta_j^k \dd{y^k} +  g_{jk} \dd{x_k} \right) \right)
\end{equation}
for a local section $n\in \Lambda_{B_0}^\vee(U)$. This embedding is globally defined, and by abuse of notations, we will write $\Lambda_{B_0}^\vee$ to stand for its image subsheaf. For later purpose, we introduce the notation
\begin{equation}\label{affine_vector_field}
\partial_n := \frac{\hp}{4\pi} \sum_j n^j g_{jk} \dd{x_k}.
\end{equation}

\begin{notation}\label{not:local_lattice}
Since we work in a contractible open coordinate chart $U$, we will fix a rank $n$ lattice $M \cong \inte^n$ and its dual $N = \text{Hom}(M, \inte)$, and identify $U \subset M_\real := M \otimes_\inte \real \cong \real^n$ as an open subset containing the origin $0$ and write $N_\real := N \otimes_\inte \real$. We also trivialize $\Lambda_{B_0}|_{U} \cong \underline{M}$ and $\Lambda_{B_0}^\vee|_{U} \cong \underline{N}$ and identify $ \check{X}_0(U) = \check{p}^{-1}(U) \cong U \times(N_\real/N)$. Since $TU \cong \Lambda_{B_0} \otimes_\inte \real$, a local section $m \in \underline{M}(U)$ naturally corresponds to an affine integral vector field over $U$, which will be denoted by $m$ as well. The exact sequence in Definition \ref{def:sheaf_of_affine_function} splits and we will call $m$'s or the associated $\bmc^{m}$'s the {\em Fourier modes}.
\end{notation}

\begin{definition}
We consider the sheaf $\mathcal{O}^{\text{aff}} \otimes_\inte \Lambda_{B_0}^\vee$ over $B_0$ and define a Lie bracket $[\cdot,\cdot]$ on it by restriction of the usual Lie bracket on $\check{p}_*\mathcal{O}(T^{1,0}\check{X}_0)$.
\end{definition}

Notice that the Lie bracket on $\mathcal{O}^{\text{aff}} \otimes_\inte \Lambda_{B_0}^\vee$ is well defined because in a small enough affine coordinate chart, we have the following formula from \cite{gross2010tropical}
\begin{equation}\label{vertex_lie_algebra}
\left[ \bmc^m \otimes \check{\partial}_n , \bmc^{m'} \otimes \check{\partial}_{n'} \right]= \bmc^{m+m'} \check{\partial}_{( m',n ) n' - ( m ,n' ) n},
\end{equation}
which shows that $\mathcal{O}^{\text{aff}} \otimes_\inte \Lambda_{B_0}^\vee$ is closed under the Lie bracket on $\check{p}_*\mathcal{O}(T^{1,0}\check{X}_0)$.

\begin{notation}\label{lattice_pairing}
The pairing $(m, n)$ in \eqref{vertex_lie_algebra} is the natural pairing between $m \in \Lambda_{B_0}(U)$ and $n \in \Lambda_{B_0}^\vee(U)$. Given a local section $m \in \Lambda_{B_0}(U)$, we let $m^\perp \subset \Lambda_{B_0}^\vee(U)$ be the sub-lattice perpendicular to $m$ with respect to $(\cdot,\cdot)$.
\end{notation}

\begin{definition}\label{def:fourier_transform}
On a contractible open subset $U \subset B_0 \cong \check{B}_0$, the {\em Fourier transform}
$$\mathcal{F}: \mathbf{G}^*(U):=\bigoplus_{m \in \Lambda_{B_0}(U)} \Omega^*(U) \cdot \bmc^{m} \otimes_\inte N \hookrightarrow KS_{\check{X}_0}(U)$$
is defined by sending $\alpha \in \Omega^*(U)$ to $\big( \check{p}^{*}(\alpha)\big)^{0,1} $ (where $\big( \cdot \big)^{0,1}$ denotes the $(0,1)$-part of the $1$-form) and $n \in N$ to $\check{\partial}_n$.
$\mathcal{F}$ is injective and hence induces a dgLa structure on $\mathbf{G}^*(U)$ from that on $KS_{\check{X}_0}(U)$.\footnote{Direct computation shows that we have the formula $\pdb(\sum_m \bmc^m \alpha_m^n \check{\partial}_n) = \sum_m \bmc^m (d\alpha) \check{\partial}_n$.} We also let $\mathbf{G}^*_N(U) := \mathbf{G}^*(U) \otimes R/\mathbf{m}^{N+1}$ and $\widehat{\mathbf{G}}^*(U) := \varprojlim_N \mathbf{G}^*_N(U)$.
\end{definition}

\begin{remark}
For more details on how Fourier (or SYZ) transforms can be applied to understand (semi-flat) SYZ mirror symmetry, we refer the readers to Fukaya's original paper \cite{fukaya05} and a recent survey article \cite{ma_survey_cma} by the third named author.
\end{remark} 
\section{Scattering diagrams}\label{recallscattering}

In the section, we review the notion of scattering diagrams introduced in \cite{kontsevich-soibelman04, gross2011real}. We will adopt the setting and notations from \cite{gross2010tropical} with slight modifications to fit into our context.

\subsection{The sheaf of tropical vertex groups}
We start with the same set of data $(B_0, g, \beta)$ as in Section \ref{semiflat_kahler}, and use $x = (x_1,\dots,x_n)$ as local affine coordinates on $B_0$ and $\check{x} = (x^1,\dots,x^n)$ as local affine coordinates on $\check{B}_0$ as before. Given the formal power series ring $R = \comp[[t]]$ and its maximal ideal $\mathbf{m}$, we consider the sheaf of Lie algebras $\mathfrak{g} := \left(\mathcal{O}^{\text{aff}} \otimes_\inte \Lambda_{B_0}^\vee\right) \widehat{\otimes}_\comp R$ over $B_0$.

\begin{definition}\label{trop_lie_algebra}
The subsheaf $\mathfrak{h} \hookrightarrow \mathfrak{g}$ of Lie algebras is defined as the image of the embedding $\left(\bigoplus_{m \in \Lambda_{B_0}(U)} \comp \cdot \bmc^m \otimes_\inte (m^\perp)\right) \widehat{\otimes}_\comp R \rightarrow \mathfrak{g}(U)$ over each affine coordinate chart $U \subset B_0$.\footnote{It is a subsheaf of Lie subalgebras of $\mathfrak{g}$ as can be seen from the formula \eqref{vertex_lie_algebra}.} The {\em sheaf of tropical vertex groups} over $B_0$ is defined as the sheaf of exponential groups $\exp(\mathfrak{h}\otimes_R\mathbf{m})$ which act as automorphisms on $\mathfrak{h}$ and $\mathfrak{g}$.
\end{definition}

\subsection{Kontsevich-Soibelman's wall crossing formula}\label{2dscattering}
This formulation of the wall crossing formula originated from \cite{kontsevich-soibelman04} but we will mostly follow \cite{gross2010tropical} as we want to work on $B_0$ instead of $\check{B}_0$. From now on, we will work locally in a contractible coordinate chart $U \subset B_0$. We use the same notations as in Section \ref{dgladeformation}.

\begin{definition}
Given $m \in M \setminus\{0\}$ and $n \in m^\perp$, we let $\mathfrak{h}_{m,n} := (\comp[\bmc^m]\cdot \bmc^m)\check{\partial}_n \widehat{\otimes}_\comp \mathbf{m}  \hookrightarrow \mathfrak{g}$ whose general elements are of the form
$\sum_{j,k \geq 1} a_{jk} \bmc^{km}\check{\partial}_{n} t^{j}$, where $a_{jk} \neq 0$ for only finitely many $k$'s for each fixed $j$.
This defines an abelian Lie subalgebra of $\mathfrak{g}$ by the formula \eqref{vertex_lie_algebra}.
\end{definition}

\begin{definition}\label{wall}
A {\em wall} $\mathbf{w}$ in $U$ is a triple $(m, P, \Theta)$, where
\begin{itemize}
\item
$m \in M \setminus \{0\}$ parallel to $P$,
\item
$P$ is a connected oriented codimension one convex tropical polyhedral subset of $U$ (by a convex tropical polyhedral subset we mean a convex subset which is locally defined by affine linear equations and inequalities defined over $\mathbb{Q}$),
\item
$\Theta \in \exp(\mathfrak{h}_{m,n})|_P$ is a germ of sections near $P$, where
$n \in \Lambda_{B_0}^\vee(U) \cong N$ is the unique primitive element satisfying $n \in (TP)^\perp$ and $(\nu_P, n) < 0$, and $\nu_P \in TU \cong U \times M_{\real}$ here is a vector normal to $P$ such that the orientation of $TP \oplus \real \cdot \nu_P$ agrees with that of $U$.
\end{itemize}
\end{definition}

\begin{definition}\label{scattering_diagram_def}
A {\em scattering diagram} $\mathscr{D}$ is a set of walls $\left\{ ( m_\alpha, P_\alpha, \Theta_\alpha) \right\}_{\alpha }$ such that there are only finitely many $\alpha$'s with $\Theta_\alpha \neq id $ $(\text{mod $\mathbf{m}^N$})$ for every $N \in \inte_{>0}$. We define the {\em support} of $\mathscr{D}$ to be $
\text{supp}(\mathscr{D}) := \bigcup_{\mathbf{w} \in \mathscr{D}} P_{\mathbf{w}}$, and the {\em singular set} of $\mathscr{D}$ to be $
\text{Sing}(\mathscr{D}) := \bigcup_{\mathbf{w} \in \mathscr{D}} \partial P_{\mathbf{w}} \cup \bigcup_{\mathbf{w}_1\pitchfork \mathbf{w}_2} P_{\mathbf{w}_1} \cap P_{\mathbf{w}_2}$,
where $\mathbf{w}_1 \pitchfork \mathbf{w}_2$ means transversally intersecting walls.\footnote{There is a natural (possibly up to further subdivisions) polyhedral decomposition of $\text{Sing}(\mathscr{D})$ whose codimension $2$ cells are called {\em joints} in the Gross-Siebert program \cite{gross2011real}.}
\end{definition}

\subsubsection{Path ordered products}\label{analytic_continuation}

An embedded path $\gamma: [0,1] \rightarrow B_0 \setminus \text{Sing}(\mathscr{D})$
is said to be {\em intersecting $\mathscr{D}$ generically}
if $\gamma(0), \gamma(1) \notin \text{supp}(\mathscr{D})$, $\text{Im}(\gamma) \cap \text{Sing}(\mathscr{D}) = \emptyset$ and it intersects all the walls in $\mathscr{D}$ transversally.
Given such an embedded path $\gamma$, we define the {\em path ordered product} along $\gamma$ as an element of the form
$\Theta_{\gamma, \mathscr{D}} = \prod^{\gamma}_{\mathbf{w} \in \mathscr{D}} \Theta_{\mathbf{w}} \in \exp(\mathfrak{h}\otimes_R \mathbf{m})_{\gamma(1)}$ in the stalk of $\exp(\mathfrak{h}\otimes_R\mathbf{m})$ at $\gamma(1)$, following \cite{gross2010tropical}.
More precisely, for each $k \in \inte_{>0}$, we define $\Theta_{\gamma, \mathscr{D}}^k \in \exp(\mathfrak{h}\otimes_R (\mathbf{m}/\mathbf{m}^{k+1}))_{\gamma(1)}$ and let $
\Theta_{\gamma, \mathscr{D}} := \lim_{k \rightarrow +\infty} \Theta^k_{\gamma, \mathscr{D}}$, where $\Theta^k_{\gamma, \mathscr{D}}$ is defined as follows.

Given $k$, there is a finite subset $\mathscr{D}^k \subset \mathscr{D}$ consisting of walls $\mathbf{w}$ with $\Theta_{\mathbf{w}} \neq \text{Id}$ $(\text{mod $\mathbf{m}^{k+1}$})$ from Definition \ref{scattering_diagram_def}. We then have a sequence of real numbers $0=t_0 < t_1< t_2<\cdots<t_s<t_{s+1} = 1$ such that $\left\{ \gamma(t_1), \ldots, \gamma(t_s) \right\} = \gamma \cap \text{supp}(\mathscr{D}^k)$. For each $1 \leq i \leq s$, there are walls $\mathbf{w}_{i,1}, \ldots, \mathbf{w}_{i,l_i}$ in $\mathscr{D}^k$ such that $\gamma(t_i) \in P_{i,j} := \text{supp}(\mathbf{w}_{i,j})$ for all $j = 1, \dots, l_i$. Since $\gamma$ does not hit $\text{Sing}(\mathscr{D})$, we have $\text{codim}(\text{supp}(\mathbf{w}_{i,j_1} )\cap \text{supp}(\mathbf{w}_{i,j_2} )) = 1$ for any $j_1, j_2$, i.e. the walls $\mathbf{w}_{i,1}, \ldots, \mathbf{w}_{i,l_i}$ are overlapping with each other and contained in a common tropical hyperplane. Then we have an element
$\Theta_{\gamma(t_i)} := \prod_{j=1}^k \Theta^{\sigma_j}_{\mathbf{w}_{i,j}}$,
where $\sigma_j =1$ if orientation of $P_{i,j} \oplus \real \cdot \gamma'(t_i)$ agree with that of $B_0$ and $\sigma_j = -1$ otherwise. (Note that this element is well defined without prescribing the order of the product since the elements $\Theta_{\mathbf{w}_{i,j}}$'s are commuting with each other.)
We treat $\Theta_{\gamma(t_i)}$ as an element in $\exp(\mathfrak{h}\otimes_R\mathbf{m})_{\gamma(1)}$ by parallel transport and take the ordered product along the path $\gamma$ as $\Theta^{k}_{\gamma, \mathscr{D}} := \Theta_{\gamma(t_s)}\cdots \Theta_{\gamma(t_i)} \cdots \Theta_{\gamma(t_1)}$.

\begin{definition}\label{consistent_def}
A scattering diagram $\mathscr{D}$ is said to be {\em consistent} if we have $\Theta_{\gamma,\mathscr{D}} = \text{Id}$, for any embedded loop $\gamma$ intersecting $\mathscr{D}$ generically. Two scattering diagrams $\mathscr{D}$ and $\tilde{\mathscr{D}}$ are said to be {\em equivalent} if $\Theta_{\gamma,\mathscr{D}} = \Theta_{\gamma,\tilde{\mathscr{D}}}$ for any embedded path $\gamma$ intersecting both $\mathscr{D}$ and $\tilde{\mathscr{D}}$ generically.
\end{definition}

\begin{remark}
Given a scattering diagram $\mathscr{D}$, there is a unique representative $\mathscr{D}_{min}$ from its equivalence class which is {\em minimal}. First, we may remove those walls with trivial automorphisms $\Theta$ as they do not contribute to the path ordered product. Second, if two walls $\mathbf{w}_1$ and $\mathbf{w}_2$ share the same $P$ and $m$, we can simply take the multiplication $\Theta = \Theta_1 \circ \Theta_2$ and define a single wall $\mathbf{w}$. After doing so, we obtain a minimal scattering diagram equivalent to $\mathscr{D}$.
\end{remark}

\subsubsection{The wall crossing formula}\label{sec:wall_crossing_formula}

Next, we consider the case where $\mathscr{D}$ is a scattering diagram consisting of only two walls $\mathbf{w}_1 = (m_1, P_1, \Theta_1)$ and $\mathbf{w}_2 = (m_2, P_2, \Theta_2)$ where the supports $P_i$'s are tropical hyperplanes of the form $P_i = Q - \real \cdot m_i$ intersecting transversally in a codimension two tropical subspace $Q := P_1 \cap P_2 \subset U$.
In this case, we have the following theorem due to Kontsevich-Soibelman \cite{kontsevich-soibelman04}:
\begin{theorem}[Kontsevich and Soibelman \cite{kontsevich-soibelman04}]\label{KSscatteringtheorem}
Given a scattering diagram $\mathscr{D}$ consisting of two walls $\mathbf{w}_1 = (m_1, P_1, \Theta_1)$ and $\mathbf{w}_2 = (m_2, P_2, \Theta_2)$ supported on tropical hyperplanes $P_1, P_2$ intersecting transversally in a codimension two tropical subspace $Q := P_1 \cap P_2$, there exists a unique minimal consistent scattering diagram $\mathcal{S}(\mathscr{D}) \supset \mathscr{D}_{min}$, obtained by adding walls to $\mathscr{D}_{min}$ supported on tropical half-hyperplanes of the form $Q - \real_{\geq 0 } \cdot (a_1 m_1 + a_2 m_2)$ for $a = (a_1, a_2) \in \left(\inte_{>0}^2\right)_{\text{prim}}$.
\end{theorem}

\begin{remark}
Interesting relations between these wall crossing factors and relative Gromov-Witten invariants of weighted projective planes were established in \cite{gross2010tropical}. In general it is expected that these wall crossing factors encode counts of holomorphic disks on the mirror A-side, which was conjectured by Fukaya in \cite[Section 3]{fukaya05} to be closely related to Witten's Morse theory.
\end{remark} 
\section{Single wall diagrams as deformations}\label{onewall}

As before, we will work with a contractible open coordinate chart $U \subset B_0$. In this section, we consider a scattering diagram with only one wall $\mathbf{w} = (m, P, \Theta)$,
where $P$ is a connected oriented tropical hyperplane in $U$. Recall that we can write
\begin{equation}\label{wallcrossingfactor}
\text{Log}(\Theta) = \sum_{j, k \geq 1} a_{jk} \bmc^{km}\check{\partial}_{n} t^{j},
\end{equation}
where $a_{jk} \neq 0$ for only finitely many $k$'s for each fixed $j$.

The hyperplane $P$ divides the base $U$ into two half spaces $H_+$ and $H_-$ according to the orientation of $P$,
meaning that $\nu_P$ should be pointing into $H_+$ where $\nu_P \in TU$ is the normal to $P$ we choose so that the orientation of $TP \oplus \real \cdot \nu_P$ agrees with that of $U$.
We consider a step-function-like section $\psi \in \Omega^{0,0}(\check{X}_0(U) \setminus \check{p}^{-1}(P), T^{1,0}\check{X}_0)[[t]]$ of the form
\begin{equation}\label{eqn:step_function_gauge}
\psi
= \left\{
\begin{array}{ll}
\text{Log}(\Theta) & \text{on $H_+$},\\
0& \text{on $H_-$}.
\end{array}\right.
\end{equation}
Our goal is to write down an ansatz $\onewall = \onewall_{\hp}$ (depending on $\hp$) solving the Maurer-Cartan equation \eqref{eqn:MC_eqn_A-side} such that
$\onewall = e^{\facs} * 0 \in \Omega^{0,1}(\check{X}_0(U), T^{1,0})$
represents a smoothing of $e^{\psi} * 0$ (which is delta-function-like and not well defined by itself),
and show that the semi-classical limit of $\facs$ is precisely $\psi$ as $\hp \rightarrow 0$.

\subsection{Ansatz corresponding to a single wall}\label{sec:ansatz_one_wall}
We are going to use the Fourier transform $\mathcal{F}: \widehat{\mathbf{G}}^*(U) \rightarrow KS_{\check{X}_0}(U)[[t]]$ defined in Definition \ref{def:fourier_transform} to obtain an element $\onewall \in \widehat{\mathbf{G}}^*(U)$, and perform all the computations on $\mathbf{G}^*_N(U)$  or $\widehat{\mathbf{G}}^*(U)$ following Fukaya's ideas \cite{fukaya05}. We will omit the Fourier transform $\mathcal{F}$ in our notations, and we will work with tropical geometry on $B_0$ instead of Witten-Morse theory on $\check{B}_0$ following Gross-Siebert's idea \cite{gross2011real}. We start by choosing some convenient affine coordinates $u_{m,i}$'s (or simply $u_i$'s, if there is no confusion) on $U$ for each Fourier mode $m$.

\begin{notation}\label{orthonormalcoordinates}
For each Fourier mode $m \in M \setminus \{0\}$, we choose affine coordinates $(u_1, \dots, u_n)$ for $U$ with the properties that $u_1$ is along $-m$. We will denote the remaining coordinates by $\uperp := (u_2, \ldots, u_n)$. We further require that the coordinates $u_i$'s for $m$ and $km$ is the same with $k\in \inte_+$ for convenience.
\end{notation}

Given a wall $\mathbf{w} = (m, P, \Theta)$ as above, we choose $u_2$ to be the coordinate normal to $P$ and pointing into $H_+$. We consider a $1$-form depending on $\hp \in \real_{>0}$ given by
\begin{equation}\label{deltadefinition}
\delta_{m} = \delta_{m,\hp} := \left( \frac{1}{\hp \pi} \right)^{\half} e^{-\frac{u_2^2}{\hp}} du_2,
\end{equation}
which has the property that $\int_{L} \delta_{m} = 1 + O(e^{-c/\hp})$ for any line $L \cong \real$ intersecting $P$ transversally. This gives a bump form which can be viewed as a smoothing of the delta 1-form over $P$, as shown in Figure \ref{fig:delta_function}.
We will sometimes write $\delta_{m} = e^{-\frac{u_2^2}{\hp}} \mu_{m}$, where $\mu_{m} := \left(\frac{1}{\hp \pi} \right)^{\half} du_2$, to avoid repeated appearances of the constant $\left(\frac{1}{\hp \pi} \right)^{\half}$.

\begin{figure}[h]
\begin{center}
\includegraphics[scale=0.35]{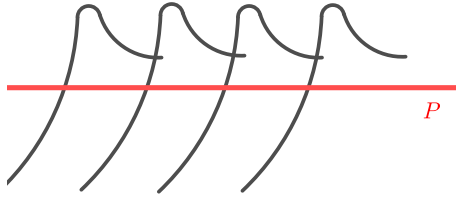}
\end{center}
\caption{$\delta_{m}$ concentrating along $P$}\label{fig:delta_function}
\end{figure}

\begin{definition}\label{ansatz}
Given a wall $\mathbf{w}= ( m ,P, \Theta)$ where $\text{Log}(\Theta)$ is as in \eqref{wallcrossingfactor}, we let
\begin{equation}
\onewall := - \delta_m \cdot \text{Log}(\Theta) = - \delta_m \sum_{j, k \geq 1} a_{jk} \bmc^{km} \check{\partial}_{n} t^{j} \in \widehat{\mathbf{G}}^1(U)
\end{equation}
be the {\em ansatz} associated to the wall $\mathbf{w}$, by viewing $\mathbf{G}^*$ as a module over $\Omega^*(U)$.
\end{definition}

\begin{remark}
Our ansatz depends on the choice of the affine coordinates $(u_1, \dots, u_n)$ because $\delta_m$ does so, but the property that it has support concentrated along the tropical hyperplane $P$ is an abstract notion which does not depend on the choice of coordinates, as we will see shortly.
\end{remark}

\begin{prop}\label{A_mcequation}
The ansatz $\onewall$ satisfies the Maurer-Cartan (MC) equation 
$\pdb \onewall + \half \left[\onewall, \onewall \right] = 0$.
\end{prop}
\begin{proof}
In fact we will show that both terms $\pdb\onewall$ and $\half [\onewall, \onewall ]$ vanish. First we have $\pdb ( \delta_{m} \bmc^{km} \check{\partial}_n) =(d\delta_{m}) \bmc^{km} \check{\partial}_n=  0$ from the fact that $d (\delta_{m}) = 0$ which is obvious from \eqref{deltadefinition}. Next we show that $\left[ \delta_{m} \bmc^{k_1m} \check{\partial}_n, \delta_{m} \bmc^{k_2m} \check{\partial}_n \right] =0$ for any $k_1, k_2$. This is simply because $\delta_{m} = e^{-\frac{u_2^2}{\hp}} \mu_{m}$ and $\mu_{m}$ is a covariant constant form (with respect to the affine connection), so we have
$
\left[ \delta_{m} \bmc^{k_1m} \check{\partial}_n, \delta_{m} \bmc^{k_2m} \check{\partial}_n \right]
= \mu_{m} \wedge \mu_{m} \left[e^{-\frac{u_2^2}{\hp}}\bmc^{k_1m} \check{\partial}_n, e^{-\frac{u_2^2}{\hp}}  \bmc^{k_2m} \check{\partial}_n \right]= 0.
$
\end{proof}

\subsection{Relation with the wall crossing factor}\label{sec:relating_wall_crossing_fac}

Since $\check{X}_0(U) \cong U \times T^n$ has no non-trivial deformations, the element $\check{\onewall}$ must be gauge equivalent to $0$. In this subsection, we will explain how the semi-classical limit of the gauge is related to the wall crossing factor $\text{Log}(\Theta)$.

\subsubsection{Solving for the gauge $\varphi$}

So we are going to solve the equation $e^{\facs} * 0 = \onewall$
for $\facs \in \widehat{\mathbf{G}}^*(U)$ with desired asymptotic behavior. Using the definition in \cite[Section 1]{manetti2005differential} for gauge action, we are indeed solving
\begin{equation}\label{gaugeequation}
- \left( \frac{e^{ad_\facs}- \text{Id}}{ad_\facs} \right) \pdb \facs = \onewall.
\end{equation}
Solutions $\facs$ to \eqref{gaugeequation} is not unique. We will make a choice by choosing a homotopy operator $\hat{H}$. Since $\mathbf{G}^*(U) = \bigoplus_m (\Omega^*(U) \bmc^m) \otimes_\inte N$ is a tensor product of $N$ with a direct sum, it suffices to define a homotopy operator $\hat{H}_m$ for each Fourier mode $m$ contracting $ \Omega^*(U)$ to its cohomology $H^*(U) \cong \comp$.

\begin{definition}\label{real2homotopy}
We fix a based point $q \in H_-$. By contractibility, we have the map $\rho_q : [0,1] \times U \rightarrow U$ satisfying $\rho_q(0,\cdot) = q$ and $\rho_q(1,\cdot) = \text{Id}$, which contracts $U$ to $\{q\}$.

This defines a homotopy operator $\hat{H}_m : \Omega^*(U) \bmc^m \rightarrow \Omega^*(U)[-1] \bmc^m$ by $\hat{H}_m (\alpha \bmc^m):= \int_{0}^1 \rho_q^*(\alpha) \bmc^m$. We also define the projection $\hat{P}_m: \Omega^*(U) \bmc^m \rightarrow H^*(U) \bmc^m$ by setting $\hat{P}_m (\alpha \bmc^m) = \alpha|_q \bmc^m$ for $\alpha \in \Omega^0(U)$ and $0$ otherwise,
and $\iota_m: H^*(U) \bmc^m  \rightarrow \Omega^*(U) \bmc^m$ by setting
$\iota_m: H^*(U) \bmc^m  \rightarrow \Omega^*(U) \bmc^m$ to be the embedding of constant functions on $U$ at degree $0$ and $0$ otherwise.

These operators can be put together to define operators on $\mathbf{G}^*(U)$, $\mathbf{G}^*_N(U)$ or $\widehat{\mathbf{G}}^*(U)$ and they are denoted by $\hat{H}$, $\hat{P}$ and $\iota$ respectively.
\end{definition}

\begin{remark}
The based point $q$ is chosen so that the semi-classical limit of the gauge $\facs_0$ (as $\hp \to 0$) behaves like a step-function across the wall $P$. There are many possible choices of $\hat{H}$, corresponding to choices of $\rho_q$ for this purpose.
In Definition \ref{MC_homotopy}, we will write down another particular choice (suitable for later purposes) in the case when the open subset $U$ is spherical (see Section \ref{fig:spherical_nbh}).
\end{remark}

In the rest of this section, we will fix $q \in H_-$ in the half space $H_-$ and impose the gauge fixing condition $\hat{P}\facs= 0$ to solve for $\facs$ satisfying \eqref{gaugeequation}; in other words, we look for a solution satisfying $\facs = \hat{H} \pdb \facs + \pdb \hat{H} \facs = \hat{H} \pdb \facs$
to solve the equation \eqref{gaugeequation} order by order; here $\hat{H}\facs = 0$ by degree reasons. This is possible because of the following lemma which we learn from \cite{manetti2005differential}.
\begin{lemma}\label{gauge_fixing_lemma}
Among all solutions of $e^\facs * 0 = \onewall$, there exists a unique one satisfying $\hat{P}\facs = 0$.
\end{lemma}
\begin{proof}
Notice that for any $\sigma = \sigma_1 + \sigma_2 + \cdots \in t \cdot \widehat{\mathbf{G}}^*(U)$ with $\pdb \sigma = 0$, we have $e^\sigma * 0 = 0$, and hence $e^{\facs \bullet \sigma} * 0 = \onewall$ is still a solution for the same equation. With $\facs \bullet \sigma$ given by the Baker-Campbell-Hausdorff formula as $\facs \bullet \sigma = \facs+\sigma + \half \{ \facs,\sigma\} + \cdots$, we can then solve the equation $\hat{P} (\facs \bullet \sigma) = 0$ order by order under the assumption that $\pdb \sigma = 0$.
\end{proof}

Under the gauge fixing condition $\hat{P}\facs = 0$, setting
\begin{equation}\label{vartheta_s_definition}
\facs_{s+1} := -\hat{H} \left( \onewall + \sum_{k\geq 0 }\frac{ad_{\facs^s}^k}{(k+1)!} \pdb\facs^s \right)_{s+1},
\end{equation}
where the subscript $s+1$ on the RHS means taking the coefficient of $t^{s+1}$ and $\facs^s := \facs_1 + \cdots \facs_s$,
defines $\facs = \facs_1 + \facs_2 +\cdots$ inductively.
\begin{remark}\label{pathindependent}
Notice that
$$
\pdb \left( \onewall + \sum_{k\geq 0 }\frac{ad_{\facs^s}^k}{(k+1)!} \pdb\facs^s \right)_{s+1} = 0,
$$
so the operator $\hat{H}$, which is defined by integration along paths, is independent of the paths chosen upon applying to these terms.
\end{remark}

\begin{remark}
We also observe that the terms $\facs_s$'s vanish on the direct summand $(\Omega^*(U) \bmc^{\hat{m}})\otimes_\inte N$ whenever $\hat{m} \neq km$. Furthermore, we can see that each $\pdb\facs_s$ (and all its derivatives) decay exponentially as $O_{s,K}(e^{-c_{s,K}/\hp})$ on any compact subset $K \subset H_-$ away from $P$.
\end{remark}

We are going to analyze the behavior of $\facs$ as $\hp\rightarrow 0$ to show that it admits an asymptotic expansion with leading order term exactly given by $\psi$ on $\check{X}_0(U)\setminus \check{p}^{-1}(P)$.

\subsubsection{Asymptotic analysis for the gauge $\facs$}\label{sec:analysis_one_wall}
Observe that when we are considering a single wall $\mathbf{w} = (m, P, \Theta)$, the Maurer-Cartan solution $\onewall$ and hence the gauge $\facs$ will be non-trivial only for the summand $(\Omega^*(U) \bmc^{\hat{m}}) \otimes_\inte N$ where $\hat{m} = km$ for some $k \in \inte_{>0}$. We use the affine coordinates $u = (u_1, \dots, u_n)$ from Notations \ref{orthonormalcoordinates} for each component $U$ for all these summands.


By Remark \ref{pathindependent}, when dealing with closed 1-forms, we can replace the operator $\hat{H}_{\hat{m}}$ by the path integral over any path with the same end points. Let us consider the path $\varrho_{u}$ defined by
$$
\varrho_{u} = \varrho_{u_1,\uperp}(t)
= \left\{
\begin{array}{ll}
((1-2t)u_1^0 + 2tu_1, \uperp^0) & \text{if $t \in [0,\half]$},\\
(u_1, (2t-1)\uperp + (2-2t)\uperp^0) &  \text{if $t\in [\half,1]$},
\end{array}\right.
$$
where $u^0 = q$ (see the left picture of Figure \ref{fig:two_integrals}). From now on, we will assume that $\varrho_u$ is contained in the contractible open set $U$ by shrinking $U$ if necessary. Then we define the operator $\hat{I}$ by
\begin{equation}\label{I_integral}
\hat{I}(\alpha) = :\int_{\varrho_{u}} \alpha.
\end{equation}
By what we just said, we have $\hat{H}_{\hat{m}}  (\alpha \bmc^{\hat{m}})  = \hat{I}(\alpha) \bmc^{\hat{m}} = (\int_{\varrho_{u}} \alpha) \bmc^{\hat{m}}$ for closed $1$-forms $\alpha$.

We are going to apply $\hat{I}$, instead of $\hat{H}$, to the closed 1-form
$$
\left( \onewall + \sum_{k\geq 0 }\frac{ad_{\facs^s}^k}{(k+1)!} \pdb\facs^s \right)_{s+1}
$$
to solve for $\varphi_{s+1}$ because this could somewhat simplify the asymptotic analysis below.


First of all, the first term $\facs_1$ can be explicitly expressed as $\facs_1 = \sum_k  a_{1k} \hat{I}(\delta_m) \bmc^{km} \check{\partial}_n$,
where $\delta_m$ is the 1-form defined in \eqref{deltadefinition} and $\check{\partial}_n$ is the affine vector field defined in \eqref{holo_vector_field_2}.\footnote{Note that there is a factor $\frac{\hp}{4\pi}$ in front of the expression of $\check{\partial}_n$ in \eqref{holo_vector_field_2}, which will become important later when we count the $\hp$ orders in the asymptotic expansions.}
Since
$$
\hat{I}(\delta_m) = \int_{\varrho_{u}}\delta_m = \left(\frac{1}{\hp \pi}\right)^{\half} \int_{\varrho_{u}} e^{-\frac{u_2^2}{\hp}} du_2
= \left\{
\begin{array}{ll}
1 + O_{loc}(\hp) & \text{if $u \in H_+$},\\
O_{loc}(\hp) & \text{if $u \in H_-$},
\end{array}\right.
$$
we see that $\facs_1$ has the desired asymptotic expansion, with leading order term given by the coefficient of $t^1$ in $\psi $ given in \eqref{eqn:step_function_gauge}, where the notation $O_{loc}(\hp)$ means the following:

\begin{notation}\label{O_loc}
We say that a function $f(x,\hp)$ on an open subset $U\times \real_{>0} \subset B_0 \times \real_{>0}$ belongs to $O_{loc}(\hp^l)$ if it is bounded by $C_K \hp^l$ for some constant $C_K$ (independent of $\hp$) on every compact subset $K \subset U$.
\end{notation}

Next we consider the second term $\facs_2$. Notice that $\left[ \bmc^{k_1 m } \check{\partial}_{n}, \bmc^{k_2m} \check{\partial}_{n} \right] = 0$ for all positive $k_1,k_2$. Therefore we have
\begin{equation}\label{lemma_s_equal_2}
\left[ \facs_1 , \onewall_1 \right]  =
- \sum_{k_1,k_2} a_{1k_1} a_{1k_2} \left( \hat{I}(\delta_m) (\nabla_{\partial_n} \delta_m) \check{\partial}_{n} - \delta_m \nabla_{\partial_n}(\hat{I}(\delta_m)) \check{\partial}_{n} \right) \bmc^{(k_1+k_2)m},
\end{equation}
where $\onewall_s$ refers to the coefficient of $t^s$ in $\onewall$ and $\partial_n$ was introduced in \eqref{affine_vector_field}.

To compute the order of $\hp$ in each term in \eqref{lemma_s_equal_2}, we first have $|\hat{I}(\delta_m)| \leq 2$ from the definition of $\delta_m$ in \eqref{deltadefinition}, and using the formula $\delta_{m} = (\frac{1}{\hp \pi})^{\half}(e^{-\frac{u_2^2}{\hp}} du_2)$, we get
$$
|\hat{I} (\hat{I}(\delta_m)\nabla_{\partial_n} \delta_m)|
=  \frac{\hp}{4\pi} \big| \int_{\varrho_{u}}\hat{I}(\delta_m) \nabla_{g_{jk} n^j \dd{x_k}} (\delta_m) \big| \leq C \hp^{1/2} \big| \int_{\varrho_{u}} (\nabla_{g_{jk} n^j \dd{x_k}} e^{-\frac{u_2^2}{\hp}} du_2) \big|
 \leq C \hp^{1/2}.
$$
This follows from the fact that $\nabla (u_2)^2$ vanishes along $P$ up to first order, giving an extra $\hp^{1/2}$ upon integrating against $e^{-\frac{u_2^2}{\hp}}$. Similarly, we can show that $|\hat{I} \left( \delta_m \nabla_{\partial_n}(\hat{I}(\delta_m)) \right) |\leq C \hp^{1/2}$.
Therefore we have
$$
\facs_2
=\begin{cases}
\displaystyle \sum_{k \geq 1} a_{2k} \bmc^{km} \check{\partial}_n +\sum_{k\geq 1} O_{loc}(\hp^{1/2}) \bmc^{km} \check{\partial}_{n}  &\text{on $\check{p}^{-1}(H_+)$},\\
\displaystyle \sum_{k \geq 1} O_{loc}(\hp^{1/2}) \bmc^{km} \check{\partial}_{n}  & \text{on $\check{p}^{-1}(H_-)$},
\end{cases}
$$
where the notation $O_{loc}(\hp^{1/2})\bmc^{km}\check{\partial}_n$ means a finite sum of terms of the form $\phi\bmc^{km}\check{\partial}_n$ with $\phi \in O_{loc}(\hp^{1/2})$.

We would like to argue that the same kind of asymptotic formula holds for a general term $\facs_s$ as well. To study the order of $\hp$ in derivatives of the function $e^{-\frac{u_2^2}{\hp}}$, we need the following stationary phase approximation (see e.g. \cite{dimassi1999spectral}).

\begin{lemma}\label{stat_phase_exp}
Let $U\subset \real^n$ be an open neighborhood of $0$ with coordinates $x_1,\ldots,x_n$. Let $\varphi:U\rightarrow\real_{\geq 0}$ be a Morse function with unique minimum $\varphi(0)=0$ in $U$. Let $\tilde{x}_1,\ldots,\tilde{x}_n$ be a set of Morse coordinates near $0$ so that $\varphi(x) = \frac{1}{2}(\tilde{x}_1^2 + \cdots + \tilde{x}_n^2)$. For every compact subset $K\subset U$, there exists a constant $C=C_{K,N}$ such that for every $u \in C^{\infty}(U)$ with $\text{supp}(u) \subset K$, we have
\begin{equation}\label{stat_ineq}
 |(\int_K e^{-\varphi(x)/\he} u ) - (2\pi\hp)^{n/2}\left(\sum_{k=0}^{N-1} \frac{ \hp^{k}}{2^kk!} \tilde{\Delta}^k(\frac{u}{\Im})(0) \right)|\leq C \he^{n/2+N} \sum_{|\alpha|\leq 2N+n+1} \sup|\pd^{\alpha}u|,
\end{equation}
where $\tilde{\Delta} = \sum\ppd{\tilde{x}_j}$, $ \Im = \pm\det(\frac{d\tilde{x}}{dx})$ and $\Im(0)=(\det\Hess\varphi(0))^{1/2}$. In particular, if $u$ vanishes at $0$ up to order $L$, then we can take $N = \lceil L/2 \rceil$ and get $|\int_K e^{-\varphi(x)/\he} u | \leq C \he^{n/2+\lceil L/2 \rceil}$.
\end{lemma}

We will keep track of the order of $\hp$ in solving the general equation \eqref{vartheta_s_definition}, and will see that the leading order contribution of $\facs_{s+1}$ simply comes from $- H (\onewall_{s+1})$. From the above calculation, we learn that for the 1-form $\delta_m$ defined in \eqref{deltadefinition}, any differentiation $\nabla_{\check{\partial}_n} (\delta_m)$ will contribute an extra vanishing of order $\hp^{1/2}$, and hence can be considered as an error term. Systematic tracking of these $\hp$ orders during the iteration \eqref{vartheta_s_definition} is necessary. So we extract such properties of $\delta_m$ which we need later in the following lemma.

Given a wall $P \subset U$, there is an affine foliation $\{P_q\}_{q \in N}$ of $U$, where each $P_q$ is a tropical hyperplane parallel to $P$ and $N$ is an affine line transversal to $P$ which parametrizes the leaves, as shown in Figure \ref{fig:foliation}. Given any point $p \in P$ and a neighborhood $V \subset U$ containing $p$, there is an induced affine foliation $\{(P_{V,q})\}_{q \in N}$ on $V$.
\begin{lemma}\label{gmnormestimate}
Using $u_2$ as a coordinate for $N$ so that $q = u_2 \in N$ (recall that specific affine coordinates $(u_1, \dots, u_n)$ in $U$ have been chosen in Notations \ref{orthonormalcoordinates}), and considering the function $g := (u_2)^2$, we have the integral estimate
$$
\int_{N} u_2^r \left(\sup_{P_{V, u_2}}| \nabla^j (e^{-\frac{g}{\hp}})| \right) du_2  \leq  C_{j,r,V} \hp^{-\frac{j-r}{2} + \half}
$$
for any $j, r\in \inte_{\geq 0}$.
\end{lemma}

\begin{proof}
First we notice that $\nabla^j (e^{-g/\hp})$ consists of terms of the form $\hp^{-M}\left( \prod_{i=1}^M (\nabla^{s_i} g) \right) e^{-g/\hp}$, where $\sum_{i} s_i  = j$. We see that $\nabla^l \left( \prod_{i=1}^M (\nabla^{s_i} g) \right) |_{P} \equiv 0$ for $l \leq \sum_{i=1}^{M} \max(0,2-s_i) =: L$. We observe that the terms contributing to the lowest $\hp$ power are either of the form $\hp^{-\lfloor\frac{j+1}{2}\rfloor} u_2^r \prod_{i=1}^{\lfloor\frac{j+1}{2}\rfloor}( \nabla^{s_i}g) e^{-g/\hp}$ having $s_i \leq 2$, or of the form $\hp^{-j} u_2^r \prod_{i=1}^j (\nabla^{s_i} g) e^{-g/\hp}$ having $s_i = 1$. In both cases, applying the stationary phase approximation in Lemma \ref{stat_phase_exp} and counting the vanishing order along $P$, we obtain
$
\int_{N} u_2^r \left(\sup_{P_{V,q}}| \nabla^j (e^{-\frac{g}{\hp}})| \right) du_2  \leq C_{j,r,V} \hp^{1/2 + (r+j)/2 - j} = C_{j,r,V} \hp^{(r-j)/2+1/2}.
$
\end{proof}

The reader may notice that taking the supremum $\sup_{P_{V,u_2}}$ in Lemma \ref{gmnormestimate} is redundant because $g$ is constant along the leaves of the foliation $\{(P_{V,q})\}_{q \in N}$; we write it in this way in order to match Definition \ref{asypmtotic_support_def} below.

\begin{remark}
The order $\hp^{-\frac{j-r}{2} + \half}$ which appears in Lemma \ref{gmnormestimate} is related to a similar weighted $L^2$ norm in \cite{HelSj1}.
\end{remark}

\subsubsection{Differential forms with asymptotic support}\label{sec:asy_support}

Motivated by the procedure of tracking the $\hp$ orders as in Lemma \ref{gmnormestimate}, we would like to formulate the notion of a differential $k$-form
having {\em asymptotic support on a closed codimension $k$ tropical polyhedral subset $P \subset U$}; by a tropical polyhedral subset we mean a connected locally convex subset which is locally defined by affine linear equations or inequalities over $\mathbb{Q}$, as in the codimension 1 case above (Definition \ref{wall}).
Before doing so, we first need to define the notion of a differential $k$-form having exponential decay, or more precisely, having exponential order $O(e^{-c/\hp})$; the error terms which appear in our later discussion will be of such shape:

\begin{notation}\label{hp_diff_forms}
We will use the notation $\Omega^*_\hp(B_0)$ (similarly for $\Omega^*_\hp(U)$) to stand for $\Gamma(B_0 \times \real_{>0},\bigwedge^*T^*B_0)$, where the extra $\real_{>0}$ direction is parametrized by $\hp$.
\end{notation}

\begin{definition}\label{exponential_decay}
We define $\filt^{-\infty}_k(U) \subset \Omega^k_\hp(U)$ to be those differential $k$-forms $\alpha \in \Omega^k_\hp(U)$ such that for each point $q \in U$, there exists a neighborhood $V$ of $q$ where we have
$\|\nabla^j \alpha\|_{L^\infty(V)} \leq D_{j,V} e^{-c_V/\hp}$
for some constants $c_V$ and $D_{j,V}$. The association $U \mapsto \filt^{-\infty}_k(U)$ defines a sheaf over $B_0$ which is denoted by $\filt^{-\infty}_k$.
\end{definition}

We will also consider differential forms which only blow up at polynomial orders in $\hp^{-1}$:

\begin{definition}\label{polynomial_growth}
We define $\filt^{\infty}_k(U) \subset \Omega^k_\hp(U)$ to be those differential $k$-forms $\alpha \in \Omega^k_\hp(U)$ such that for each point $q \in U$, there exists a neighborhood $V$ of $q$ where we have
$\|\nabla^j \alpha\|_{L^\infty(V)} \leq D_{j,V} \hp^{-N_{j,V}}$
for some constant $D_{j,V}$ and $N_{j,V} \in \inte_{>0}$. The association $U \mapsto \filt^{\infty}_k(U)$ defines a sheaf over $B_0$ which is denoted by $\filt^{\infty}_k$.
\end{definition}

Notice that the sheaves $\filt^{\pm \infty}_k$ in Definitions \ref{exponential_decay} and \ref{polynomial_growth} are closed under application of $\nabla_{\dd{x}}$, the deRham differential $d$ and wedge product of differential forms. We also observe the fact that $\filt^{-\infty}_k$ is a differential graded ideal of $\filt^{\infty}_k$; this will be useful later in Section \ref{sec:solve_MC_two_walls}. In particular, we can consider the sheaf of differential graded algebras $\filt^{\infty}_*/\filt^{-\infty}_*$, equipped with the deRham differential.

The following lemma will be useful in Section \ref{sec:key_lemmas} (readers may skip it until Section \ref{sec:key_lemmas}):

\begin{lemma}\label{path_independent}
\begin{enumerate}
\item
Suppose that $\alpha \in (\filt^{\infty}_k/\filt^{-\infty}_k)(U) = \filt^{\infty}_k(U)/\filt^{-\infty}_k(U)$ (note that the sheaves $\filt^{\pm \infty}_*$ are both soft sheaves) satisfies $d\alpha = 0$ in $(\filt^{\infty}_{k+1}/\filt^{-\infty}_{k+1})(U)$. Then for any compact family of smooth $(k+1)$-chain $\{\gamma_u\}_{u \in K}$ in $U$, we have
$$
|\int_{\partial \gamma_u} \hat{\alpha} | \leq D_{K,\hat{\alpha}} e^{-c_{K,\hat{\alpha}}/\hp}
$$
for any $u \in K$, where $\hat{\alpha}$ is any choice of lifting of $\alpha$ to $\filt^{\infty}_k(U)$.

\item
For any $\alpha \in (\filt^{\infty}_1/\filt^{-\infty}_1)(U)$ and a fixed based point $x^0 \in U$, the path integral $f_\alpha := \int_{x^0}^x \hat{\alpha}$, defined locally by first choosing a contractible compact subset $K\subset U$, then a family of paths $\varrho:[0,1]\times K \rightarrow U$ joining $x^0$ to $x \in K$, and also a lifting $\hat{\alpha}$ of $\alpha$ to $\filt^{\infty}_1(U)$, gives a well-defined element in $(\filt^{\infty}_0/\filt^{-\infty}_0)(U)$, meaning that for different choices of $K$, $\varrho$ and $\hat{\alpha}$, the path integrals only differ by elements in $\filt^{-\infty}_0(U)$.
\end{enumerate}
\end{lemma}

\begin{proof}
For the first statement, the equation $d\alpha = 0$ in $(\filt^{\infty}_k/\filt^{-\infty}_k)(U)$ means that we have $d \hat{\alpha} = \beta$ for some $\beta \in \filt^{-\infty}_{k+1}(U)$. Stokes' Theorem then implies that
$\int_{\partial \gamma_u } \hat{\alpha} = \int_{\gamma_u} \beta = O_{K,\beta}(e^{-c_{K,\beta}/\hp})$.

For the second statement, we first fix a point $x \in U$ and a contractible compact subset $K\subset U$ such that $x \in \text{int}(K)$, and also a lifting $\hat{\alpha}$ with $d \hat{\alpha} = \beta \in \filt^{-\infty}_2(U)$. Suppose that we have two families of paths $\varrho_1, \varrho_2 : [0,1] \times K \rightarrow U$ parametrized by $K$. Using contractibility of $U$, we have a homotopy $h : [0,1]^2 \times K \rightarrow U$ between $\varrho_1$ and $\varrho_2$ satisfying $h(0,\cdot) = \varrho_0$, $h(1,\cdot) = \varrho_1$, $h(\cdot,0) = x^0$ and $h(\cdot,1) = id_K$. Therefore we have $\varrho_1 - \varrho_0 = \partial h$, and hence the difference of the two path integrals is given by
$$
\int_{ \varrho_{1}(\cdot,x) } \hat{\alpha} - \int_{\varrho_0(\cdot,x)} \hat{\alpha} = \int_{h(\cdot,x)} d\hat{\alpha} = \int_{h(\cdot,x)}\beta = \int_{[0,1]^2} h^*(\beta).
$$
Taking the covariant derivatives by $\nabla^j$ of this difference, we have
$
\nabla^j\left(\int_{ \varrho_{1}(\cdot,x) } \hat{\alpha} - \int_{\varrho_0(\cdot,x)} \hat{\alpha}\right) = \int_{[0,1]^2} \nabla^j(h^*(\beta)).
$
From the fact that $\beta \in \filt^{-\infty}_2(U)$, we have $|\nabla^j(h^*(\beta))|(s,t,u) \leq D_{j,K,h} e^{-c_{j,K,h}/\hp}$ for any point $(s,t,u) \in [0,1]^2 \times K$, and therefore $\|\nabla^j(f_{1,\hat{\alpha}} - f_{2,\hat{\alpha}})\|_{L^{\infty}(K)} \leq D_{j,K,h}e^{-c_{j,K,h}/\hp}$. Hence $f_\alpha$, as an element of $(\filt^{\infty}_0/\filt^{-\infty}_0)(U)$, is independent of the choice of the family of paths $\varrho$.

Now if $K_1, K_2$ are two contractible compact subsets with $x \in \text{int}(K_1 \cap K_2)$, we can change the family of paths parametrized by each $K_i$ to an auxiliary one parametrized by $K_1 \cap K_2$. By above, the path integral will only differ by elements in $\filt^{-\infty}_0(U)$.
Finally, for two different liftings $\hat{\alpha}_1, \hat{\alpha}_2$ of $\alpha$, we have $\hat{\alpha}_1 - \hat{\alpha}_2 \in \filt^{-\infty}_1(U)$ and so
$\int_{x_0}^x \hat{\alpha}_1 - \hat{\alpha}_2 \in \filt^{-\infty}_0(U)$. This completes the proof of the second statement.
\end{proof}


\begin{notation}\label{foliation_notation}
Let $P \subset U$ be a closed codimension $k$ tropical polyhedral subset.
\begin{enumerate}
\item
There is a natural foliation $\{P_q\}_{q \in N}$ in $U$ obtained by parallel transporting the tangent space of $P$ (at some interior point in $P$) to every point in $U$ by the affine connection $\nabla$ on $B_0$.

We let $\nu_P \in \Gamma(U,\bigwedge^k (N^*_P))$ be a top covariant constant form (i.e. $\nabla(\nu_P) = 0$) in the conormal bundle $N^*_P$ of $P$ (which is unique up to scaling by constants); we regard $\nu_P$ as a volume form on space of leaves $N$ if it admits a smooth structure. We also let $\nu_P^\vee \in \wedge^k N_P$ be a volume element dual to $\nu_P$, and choose a lifting of $\nu_P^\vee$ as an element in $\wedge^k TU$ (which will again be denoted by $\nu_P^\vee$ by abusing notations).

\begin{figure}[h]
\centering
\includegraphics[scale=0.4]{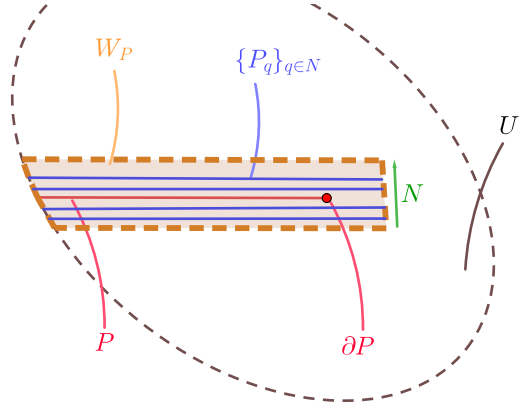}
\caption{The foliation near $P$}
\label{fig:foliation}
\end{figure}

\item
For any point $p \in P$, we choose a sufficiently small convex neighborhood $V \subset U$ containing $p$ so that there exists a slice $N_V \subset V$ transversal to the foliation $\{P_q \cap V\}$ given by intersection of $\{P_q\}_{q \in N}$ with $V$, i.e. a dimension $k$ affine subspace which is transversal to all the leaves in $\{P_q \cap V\}$; we denote this foliation on $V$ by $\{(P_{V,q})\}_{q \in N_V}$, using $N_V$ as the parameter space. See Figure \ref{fig:foliation} for an illustration.

In $V$, we take local affine coordinates $x = (x_1,\dots x_n)$ such that $x' := (x_1, \dots, x_k)$ parametrizes $N_V$ with $x' = 0$ corresponding to the unique leaf containing $P$.
Using these coordinates, we can write $\nu_P = dx_1\wedge\cdots\wedge dx_k$ and $\nu_P^\vee = \dd{x_1}\wedge\cdots \wedge\dd{x_k}$.
\end{enumerate}
\end{notation}

\begin{definition}\label{asypmtotic_support_def}
A differential $k$-form $\alpha \in \filt^{\infty}_k(U)$ is said to have {\em asymptotic support on a closed codimension $k$ tropical polyhedral subset $P \subset U$} if the following conditions are satisfied:
\begin{enumerate}
\item
For any $p \in U \setminus P$, there is a neighborhood $V \subset U \setminus P$ of $p$ such that $\alpha|_V \in \filt^{-\infty}_k(V)$ on V.

\item
There exists a neighborhood $W_P$ of $P$ in $U$ such that we can write
$$\alpha =  h(x,\hp) \nu_P + \eta,$$
where $\nu_P$ is the volume form Notations \ref{foliation_notation}(1), $h(x,\hp) \in C^\infty(W_P \times \real_{>0})$ and $\eta$ is an error term satisfying $\eta \in \filt^{-\infty}_k(W_P)$ on $W_P$ (see Figure \ref{fig:delta_function}).

\item
For any $p \in P$, there exists a sufficiently small convex neighborhood $V$ containing $p$ such that using the coordinate system chosen in Notations \ref{foliation_notation}(2) and considering the foliation $\{(P_{V, x'})\}_{x' \in N_V}$ in $V$, we have, for all $j \in \inte_{\geq 0}$ and multi-index $\beta = (\beta_1,\dots,\beta_k) \in \inte_{\geq 0}^k$, the estimate
\begin{equation}\label{estimate_order_s}
\int_{x'\in N_V}  (x')^\beta \left(\sup_{P_{V,x'}}|\nabla^j (\iota_{\nu_P^\vee} \alpha)| \right) \nu_P \leq D_{j,V,\beta} \hp^{-\frac{j+s-|\beta|-k}{2}},
\end{equation}
for some constant $D_{j,V,\beta}$ and some $s \in \inte$, where $|\beta| = \sum_l \beta_l$ is the vanishing order of the monomial $(x')^{\beta}= x_1^{\beta_1}\cdots x_k^{\beta_k}$ along $P_{x'=0}$.
\end{enumerate}
\end{definition}

\begin{remark}
Note that condition (3) in Definition \ref{asypmtotic_support_def} is independent of the choice of the convex neighborhood $V$, the transversal slice $N_V$ and the choice of the local affine coordinates $x = (x_1,\dots x_n)$ (although the constant $D_{j,V,\beta}$ may depends these choices). Therefore this condition can be checked by choosing a sufficiently nice neighborhood $V$ at every point $p \in P$.
\end{remark}

\begin{remark}
The idea of putting the weight $(x')^\beta$ and the differentiation $\nabla^j$ in condition (3) in Definition \ref{asypmtotic_support_def} comes from a similar weighted $L^2$ norm used in \cite{HelSj1}. In this paper, instead of $L^2$ norms, we use a mixture of $L^\infty$ and $L^1$ norms for the purpose of Lemma \ref{support_product}.
\end{remark}

The estimate in condition (3) of Definition \ref{asypmtotic_support_def} defines the following filtration
\begin{equation}\label{filtration}
\filt^{-\infty}_k \cdots \subset \filt^{-s}_P\subset \cdots \filt^{-1}_P\subset \filt^0_P \subset \filt^1_P \subset \filt^2_P\subset \cdots \subset \filt^s_P \subset \cdots \subset \filt^{\infty}_k \subset \Omega^k_\hp(U),
\end{equation}
where, for any given $s \in \inte$, $\filt^s_P = \filt^s_P(U)$ denotes the set of $k$-forms $\alpha \in \filt^{\infty}_k(U)$ with asymptotic support on $P$ such that the estimate \eqref{estimate_order_s} holds with the given integer $s$. Note that the degree $k$ of the differential forms has to be equal to the codimension of $P$. Also note that the sets $\filt^{\pm \infty}_k(U)$ are independent of the choice of $P$.
This filtration keeps track of the polynomial order of $\hp$ for $k$-forms with asymptotic support on $P$, and it provides a convenient tool for us to prove and express our results in the subsequent asymptotic analysis.
In these terms, Lemma \ref{gmnormestimate} simply means $\delta_m \in \filt^1_P(U)$, where $P$ is the tropical hyperplane supporting a wall.

The filtration satisfies
$
\nabla_{\dd{x_l}} \filt^s_P(U) \subset \filt^{s + 1}_P(U)
$
for any $l = 1, \dots, n$, and
$
(x')^{\beta}\filt^s_P(U) \subset \filt^{s-|\beta|}_P(U)
$
for any affine monomial $(x')^{\beta}$ with vanishing order $|\beta|$ along $P$, so we have the nice property that
\begin{equation}\label{asy_support_basic_property}
(x')^{\beta} \nabla_{\dd{x_{l_1}}}\cdots \nabla_{\dd{x_{l_j}}} \filt^s_P(U) \subset \filt^{s+j-|\beta|}_P(U).
\end{equation}

\begin{lemma}\label{support_product}
For two closed tropical polyhedral subsets $P_1, P_2 \subset U$ of codimension $k_1, k_2$ respectively, we have $\filt^s_{P_1}(U) \wedge \filt^r_{P_2}(U) \subset \filt^{r+s}_{P}(U)$ for any codimension $k_1+k_2$ polyhedral subset $P$ containing $P_1 \cap P_2$ normal to $\nu_{P_1} \wedge \nu_{P_2}$ if they intersect transversally,\footnote{In particular we can take $P = P_1 \cap P_2$ if $\text{codim}_\real(P_1\cap P_2) = k_1+k_2$.} and $\filt^s_{P_1}(U) \wedge \filt^r_{P_2}(U) \subset \filt^{-\infty}_{k_1 + k_2}(U)$ if their intersection is not transversal.
\end{lemma}

Before giving the proof, let us clarify that, when we say two closed tropical polyhedral subsets $P_1, P_2 \subset U$ of codimension $k_1, k_2$ are {\em intersecting transversally}, we mean the affine subspaces containing $P_1, P_2$ and of codimension $k_1, k_2$ respectively are intersecting transversally; this definition also applies to the case when $\partial P_i \neq \emptyset$, as shown in Figure \ref{fig:double}.
\begin{figure}[h]
\centering
\includegraphics[scale=0.55]{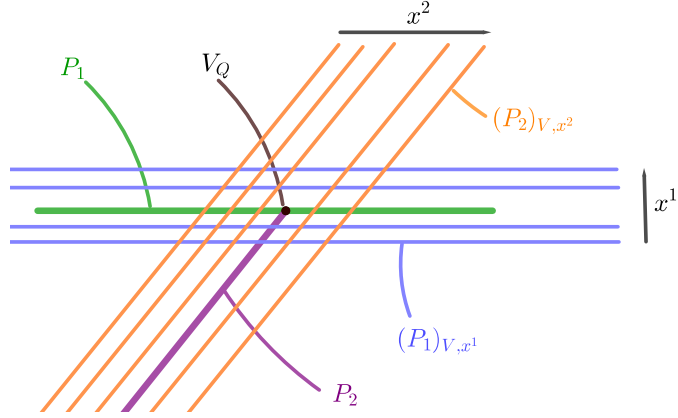}
\caption{Foliation in the neighborhood $V$}
\label{fig:double}
\end{figure}

\begin{proof}[Proof of Lemma \ref{support_product}]
We first consider the case when $P_1$ and $P_2$ are not intersecting transversally. Part $(2)$ of Definition \ref{asypmtotic_support_def} says that we have neighborhoods $W_{P_i}$ of $P_i$ such that we can write $\alpha_i = h_i \nu_{P_i} + \eta_i$ for $i = 1, 2$. Since $\nu_{P_1} \wedge \nu_{P_2} = 0$ in $W_{P_1} \cap W_{P_2}$ by the non-transversal assumption, we have $\alpha_1 \wedge \alpha_2 \in \filt^{-\infty}_k(W_{P_1} \cap W_{P_2})$ near $P_1 \cap P_2$, and hence $\alpha_1 \wedge \alpha_2 \in \filt^{-\infty}_k(U)$ by condition (1) in Definition \ref{asypmtotic_support_def} and the fact that $\filt^{-\infty}_k(U)$ is a differential ideal of $\filt^{\infty}_k(U)$.

Next we assume that $P_1 \pitchfork P_2 = Q$. Let $\alpha_1 \in \filt^{s}_{P_1}(U)$ and $\alpha_2 \in \filt^r_{P_2}(U)$. Using again the fact that $\filt^{-\infty}_k(U)$ is a differential ideal of $\filt^{\infty}_k(U)$, same reasoning as above shows that condition (1) in Definition \ref{asypmtotic_support_def} holds for $\alpha_1 \wedge \alpha_2 \in \filt^{r+s}_{Q}(U)$. Condition (2) in Definition \ref{asypmtotic_support_def} is also satisfied because in this case we have $\nu_{Q} = \nu_{P_1}\wedge \nu_{P_2}$ in $W_Q = W_{P_1}\cap W_{P_2}$. So it remains to prove condition (3) in Definition \ref{asypmtotic_support_def}.

Fixing a point $p \in Q$, we take an affine convex coordinate chart given by $V (\subset T_p U \cong M_\real) \rightarrow U$ centered at $0 \in T_p U$. Then $V_Q := V \cap T_p Q$ is a neighborhood of $0$ in $T_p Q$. We take $\mathbb{Q}$-affine bases $m_2^1, \dots, m_2^{k_2}$ of $T_p P_1 /T_p Q$ and $m_1^1, \dots, m_1^{k_1}$ of $T_p P_2 /T_p Q$ respectively, and the corresponding dual bases in $(T_p U / T_p P_2)^*$ and $(T_p U / T_p P_1)^*$. We use $x^i \cdot m_i = \sum_{j=1}^{k_i} x^i_j m_i^j$, $i = 1, 2$ to stand for the natural pairing between $x^1 = (x^1_1, \dots, x^1_{k_1}) \in (T_p U / T_p P_1)^*$ and $m_{1}^1, \dots, m_1^{k_1}$, and that between $x^2 = (x^2_1, \dots, x^2_{k_2}) \in (T_p U / T_p P_2)^*$ and $m_2^1, \dots, m_2^{k_2}$, respectively.
By shrinking $V$ if necessary, we can write it as
$V = \bigcup_{\substack{x^1 \in (-\delta,\delta)^{k_1}\\x^2 \in (-\delta,\delta)^{k_2}}} \left( x^1\cdot m_1+x^2 \cdot m_2 + V_Q \right)$
for some small $\delta > 0$, as shown in Figure \ref{fig:double}. Then we can parametrize the foliations induced by $Q$, $P_1$ and $P_2$ respectively as $Q_{V,(x^1,x^2)}  = x^1\cdot m_1+x^2 \cdot m_2 + V_Q$, $(P_1)_{V,x^1} = x^1\cdot m_1 + \bigcup_{x^2 \in (-\delta,\delta)^{k_2}}\left( x^2 \cdot m_2 + V_Q \right)$ and $(P_2)_{V,x^2}  = x^2\cdot m_2 + \bigcup_{x^1 \in (-\delta,\delta)^{k_1}}\left( x^1 \cdot m_1 + V_Q \right)$. We also extend $(x^1,x^2)$ to local affine coordinates
$(x^1_1, \dots, x^1_{k_1}, x^2_1, \dots, x^2_{k_2}, x_{k_1 + k_2 + 1}, \dots, x_n)$
on $V$.

Now for $\alpha_1 \in \filt^r_{P_1}(U)$ and $\alpha_2 \in \filt^s_{P_2}(U)$, we first observe that we can write
$
\alpha_i = h_i(x,\hp) dx^i + \eta_i = h_i(x,\hp) dx^i_1 \wedge \cdots dx^i_{k_i} + \eta_i
$
for $i = 1, 2$, and we have
$
\nabla^j (h_1  h_2) = \sum_{j_1+j_2 = j }(\nabla^{j_1} h_1)(\nabla^{j_2} h_2).
$
Also, any affine monomial $(x')^\beta$ (in the coordinates $(x^1_1, \dots, x^1_{k_1}, x^2_1, \dots, x^2_{k_2}$) with vanishing order $|\beta|$ along $Q$ can be rewritten in the form $(x^1)^{\beta_1} (x^2)^{\beta_2}$, where $(x^i)^{\beta_i}$ has vanishing order $|\beta_i|$ along $Q$.

Since the error terms $\eta_i$'s are not contributing when we count the polynomial order in $\hp^{-1}$, it remains to estimate a term of the form $(x^1)^{\beta_1} (x^2)^{\beta_2}(\nabla^{j_1} h_1)(\nabla^{j_2} h_2)$. We have
\begin{align*}\displaystyle
     & \int_{(x^1,x^2)\in (-\delta,\delta)^{k_1+k_2}} (x^1)^{\beta_1} (x^2)^{\beta_2} \sup_{Q_{V,(x^1,x^2)}} |(\nabla^{j_1} h_1) (\nabla^{j_2} h_2)| dx^1 dx^2\\
=    & \int_{x^2\in (-\delta,\delta)^{k_2}} (x^2)^{\beta_2} \left( \int_{x^1\in (-\delta,\delta)^{k_1}} (x^1)^{\beta_1}\sup_{Q_{V,(x^1,x^2)}} |(\nabla^{j_1} h_1) (\nabla^{j_2} h_2)| dx^1 \right) dx^2 \\
\leq & \int_{x^2} (x^2)^{\beta_2} \sup_{(P_2)_{V,x^2}}|(\nabla^{j_2} h_2)|  \left( \int_{x^1}  (x^1)^{\beta_1}\sup_{Q_{V,(x^1,x^2)}} |(\nabla^{j_1} h_1)| dx^1 \right) dx^2 \\
\leq & D_{j_1,V,(x^1)^{\beta_1}} \hp^{-\frac{j_1+s-|\beta_1|-k_1}{2}} \int_{x^2} (x^2)^{\beta_2} \sup_{(P_2)_{V,x^2}}|(\nabla^{j_2} h_2)| dx^2 \\
\leq & D_{j_1,V,(x^1)^{\beta_1}} D_{j_2,V,(x^2)^{\beta_2}} \hp^{-\frac{j_1+s-|\beta_1|-k_1}{2}} \hp^{-\frac{j_2+r-|\beta_2|-k_2}{2}}
\leq  D_{j,V,x^{\beta}} \hp^{-\frac{j+s+r-|\beta|-k}{2}},
\end{align*}
which gives the desired estimate in condition (3) of Definition \ref{asypmtotic_support_def}.
\end{proof}

For a given closed tropical polyhedral subset $P \subset U$, we choose a reference tropical hyperplane $R \subset U$ which divides the base $U$ as
$U\setminus R = U_+ \cup U_-$
such that $P \subset U_+$, together with an affine vector field $v$ (meaning $\nabla v = 0$) not tangent to $R$ pointing into $U_+$.
We let
\begin{equation}\label{defining_IP}
I(P) := (P + \real_{\geq 0}v) \cap U
\end{equation}
be the image swept out by $P$ under the flow of $v$.

By shrinking $U$ if necessary, we can assume that for any point $p \in U$, the unique flow line of $v$ in $U$ passing through $p$ intersects $R$ uniquely at a point $x \in R$.
Then the time-$t$ flow along $v$ defines a diffeomorphism
$\tau : W \rightarrow U,\ (t, x) \mapsto \tau(t,x),$
where $W \subset \real \times R$ is the maximal domain of definition of $\tau$ (namely, for any $x \in R$, there is a maximal time interval $I_x \subset \real$ so that the flow line through $x$ has its image lying inside $U$).
For any point $x \in R$, we denote by $\tau_x(t) := \tau(t,x)$ the flow line of $v$ passing through $x$. Figure \ref{fig:integral_flow} illustrates the situation.

\begin{figure}[h]
\centering
\includegraphics[scale=0.45]{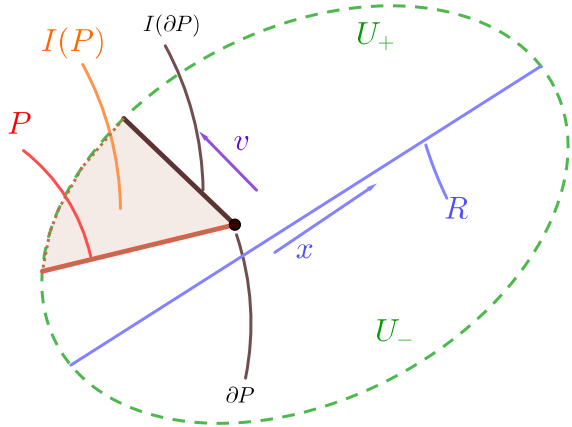}
\caption{The flow along $v$ and $I(P)$}
\label{fig:integral_flow}
\end{figure}

We now define an integral operator $I$ as
\begin{equation}\label{general_integral_operator}
I(\alpha)(t,x) := \int_{0}^t \iota_{\dd{s}}(\tau^*(\alpha))(s,x) ds.
\end{equation}
Note that $I$ depends on the choice of the tropical hyperplane $R$.

\begin{lemma}\label{integral_lemma}
For $\alpha \in \filt^s_P(U)$, we have $I(\alpha) \in \filt^{-\infty}_{k-1}(U)$ if $v$ is tangent to $P$, and $I(\alpha) \in \filt^{s-1}_{I(P)}(U)$ if $v$ is not tangent to $P$, where $I(P)$ is defined in \eqref{defining_IP}.
\end{lemma}

\begin{proof}
In order to simplify notations in this proof, we will omit $\tau^*$ in the definition \eqref{general_integral_operator} of $I$ by treating $\tau:W \to U$ as an affine coordinate chart.

Suppose that $v$ is tangent to $P$. By condition (2) of Definition \ref{asypmtotic_support_def}, we have a neighborhood $W_P \subset U$ such that $\alpha = h \nu_P + \eta$. For each point $x \in R$, the path $\tau_x(t)$ is tangent to the foliation $\{P_q\}_{q\in N}$ in $W_P$ whenever $\tau_x(t) \in W_P$ by the tangency assumption. This means $\iota_{\dd{t}}(\nu_P) = 0$ in $\tau_x^{-1}(W_P)$ and hence we have
\begin{align*}
I(\alpha)(t,x) = \int_{[0,t]} \iota_{\dd{s}}\alpha(s,x) ds = \int_{[0,t] \cap \tau_x^{-1}(U \setminus W_P)} \iota_{\dd{s}}\alpha(s,x) ds + \int_{[0,t] \cap \tau_x^{-1}(W_P)} \iota_{\dd{s}}\eta(s,x) ds.
\end{align*}
So we have $I(\alpha) \in \filt^{-\infty}_{k-1}(U)$ by conditions (1) and (2) of Definition \ref{asypmtotic_support_def}.

Now suppose that $v$ is not tangent to $P$. Let
$I(W_P) := \bigcup_{t \geq 0 } (W_P + t \cdot v) \cap U,$
which gives an open neighborhood of $I(P)$.
Concerning condition (1) in Definition \ref{asypmtotic_support_def}, we take $\tau(t_0,x_0) \in U \setminus I(P)$, and then a neighborhood $V$ of $\tau(t_0,x_0)$ in $U \setminus I(P)$ and a neighborhood $W_P' \subset W_P$ of $P$, such that, for any point $\tau(t,x) \in V$, the flow line joining $\tau(t,x)$ to $R$ does not hit $\overline{W_P'}$. This implies that $I(\alpha)|_{V} \in \filt^{-\infty}_{k-1}(V)$ since we have $\alpha|_{U \setminus \overline{W_P'}} \in \filt^{-\infty}_k(U \setminus \overline{W_P'})$ and
$$
I(\alpha)(t,x) = \int_{0}^t \iota_{\dd{s}}\alpha(s,x) ds = \int_{[0,t] \cap \tau_x^{-1}(U \setminus \overline{W_P'})} \iota_{\dd{s}}\alpha(s,x) ds.
$$
So condition (1) in Definition \ref{asypmtotic_support_def} holds for $I(\alpha)$.

Concerning condition (2) in Definition \ref{asypmtotic_support_def}, we first note that $v = \dd{t}$ is tangent to $I(P)$, so by parallel transporting the form $\iota_{\dd{t}} \nu_P$ to the neighborhood $I(W_P)$, we obtain a volume element in the normal bundle of $I(P)$, which we denote by $\nu_{I(P)}$. For a point $q \in I(W_P)$, we take a small neighborhood $V$ near $q$, and for $\tau(t,x) \in V$, we write
\begin{align*}
I(\alpha)(t,x) = \int_{0}^t \iota_{\dd{s}}\alpha(s,x) ds
=  \int_{[0,t] \cap \tau_x^{-1}(U \setminus W_P)} \iota_{\dd{s}}\alpha(s,x) ds + \int_{[0,t] \cap \tau_x^{-1}(W_P)} \iota_{\dd{s}}(h \nu_P + \eta)(s,x)ds\\
=  \left( \int_{[0,t] \cap \tau_x^{-1}(W_P)} h(s,x) ds \right) \nu_{I(P)} + \int_{[0,t] \cap \tau_x^{-1}(W_P)} \iota_{\dd{s}}\eta(s,x)ds
    + \int_{[0,t] \cap \tau_x^{-1}(U \setminus W_P)} \iota_{\dd{s}}\alpha(s,x)ds,
\end{align*}
where the last two terms are in $\filt^{-\infty}_{k-1}(V)$, and condition (2) in Definition \ref{asypmtotic_support_def} holds for $I(\alpha)$..

Concerning condition (3) in Definition \ref{asypmtotic_support_def}, we fix a point $p = \tau(b,x) \in I(P)$ and let $p' = \tau(a,x) \in P$ be the unique point such that $p ,p'$ lie on the same flow line $\tau_x$. We take local affine coordinates $x = (x_1, \dots, x_{k-1}, x_k, \dots, x_{n-1}) \in (-\delta,\delta)^{n-1}$ of $R$ centered at $p'$ (meaning that $p' = (a,0)$) such that $x' = (x_1, \dots, x_{k-1})$ are normal to the tropical polyhedral subset $p_R(\tau^{-1}(P))\subset R$, where $p_R: W (\subset \real \times R) \rightarrow R$ is the natural projection.

By taking $\delta$ small enough, we have $\tau : (a-\delta,b+\delta) \times (-\delta,\delta)^{n-1} \rightarrow U$ mapping diffeomorphically onto its image, such that it contains the part the flow line $\tau_0|_{[a,b]}$ joining $p'$ to $p$. We can also take $V = \tau((b-\delta,b+\delta)\times (-\delta,\delta)^{n-1})$ with $\tau(b,0)=p$, and arrange that $V' = \tau((a-\delta,a+\delta)\times (-\delta,\delta)^{n-1}) \subset W_P$ with $\tau(a,0)=p'$.  Notice that there is a possibility that $p = p' \in P$ and therefore $a=b$ in the above description which means $V = V'$. Figure \ref{fig:integral_lemma_pic} illustrates the situation.

\begin{figure}[h]
\centering
\includegraphics[scale=0.5]{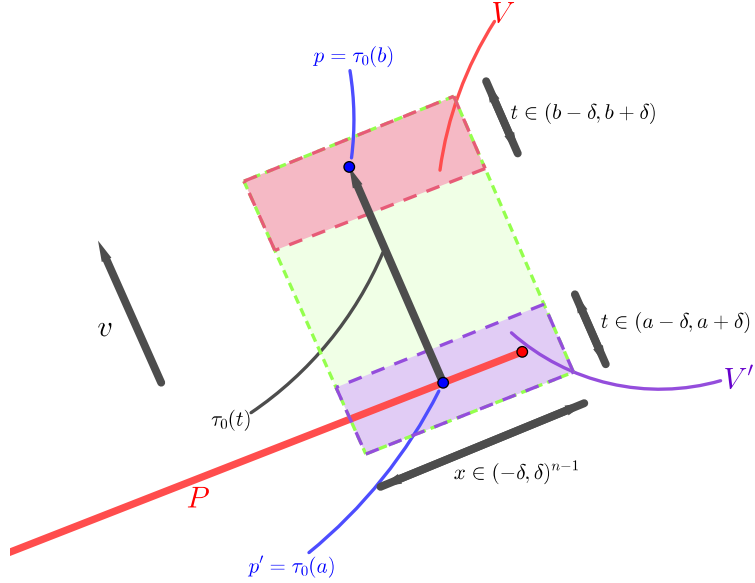}
\caption{Neighborhood along the flow line $\tau_0(t)$}
\label{fig:integral_lemma_pic}
\end{figure}

Recall that there is a foliation $\{P_{q}\}_{q\in N}$ codimension $k$ affine subspaces parallel to $P$. Then the induced foliation $\{P_{t,x'}\}_{(t,x')\in N_{V'}}$ of the neighborhood $V'$ can be parametrized by $N_{V'} := (a-\delta, a+\delta) \times (-\delta,\delta)^{k-1}$. Therefore the foliation of $V$ induced by $I(P)$ is parametrized as $\{I(P)_{x'}\}_{x'\in N_V}$, where
$I(P)_{x'} = \bigcup_{t \in (b-\delta,b+\delta)} (P_{0,x'} +  tv)$
and $N_V = (-\delta,\delta)^{k-1}$.

For $\alpha \in \filt^s_P$, we consider $I(\alpha) = \int_{0}^t \iota_{\dd{s}}\alpha(s,x) ds $ in the neighborhood $V$, and what we need to estimate is the term
$\int_{N_V}(x')^\beta \sup_{I(P)_{x'}}|\nabla^j \iota_{\nu_{I(P)}^\vee}I(\alpha)|\nu_{I(P)}.
$
The integral $I(\alpha)$ can be split into two parts as $\int_0^t = \int_{0}^{a-\delta} + \int_{a-\delta}^t$ and we only have to control the second part
$
I_{a-\delta}(\alpha) := \int_{a-\delta}^{t} \iota_{\dd{t}}\alpha(s,x) ds
$
because $\alpha \in \filt^s_P(U)$ satisfies condition (1) in Definition \ref{asypmtotic_support_def}, so $\int_{0}^{a-\delta} \iota_{\dd{s}}\alpha(s,x) ds$, as a function of $(t,x)$ which is constant in $t$, lies in $\filt^{-\infty}_{k-1}(U)$ (as the integral misses the support $P$ of $\alpha$). Writing $\nabla^j = \nabla_{\perp}^{j_1} \nabla_{\dd{t}}^{j_2}$, where $\nabla_{\perp} (t) = 0$, we have two cases depending on whether $j_2 = 0$ or $j_2 > 0$:

Case 1: $j_2 =0$. Then we have
$$
|\nabla_{\perp}^j \iota_{\nu_{I(P)}^\vee}(I_{a-\delta}(\alpha))| \leq \int_{a - \delta}^{a + \delta} |\nabla_{\perp}^j(\iota_{\nu_{P}^\vee}\alpha)| ds + \int_{a + \delta}^{b + \delta}|\nabla_{\perp}^j(\iota_{\nu_{P}^\vee}\alpha)| ds.
$$
The latter term can be dropped because the domain $\int_{a + \delta}^{b + \delta}$ misses the support of $P$, so it lies in $\filt^{-\infty}_{k-1}$.
For the first term, we treat $\int_{a-\delta}^{a+\delta} |\nabla_{\perp}^j(\iota_{\nu_{P}^\vee}\alpha)| ds$ as a function of $(t,x)$ on $V$ which is constant along the $t$-direction. Therefore we estimate
\begin{align*}
    & \int_{x'} (x')^\beta \sup_{I(P)_{x'}} \left( \int_{a-\delta}^{a+\delta} |\nabla_{\perp}^j(\iota_{\nu_{P}^\vee}\alpha)| ds \right) \nu_{I(P)}
=     \int_{x'} (x')^\beta  \sup_{P_{0,x'}+bv} \left( \int_{a-\delta}^{a+\delta} |\nabla_{\perp}^j(\iota_{\nu_{P}^\vee}\alpha)| ds \right) \nu_{I(P)}\\
\leq&  \int_{x'} \sup_{P_{0,x'}+bv}\left( \int_{a-\delta}^{a+\delta} (x')^\beta \sup_{P_{s,x'}}|\nabla_{\perp}^j(\iota_{\nu_{P}^\vee}\alpha)| ds \right) \nu_{I(P)}
=     \int_{x'} \left( \int_{a-\delta}^{a+\delta} (x')^\beta \sup_{P_{s,x'}}|\nabla_{\perp}^j(\iota_{\nu_{P}^\vee}\alpha)| ds \right) \nu_{I(P)}\\
\leq&  C_{j,V',\beta} \hp^{-\frac{j+s-|\beta|-k}{2}},
\end{align*}
where the first inequality follows from the inequality
$$
\int_{a-\delta}^{a+\delta}|\nabla_{\perp}^j(\iota_{\nu_{P}^\vee}\alpha)|ds \leq \int_{a-\delta}^{a+\delta}\sup_{P_{t,x'}}|\nabla_{\perp}^j(\iota_{\nu_{P}^\vee}\alpha)|ds,
$$
and the second equality is due to the fact that $\int_{a-\delta}^{a+\delta}\sup_{P_{t,x'}}|\nabla_{\perp}^j(\iota_{\nu_{P}^\vee}\alpha)|ds$, treated as function on $V$, is constant along the leaf $P_{0,x'} + bv$.
Writing $j + s - |\beta| - k = j + (s-1) - |\beta| - (k-1)$, we obtain the desired estimate so that $\alpha \in \filt^{s-1}_{I(P)}(U)$.

Case 2: $j_2>0$. Then we have $\nabla_{\dd{t}}^{j_2} \iota_{\nu_{I(P)}^\vee}(I_{a-\delta}(\alpha)) =  \nabla_{\dd{t}}^{j_2-1} (\iota_{\nu_{P}^\vee} \alpha)$.
We can rewrite it as
$$
\nabla_\perp^{j_1}\nabla_{\dd{t}}^{j_2} \iota_{\nu_{I(P)}^\vee}(I_{a-\delta}(\alpha))(t,x)
= \int^t_{a-\delta} \nabla^{j}(\iota_{\nu_{P}^\vee} \alpha)(s,x) ds + \left(\nabla_{\dd{t}}^{j_2-1}\nabla_\perp^{j_1}(\iota_{\nu_{P}^\vee} \alpha)\right)(a-\delta,x),
$$
where the latter term lies in $\filt^{-\infty}_{k-1}$ because it misses the support $P$ of $\alpha$, and
the first term is bounded by
$$
|\int^t_{a-\delta} \nabla^{j}(\iota_{\nu_{P}^\vee} \alpha)(s,x) ds| \leq \int^{a+\delta}_{a-\delta} |\nabla^{j}(\iota_{\nu_{P}^\vee} \alpha)|(s,x) ds + \int^{b+\delta}_{a+\delta} |\nabla^{j}(\iota_{\nu_{P}^\vee} \alpha)|(s,x) ds.
$$
The same argument as Case 1 can then be applied to get the desired estimate.
\end{proof}

\begin{remark}
Lemmas \ref{support_product} and \ref{integral_lemma} say that we can relate the differential-geometric operations $\wedge$ and $I$ to intersection and suspension of asymptotic supports. These properties are essential for relating Maurer-Cartan solutions, which are differential-geometric in nature, to combinatorics of scattering diagrams.
\end{remark}

In order to apply the notion of asymptotic support to keep track of the $\hp$ order in asymptotic expansions of the gauge element $\facs^s = \facs_1  + \facs_2 + \cdots + \facs_s$, we will restrict our attention to the dg Lie subalgebra
$\left(\bigoplus_m \filt^\infty_*(U)\bmc^m\right) \otimes_\inte N \subset \mathbf{G}^*(U),$
whose elements are finite sums of the form
$\sum_{m, n} \alpha_m^n \bmc^m \check{\partial}_n,$
where $\alpha_m^n \in \filt^{\infty}_*(U)$. Restriction of $\hat{H}$ defined in Definition \ref{real2homotopy} to $\left(\bigoplus_m \filt^\infty_*(U)\bmc^m\right) \otimes_\inte N$ gives the homotopy operator
$\hat{H}: \left(\bigoplus_m \filt^\infty_*(U)\bmc^m\right) \otimes_\inte N \rightarrow \left(\bigoplus_m \filt^\infty_{*-1}(U)\bmc^m\right) \otimes_\inte N$ defined as
$
\hat{H}\left(\sum_{m, n} \alpha_m^n \bmc^m \check{\partial}_n \right) = \sum_{m, n} \int_{0}^1 \rho_q^*(\alpha_m^n) \bmc^m \check{\partial}_n,
$
using $\rho_q$ in Definition \ref{real2homotopy}. We also write $\hat{I}\left(\sum_{m, n} \alpha_m^n \bmc^m \check{\partial}_n\right) = \sum_{m, n} \hat{I}(\alpha_m^n) \bmc^m \check{\partial}_n$ when the $\alpha_m^n$'s are $1$-forms. Extending Lemma \ref{support_product} to this dg Lie subalgebra, we have the following:

\begin{lemma}\label{filtrationlemma}
Given $m_1, m_2 \in M$, $n_1 ,n_2 \in N$, and $\alpha_1 \in \filt^s_{P_1}(U)$ and $\alpha_2 \in \filt^r_{P_2}(U)$. If $P_1$ of codimension $k_1$ intersects $P_2$ of codimension $k_2$ transversally, then we have
$$
[\alpha_1 \bmc^{m_1} \check{\partial}_{n_1}, \alpha_2 \bmc^{m_2} \check{\partial}_{n_2} ] \in \alpha_1\wedge \alpha_2 \bmc^{m_1+m_2} \check{\partial}_{( m_2,n_1 ) n_2 - ( m_1 ,n_2 ) n_1} + \filt^{r+s-1}_{P}(U) \bmc^{m_1+m_2} \otimes_\inte N
$$
and $\alpha_1 \wedge \alpha_2 \in \filt^{r+s}_{P}(U)$ for any codimension $k_1+k_2$ polyhedral subset $P$ containing $P_1 \cap P_2$ normal to $\nu_{P_1} \wedge \nu_{P_2}$. If the intersection is not transversal, then we have
$
[\alpha_1 \bmc^{m_1} \check{\partial}_{n_1}, \alpha_2 \bmc^{m_2} \check{\partial}_{n_2} ] \in \filt^{-\infty}_k(U) \bmc^{m_1+m_2} \otimes_\inte N.
$
\end{lemma}

\begin{proof}
From the definition of the Lie bracket we have
\begin{align*}
[ \alpha_1 \bmc^{m_1} \check{\partial}_{n_1}, \alpha_2 \bmc^{m_2} \check{\partial}_{n_2} ] = 
& \alpha_1\wedge \alpha_2 \bmc^{m_1+m_2} \check{\partial}_{( m_2,n_1 ) n_2 - ( m_1 ,n_2 ) n_1}\\
& \quad\quad + \alpha_1 \wedge \nabla_{\partial_1}(\alpha_2) \bmc^{m_1+m_2} \check{\partial}_{n_2} \pm \alpha_2 \wedge \nabla_{\partial_2}(\alpha_1) \bmc^{m_1+m_2} \check{\partial}_{n_1} 
\end{align*}
When $P_1$ and $P_2$ are intersecting transversally and let $P$ as above, Lemma \ref{support_product} says that $\alpha_1 \wedge \alpha_2 \in \filt^{r+s}_{P}(U)$, so it remains to show that the last two terms are lying in $\filt^{r+s-1}_{P}(U) \bmc^{m_1+m_2} \otimes_\inte \Lambda^\vee_{B_0}$. Notice that we have $\nabla_{\partial_{n_1}}(\alpha_1) \in \filt^{s-1}_{P_1}(U)$ and $\nabla_{\partial_{n_1}}(\alpha_2) \in \filt^{r-1}_{P_2}(U)$ and hence result follows by applying Lemma \ref{support_product} again. When $P_1$ and $P_2$ are not intersecting transversally, it follows from the non-transversal case of Lemma \ref{support_product} that all the terms lie in $\filt^{-\infty}_k(U) \bmc^{m_1+m_2} \otimes_\inte N$.
\end{proof}

\begin{remark}
The terms $\bmc^{m_1+m_2} \check{\partial}_{( m_2,n_1 ) n_2 - ( m_1 ,n_2 ) n_1}$ which appear in Lemma \ref{filtrationlemma} come from $[\bmc^{m_1} \check{\partial}_{n_1}, \bmc^{m_2} \check{\partial}_{n_2} ]$ using the formula \ref{vertex_lie_algebra}. In particular, if we have both $(m_1,n_1) = 0$ and $(m_2,n_2) = 0$ (which means that both $\bmc^{m_1} \check{\partial}_{n_1}$ and $\bmc^{m_2} \check{\partial}_{n_2}$ are elements in the tropical vertex group), then the leading order term of $[\alpha_1 \bmc^{m_1} \check{\partial}_{n_1}, \alpha_2 \bmc^{m_2} \check{\partial}_{n_2} ]$ is given by $\alpha_1 \wedge \alpha_2 \bmc^{m_1+m_2} \check{\partial}_{( m_2,n_1 ) n_2 - ( m_1 ,n_2 ) n_1}$, and $\bmc^{m_1+m_2} \check{\partial}_{( m_2,n_1 ) n_2 - ( m_1 ,n_2 ) n_1}$ is an element in the tropical vertex group as well. This property will be important to us in Section \ref{sec:key_lemmas}.
\end{remark}

At this point we are ready to go back to the asymptotic analysis of the gauge $\facs$. Recall that there are two integral operators: $\hat{I}$ defined in \eqref{I_integral} and $I$ defined in \eqref{general_integral_operator}. If we restrict ourselves to differential 1-forms, we can treat both $\hat{I}$ and $I$ as path integrals, where the choices of paths differ only by a path lying inside $R$, as shown in Figure \ref{fig:two_integrals}. (Indeed, the requirement that, for any point $p \in U$, the unique flow line of $v$ in $U$ passing through $p$ intersects $R$ uniquely at a point $x \in R$ when we define $I$ is equivalent to the condition that $U$ contains the path $\varrho_u$ when we define $\hat{I}$.)
\begin{figure}[h]
\centering
\includegraphics[scale=0.3]{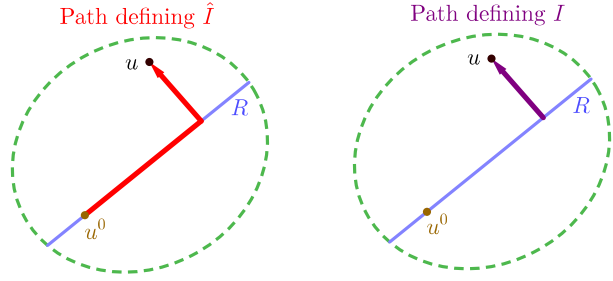}
\caption{The difference between $I$ and $\hat{I}$}
\label{fig:two_integrals}
\end{figure}

The key observation is that Lemma \ref{integral_lemma}, which applies to the operator $I$, can be applied to $\hat{I}$ as well because $R$ is chosen so that $R \cap P = \emptyset$, and hence integration of terms with asymptotic support on $R$ over any path in $R$ will produce elements in $\filt^{-\infty}_*(U)$.

We show by induction that the term
$\hat{I} \left( \sum_{k\geq 0 } \frac{ad^k_{\facs^s}}{(k+1)!} \pdb \facs^s \right)_{s+1}$
does not contribute to the leading $\hp$ order term in $\facs_{s+1}$ defined in \eqref{vartheta_s_definition}. For that we take $P$ to be codimension $1$ hyperplane in $U$.

\begin{lemma}\label{leadingorderlemma}
For the gauge $\facs = \facs_1  + \facs_2  + \cdots$ defined iteratively by \eqref{vartheta_s_definition}, we have
$$\facs_s  \in \bigoplus_{k\geq 1} \filt^{0}_{I(P)}(U) \bmc^{km} \check{\partial}_n t^s; \quad ad^l_{\facs^s}(\pdb \facs^s)  \in \bigoplus_{\substack{k \geq 1\\ 1 \leq j \leq s(l+1)}} \filt^{0}_P(U) \bmc^{km}\check{\partial}_n t^j$$
for all $s \geq 1$ and $l \geq 1$, where $\facs^s = \facs_1  + \facs_2 + \cdots + \facs_s$. 
\end{lemma}

\begin{proof}
We prove by induction on $s$. The $s = 1$ case concerns the term $\facs_1 = -\hat{I}(\onewall_1)$, and we have $ad_{\facs_1}^l (\pdb \facs_1) = -ad_{\facs_1}^l (\onewall_1)$.
Now $\onewall_1 \in \bigoplus_{k \geq 1} \filt^1_{P}(U) \bmc^{km}\check{\partial}_n t^1$ from Definition \ref{ansatz}, so we have
$\facs_1 = -\hat{I}(\onewall_1) \in \bigoplus_{k\geq 1} \filt^{0}_{I(P)}(U) \bmc^{km}\check{\partial}_n t^1$
by Lemma \ref{integral_lemma}.
Applying Lemma \ref{filtrationlemma} $l$ times, together with the fact that $I(P)$ and $P$ intersect transversally, we have
$
-ad_{\facs_1}^l (\onewall_1) \in \bigoplus_{\substack{k \geq 1 \\ 1 \leq j \leq l+1}} \filt^{0}_P(U)  \bmc^{km}\check{\partial}_n t^j.
$
The key here is that all the terms are linear combinations of $\bmc^{km} \check{\partial}_n$'s, between which the Lie bracket vanish since $m$ is tangent to $P$ and $n$ is normal to $P$, and hence the leading contribution in Lemma \ref{filtrationlemma} vanishes.

Now we assume that the statement is true for all $s' \leq s$. The induction hypothesis together with the fact that $\onewall_{s+1} \in \bigoplus_{k \geq 1} \filt^{1}_P(U)  \bmc^{km}\check{\partial}_n t^{s+1}$ imply that
$$
\pdb \facs_{s+1}  = -\left( \onewall + \sum_{l\geq 0 }\frac{ad_{\facs^s}^l}{(l+1)!} \pdb\facs^s \right)_{s+1} \in \bigoplus_{k\geq 1} \filt^{1}_P(U) \bmc^{km}\check{\partial}_nt^{s+1}.
$$
Applying Lemma \ref{integral_lemma} to $\facs_{s+1} = - \hat{I}\left( \onewall + \sum_{l\geq 0 }\frac{ad_{\facs^s}^l}{(l+1)!} \pdb\facs^s \right)_{s+1}$ then gives
$
\facs_{s+1} \in \bigoplus_{k\geq 1} \filt^0_{I(P)}  \bmc^{km}\check{\partial}_n t^{s+1}.
$
The second statement follows by applying Lemma \ref{filtrationlemma} multiple times with the same reasoning as above. This completes the proof.
\end{proof}

By Lemma \ref{leadingorderlemma}, we have
\begin{equation}
\facs_s \in - \hat{I}(\onewall_s) + \bigoplus_{l\geq 1} \hat{I} (\filt^0_P)(U) \bmc^{lm}\check{\partial}_n t^s
\end{equation}
for all $s$. Lemma \ref{integral_lemma} tells us that $\hat{I}(\filt^0_P(U)) \subset \filt^{-1}_{I(P)}(U)$, so $-\hat{I}(\onewall_s) \in \filt^{0}_{I(P)}(U)$ is the only term which contributes to the leading order in $\hp$. Since $I(P)$ is of codimension $0$, $\filt^{-1}_{I(P)}(U) \subset O_{loc}(\hp^{1/2})$ (where $O_{loc}(\hp^{1/2})$ is defined in Notations \ref{O_loc}). We conclude that:

\begin{prop}\label{prop:MC_sol_one_wall}
For the gauge $\facs = \facs_1 + \facs_2 + \cdots$ defined iteratively by \eqref{vartheta_s_definition}, we have
\begin{equation*}
\facs_s
\in \left\{
\begin{array}{ll}
\displaystyle \sum_{k \geq 1 } a_{sk} \bmc^{km} \check{\partial}_n t^s + \bigoplus_{k\geq 1} \filt^{-1}_{I(P)}(U) \bmc^{km} \check{\partial}_n t^s & \text{on $\check{p}^{-1}(H_+)$},\\
\displaystyle \bigoplus_{k\geq 1} \filt^{-\infty}_0(U) \bmc^{km} \check{\partial}_n t^s & \text{on $\check{p}^{-1}(H_-)$};
\end{array}\right.
\end{equation*}
which implies that
$
\facs = \psi + \sum_{j, k\geq 1} O_{loc}(\hp^{1/2}) \bmc^{km} \check{\partial}_n t^j
$
over $\check{p}^{-1}(U \setminus P)$, or equivalently,
$
\check{\facs} := \mathcal{F}^{-1}(\facs) = \check{\psi} + \sum_{j, k\geq 1} O_{loc}(\hp^{1/2}) \bmc^{km}\check{\partial}_n t^j
$
over $\check{X}_0 \setminus \check{p}^{-1}(P)$.
\end{prop}

\begin{remark}\label{general_input}
Recall that the ansatz in Definition \ref{ansatz} is defined by multiplying $\delta_m$ to the wall crossing factor $\text{Log}(\Theta)$. But indeed the only properties that we need are $\delta_m \in \filt^1_P(U)$, and $\hat{I}(\delta_m)$ has its leading $\hp$ order term given by $1$ on $H_+$.

So Proposition \ref{prop:MC_sol_one_wall} still holds for any solution to the Maurer-Cartan equation in Proposition \ref{A_mcequation} (or more generally, to the Maurer-Cartan equation of the quotient dgLa $\widehat{\mathbf{g}^*/\mathcal{E}^*}(U)$ to be introduced in Section \ref{sec:MC_error}) of the form
$
\onewall \in - \sum_{j, k \geq 1} (a_{jk} \delta_{jk} + \filt^0_{P}(U)) \bmc^{km}\partial_{n} t^{j}
$
with $\pdb\onewall = 0$ such that each $\delta_{jk}^{(i)} \in \filt^1_{P_i}(U)$ and can be written as $\delta_{jk}^{(i)} = (\pi\hp)^{-1/2} e^{-\frac{x^2}{\hp}} dx$ in some neighborhood $W_{P_i}$ of $P_i$,
where $x$ is some affine linear function on $W_{P_i}$ such that $P_i$ is defined by $x = 0$ locally and $\iota_{\nu_{P_i}} dx > 0$.
\end{remark} 
\section{Maurer-Cartan solutions and scattering diagrams}\label{twowalls}
In this section, we interpret the local scattering process, which produces a consistent extension $\mathcal{S}(\mathscr{D})$ of a scattering diagram $\mathscr{D}$ consisting of two non-parallel walls, as arising from semiclassical limits (as $\hp \rightarrow 0$) of a solution of the Maurer-Cartan (MC) equation.

\subsection{Solving Maurer-Cartan equations in general}\label{algebraicsolveMC}
Let us begin by reviewing the process of solving MC equations in a general dgLa $(\mathbf{G}^*, \pdb, [\cdot, \cdot])$. We will apply Kuranishi's method \cite{Kuranishi65} to solve the MC equation using a homotopy which retracts $\mathbf{G}^*$ to its cohomology and acts as the gauge fixing (see e.g. \cite{Morrow-Kodaira_book}). 

Suppose that we are given an input
$
\incoming = \incoming_1 + \incoming_2 + \cdots \in \mathbf{G}^1 \otimes \mathbf{m}
$
satisfying $\pdb \incoming = 0$, where $\incoming_k \in \mathbf{G}^1 \otimes \mathbf{m}^k $ is homogeneous of degree $k$ in $t$.
We attempt to find
$
\varXi = \varXi_2 + \varXi_3 + \cdots \in \mathbf{G}^1  \otimes \mathbf{m},
$
where $\varXi_k \in \mathbf{G}^1 \otimes \mathbf{m}^k $ is homogeneous of degree $k$ in $t$,
such that
$
\Phi := \Phi_1 + \Phi_2 + \cdots \in \mathbf{G}^1  \otimes \mathbf{m},
$
where each term $\Phi_k := \incoming_k + \varXi_k \in \mathbf{G}^1 \otimes \mathbf{m}^k$ is homogeneous of degree $k$ in $t$, gives a solution of the following MC equation, i.e.
\begin{equation}\label{MCequation}
\pdb \Phi + \half[\Phi,\Phi] = 0.
\end{equation}

We assume that there are chain maps $\iota$, $\mathcal{P}$ and homotopy $H$
\begin{equation}\label{homotopy_diagram}
\xymatrix@1{H^*(G^*) \ar@/^/[rr]^{\iota} && G^* \ar@/^/[ll]^{\mathcal{P}} \ar@(ur,dr)^{H}
},
\end{equation}
such that $\mathcal{P} \circ \iota = \text{Id}$, and $\text{Id} - \iota \circ \mathcal{P} = \pdb H + H \pdb$. Then, instead of the MC equation, we look for solutions $\Phi$ of the equation
\begin{equation}\label{pseudoMC}
\Phi= \incoming - \half H [\Phi,\Phi].
\end{equation}
This originates from a method of Kuranishi \cite{Kuranishi65} used to solve the MC equation of the Kodaira-Spencer dgLa. His method can be generalized to a general dgLa as follows (see e.g. \cite{manetti2005differential})
\begin{prop}\label{pseudoMC_to_MC}
Suppose that $\Phi$ satisfies the equation \eqref{pseudoMC}. Then $\Phi$ satisfies the MC equation \eqref{MCequation} if and only if $\mathcal{P}[\Phi,\Phi] = 0$.
\end{prop}


In general, the $k$-th equation of the above equation \eqref{pseudoMC} is given by
\begin{equation}\label{MCiterationequation}
\varXi_k + \sum_{j+l = k} \half H [\Phi_j, \Phi_l] = 0,
\end{equation}
and $\varXi_k$ (recall that $\varXi = \Phi - \incoming$) is uniquely determined by $\varXi_j$, $j < k$. In this way, the solution $\varXi$ to \eqref{pseudoMC} is uniquely determined.



There is a beautiful way to express the unique solution $\varXi$ as a sum of terms involving the input $\incoming$ over directed trees (reminiscent of a Feynman sum). To this end, we will introduce the notions of {\em a directed tree} and {\em a directed tree with ribbon structure}, following \cite{fukaya2003deformation}.

\begin{definition}\label{k_tree_def}
A {\em (directed) $k$-tree} $\tr$ consists of the following data:
\begin{itemize}
\item
a finite set of vertices $\bar{\tr}^{[0]}$ together with a decomposition
$\bar{\tr}^{[0]} = \tr^{[0]}_{in} \sqcup \tr^{[0]} \sqcup \{v_o\},$
where $\tr^{[0]}_{in}$, called the set of incoming vertices, is a set of size $k$ and $v_o$ is called the outgoing vertex (we also write $\tr^{[0]}_\infty := \tr^{[0]}_{in} \sqcup \{v_o\}$),
\item
a finite set of edges $\bar{\tr}^{[1]}$, and
\item
two boundary maps $\partial_{in} , \partial_o : \bar{\tr}^{[1]} \rightarrow \bar{\tr}^{[0]}$ (here $\partial_{in}$ stands for incoming and $\partial_o$ stands for outgoing)
\end{itemize}
satisfying all of the following conditions:
\begin{enumerate}
\item
Every vertex $v \in \tr^{[0]}$ is trivalent, and satisfies $\# \partial_{o}^{-1}(v) = 2$ and $\# \partial_{in}^{-1}(v) = 1$.

\item
Every vertex $v \in \tr^{[0]}_{in}$ has valency one, and satisfies $\# \partial_{o}^{-1}(v) = 0$ and $\# \partial_{in}^{-1}(v) = 1$; we let $\tr^{[1]} := \bar{\tr}^{[1]}\setminus \partial_{in}^{-1}(\tr^{[0]}_{in})$.

\item
For the outgoing vertex $v_o$, we have $\# \partial_{o}^{-1}(v_o) = 1$ and $\# \partial_{in}^{-1}(v_o) = 0$; we let $e_o := \partial_o^{-1}(v_o)$ be the outgoing edge and denote by $v_r \in \tr^{[0]}_{in} \sqcup \tr^{[0]}$ the unique vertex (which we call the root vertex) with $e_o = \partial^{-1}_{in}(v_r)$.

\item
The topological realization
$|\bar{\tr}| := \left( \coprod_{e \in \bar{\tr}^{[1]}} [0,1] \right) / \sim$
of the tree $\tr$ is connected and simply connected; here $\sim$ is the equivalence relation defined by identifying boundary points of edges if their images in $\tr^{[0]}$ are the same.
\end{enumerate}

Two $k$-trees $\tr_1$ and $\tr_2$ are {\em isomorphic} if there are bijections $\bar{\tr}^{[0]}_1 \cong \bar{\tr}^{[0]}_2$ and $\bar{\tr}^{[1]}_1 \cong \bar{\tr}^{[1]}_2$ preserving the decomposition $\bar{\tr}^{[0]}_i = \tr^{[0]}_{i,in} \sqcup \tr^{[0]}_i \sqcup \{v_{i,o}\}$ and boundary maps $\partial_{i,in}$ and $\partial_{i,o}$. The set of isomorphism classes of $k$-trees will be denoted by $\tree{k}$. For a $k$-tree $\tr$, we will abuse notations and use $\tr$ (instead of $[\tr]$) to denote its isomorphism class.
\end{definition}

\begin{definition}\label{def:ribbon_k_tree}
A {\em ribbon structure} on a $k$-tree is a cyclic ordering of $\partial_{in}^{-1}(v) \sqcup \partial_o^{-1}(v)$ for each $v \in \tr^{[0]}$. Equivalently, it can be regarded as an embedding $|\tr| \hookrightarrow D$ of $|\tr|$ into the unit disc $D \subset \real^2$ mapping $\tr^{[0]}_\infty$ to $\partial D$, from which the cyclic ordering is induced by the clockwise orientation on $D$. We will use $\rtr$ to denote a ribbon $k$-tree, and $\underline{\rtr}$ to denote the $k$-tree underlying $\rtr$.

Two ribbon $k$ trees $\rtr_1$ and $\rtr_2$ are {\em isomorphic} if they are isomorphic as $k$-trees and the isomorphism preserves the cyclic ordering. The set of isomorphism classes of ribbon $k$-trees will be denoted by $\rtree{k}$.
We will again abuse notations by using $\rtr$ to denote an isomorphism class of ribbon $k$-trees.
\end{definition}

\begin{definition}\label{r_tree_operation}
Given a ribbon $k$-tree $\rtr \in \rtree{k}$, we label the incoming vertices by $v_1,\dots,v_k$ according to its cyclic ordering (or the clockwise orientation on $D$ if we use the embedding $|\tr| \hookrightarrow D$). We define the operator
$
\mathfrak{l}_{k,\rtr}: L[1]^{\otimes k} \rightarrow L[1]
$
by
\begin{enumerate}
\item
aligning the inputs $\zeta_1,\dots,\zeta_k \in L$ at the vertices $v_1,\dots,v_k$ respectively,
\item
applying $m_2$ at each vertex in $\rtr^{[0]}$, where $m_2: L[1] \otimes L[1] \to L[1]$ is the the graded symmetric operator on $L[1]$ (= $L$ shifted by degree $1$) defined by $m_2(\alpha,\beta):= (-1)^{\bar{\alpha}(\bar{\beta}+1)}[\alpha,\beta]$ (here $\bar{\alpha}$ and $\bar{\beta}$ denote degrees of the elements $\alpha, \beta \in L$ respectively), and
\item
applying the homotopy operator $-H$ to each edge in $\rtr^{[1]}$.
\end{enumerate}
We then define $\mathfrak{l}_{k}: L[1]^{\otimes k} \rightarrow L[1]$ by
$\mathfrak{l}_{k} := \sum_{\rtr \in \rtree{k}} \frac{1}{2^{k-1}}\mathfrak{l}_{k,\rtr}.$
\end{definition}

The operation $\mathfrak{l}_{k,\rtr}$ can be symmetrized to give the following operation $\mathfrak{I}_{k,T}$ associated to a $k$-tree $\tr \in \tree{k}$:

\begin{definition}\label{def:tree_operation}
Given a $k$-tree $\tr \in \tree{k}$, let $\rtr \in \rtree{k}$ be a ribbon tree whose underlying tree is $\underline{\rtr} = \tr$. We consider the set $\Sigma_k := \{\sigma \mid \sigma: \tr^{[0]}_{in} \rightarrow \{1,\dots,k\}\}$. Then we define the operator
$
\mathfrak{I}_{k,\tr}: \text{Sym}^k(L[1])\rightarrow L[1]
$
by
\begin{equation*}
\mathfrak{I}_{k,\tr}(\zeta_1,\dots,\zeta_k) := \sum_{\sigma \in \Sigma_k} (-1)^{\chi(\sigma,\vec{\zeta})} \mathfrak{l}_{k,\rtr}(\zeta_{\sigma(1)},\dots,\zeta_{\sigma(k)});
\end{equation*}
here the sign $(-1)^{\chi(\sigma,\vec{\zeta})}$ is determined by the rule that, when the permutation $(\zeta_1,\dots,\zeta_k) \mapsto (\zeta_{\sigma(v_1)},\dots,\zeta_{\sigma(v_k)})$ is decomposed as a product of transpositions, each transposition interchanging $\zeta_i$ and $\zeta_j$ contributes $(-1)^{(\bar{\zeta}_i+1)(\bar{\zeta}_j+1)}$ (where $\bar{\zeta}_i$ denotes the degree of $\zeta_i \in L$).
Note that $\mathfrak{I}_{k,\tr}$ is independent of the choice of the ribbon tree $\rtr$.
We then define $\mathfrak{I}_{k}: \text{Sym}^k(L[1])\rightarrow L[1]$ by
$
\mathfrak{I}_{k} = \sum_{\tr \in \tree{k}} \frac{\mathfrak{I}_{k,\tr}}{|\text{Aut}(\tr)|},
$
where $|\text{Aut}(\tr)|$ is the order of the automorphism group of a $k$-tree $\tr$.
\end{definition}

Setting
\begin{equation}\label{solve_sum_over_trees_2}
\varXi := \sum_{k\geq 2 } \frac{1}{k!}\mathfrak{I}_{k}(\incoming, \ldots, \incoming)= \sum_{k\geq 2} \mathfrak{l}_k(\incoming, \ldots, \incoming),
\end{equation}
then
\begin{equation}\label{solve_sum_over_trees}
\Phi := \incoming + \varXi = \sum_{k\geq 1 } \frac{1}{k!}\mathfrak{I}_{k}(\incoming, \ldots, \incoming) = \sum_{k\geq 1} \mathfrak{l}_k(\incoming, \ldots, \incoming),
\end{equation}
is the unique solution to the equation \eqref{pseudoMC} obtained from recursively solving \eqref{MCiterationequation}

The equality between the two sums in \eqref{solve_sum_over_trees_2} (and hence those in \eqref{solve_sum_over_trees}) follows from the facts that the inputs are all the same and of degree 1, and simple combinatorial arguments in counting of trees.
Also note that the sums in \eqref{solve_sum_over_trees} are finite sums $(\text{mod $\mathbf{m}^{N+1}$})$ for every $N \in \inte_{>0}$ because $\incoming = \incoming_1 + \incoming_2 + \cdots$ and $\incoming_k \in \mathbf{G}^1 \otimes_R \mathbf{m}^{k}$ so that, modulo $\mathbf{m}^{N+1}$, there are only finitely many trees and finitely many $\incoming_k$'s involved.

\begin{remark}
Both the operators $\mathfrak{I}_{k,\tr}$ and $\mathfrak{l}_{k,\rtr}$ will be used, but for different purposes: $\mathfrak{l}_{k,\rtr}$ does not involve automorphisms of trees, so it will be used in Section \ref{sec:leading_order_MC} to simplify some of the notations; while $\mathfrak{I}_{k,\tr}$ is conceptually more relevant to operations on dgLa's, as we will see later.
\end{remark}

\begin{remark}\label{rem:L_infinity_structure}
Sum-over-trees formulas similar to \ref{solve_sum_over_trees} appear quite often in the literature, in particular in applications of the homological perturbation lemma \cite{kontsevich00} and study of $L_\infty$ (or $A_\infty$) structures \cite{fukaya2003deformation}.

\end{remark}

\subsection{Scattering of two non-parallel walls}\label{sec:solve_MC_two_walls}

Suppose we are given two non-parallel walls
$\mathbf{w}_1 = (m_1, P_1, \Theta_1)$, and  $\mathbf{w}_2 = (m_2, P_2, \Theta_2)$ where $P_1, P_2$ are oriented tropical hyperplanes intersecting in a codimension $2$ tropical subspace $Q := P_1 \cap P_2$ in an affine convex coordinate chart $U \subset B_0$.
The ansatz in Definition \ref{ansatz} gives two Maurer-Cartan (abbrev. MC) solutions $\incoming_{\mathbf{w}_i} \in \widehat{\mathbf{G}}^1(U)$, $i=1, 2$, but their sum
$\incoming := \incoming_{\mathbf{w}_1} + \incoming_{\mathbf{w}_2} \in \widehat{\mathbf{G}}^1(U)$
does {\em not} solve the MC equation \eqref{MCequation}.

As we mentioned in the Introduction, the method of Kuranishi \cite{Kuranishi65} with a specific choice of the homotopy operator allow us to construct from $\incoming$ a MC solution $\Phi$ of $\widehat{\mathbf{G}}^1(U)$ {\em up to errors terms with exponential order in $\hp^{-1}$}, i.e. terms of the form $O(e^{-c/\hp})$.\footnote{$\incoming$ is not a MC solution even up to such errors terms.} More precisely, we will construct MC solutions of the dgLa $\widehat{\mathbf{g}^*/\mathcal{E}^*}(U)$, which is a quotient of a sub-dgLa of $\widehat{\mathbf{G}}^*(U)$, and show that they naturally give rise to consistent scattering diagrams.

We will first introduce the dgLa $\widehat{\mathbf{g}^*/\mathcal{E}^*}(U)$ in Section \ref{sec:MC_error} and construct a specific homotopy operator $H$ in Section \ref{con_propagator}, before starting the asymptotic analysis of the MC solutions we constructed in Sections \ref{sec:leading_order_MC} and \ref{sec:tropical_leading_order}. The key results are Theorem \ref{asy_support_theorem} and Lemma \ref{lem:semi_classical_integral}.

\subsubsection{Solving the MC equation modulo error terms with exponential order in $\hp^{-1}$}\label{sec:MC_error}

\begin{definition}\label{finite_Lie_alg}
We define a dg-Lie subalgebra in $\mathbf{G}^*_N(U)$ by
\begin{equation*}
\mathbf{g}^*_N(U) := \left(\bigoplus_m \filt^\infty_*(U)\bmc^m\right) \otimes_\inte \Lambda^\vee_{B_0}(U) \otimes_\comp (R/\mathbf{m}^{N+1}),
\end{equation*}
where $\filt^{\infty}_*(U) \subset \Omega^*_\hp(U)$ is the space of differential forms with polynomial $\hp^{-1}$ order defined in \ref{polynomial_growth}.
A general element of $\mathbf{g}^*_N(U)$ is a finite sum of the form
$
\sum_{j} \sum_{m, n} \alpha_{jm}^n \bmc^{m} \check{\partial}_n t^j,
$
where $\alpha_{jm}^n \in \filt^{\infty}_*(U)$. We have the inverse limit $\hat{\mathbf{g}}^*(U) : = \varprojlim \mathbf{g}^*_N(U)$.

There is a dg-Lie ideal $\mathcal{E}^*_N(U)$ of $\mathbf{g}^*_N(U)$ containing exponentially decay errors terms in $\hp^{-1}$:
\begin{equation}\label{error_ideal}
\mathcal{E}^*_N(U) := \left(\bigoplus_m \filt^{-\infty}_*(U)\bmc^m\right) \otimes_\inte \Lambda^\vee_{B_0}(U)\otimes_\comp (R/\mathbf{m}^{N+1}),
\end{equation}
where $\filt^{-\infty}_*(U) \subset \Omega^*_\hp(U)$ is the space of differential forms with exponential $\hp^{-1}$ order as in \ref{exponential_decay}.

Then we take the quotient
\begin{equation}
\mathbf{g}^*_N(U)/\mathcal{E}^*_N(U) = \left(\bigoplus_m \left(\filt^\infty_*(U)/ \filt^{-\infty}_*(U)\right) \bmc^m\right) \otimes_\inte \Lambda^\vee_{B_0}(U) \otimes_\comp (R/\mathbf{m}^{N+1})
\end{equation}
and define the dgLa $\widehat{\mathbf{g}^*/\mathcal{E}^*}(U)$ as the inverse limit
$\widehat{\mathbf{g}^*/\mathcal{E}^*}(U) := \varprojlim (\mathbf{g}^*_N(U)/\mathcal{E}^*_N(U)).$
\end{definition}

\begin{remark}\label{rem:delta_function_alg}
The advantage of working with the quotient $\filt^{\infty}_*/\filt^{-\infty}_*$ is that, given any element $\alpha \in \filt^s_{P}(U)$ and any cut off function $\chi$ (independent of $\hp$) such that $\chi \equiv 1$ in a neighborhood of $P$, we have $\alpha = \chi \alpha$ in the quotient $\filt^s_{P}(U)/\filt^{-\infty}_*(U)$, so an element in $\filt^s_{P}(U)/\filt^{-\infty}_*(U)$ can be treated as a delta function supported along $P$.
\end{remark}


\begin{lemma}\label{cohomology_lemma}
For the dgLa $\mathbf{g}^*_N(U)/\mathcal{E}^*_N(U)$ in a contractible open subset $U$, we have
$H^{>0}(\mathbf{g}^*_N(U)/\mathcal{E}^*_N(U)) = 0$
and
$$H^0(\mathbf{g}^*_N(U)/\mathcal{E}^*_N(U)) = \left(\bigoplus_m H^0(\filt^\infty_*(U)/\filt^{-\infty}_*(U)) \bmc^m\right) \otimes_\inte \Lambda^\vee_{B_0}(U) \otimes_\comp (R/\mathbf{m}^{N+1})$$
where
$$
H^0(\filt^\infty_*(U)/\filt^{-\infty}_*(U)) = \frac{\{f:\real_{>0}\rightarrow \real \mid |f(\hp)| \leq C \hp^{-N} \text{ for some $C$ and $N$}\}}{  \{f:\real_{>0}\rightarrow \real \mid |f(\hp)| \leq C e^{-c/\hp} \text{ for some $c$ and $C$}\}}.
$$
\end{lemma}
The proof of the above Lemma \ref{cohomology_lemma} relies on construction of homotopy operator. We will give the proof using the homotopy operator $H$ constructed in Section \ref{con_propagator} when $U$ is a spherical neighborhood as described in Notation \ref{choice_of_U}; the proof of a general contractible $U$ works exactly in the same way by using the homotopy operator constructed from the map $\rho_{x^0} : [0,1] \times U \rightarrow U$ which contracts $U$ to a point $x^0$.

\subsubsection{Construction of the homotopy operator}\label{con_propagator}

Recall from Section \ref{algebraicsolveMC} that a homotopy operator $H$ (sometimes called a propagator) is needed for gauge fixing if we want to apply Kuranishi's method to solve the MC equation. To define this (and other operators), we may need to shrink $U$ to a spherical neighborhood as follows.

\begin{notation}\label{choice_of_U}
Suppose that we have two non-parallel walls $\mathbf{w}_i = (m_i, P_i, \Theta_i)$ ($i = 1, 2$) intersecting transversally in a codimension $2$ tropical subspace $Q := P_1 \cap P_2$ in an affine convex open subset $V \subset B_0$.
We fix a point $q_0 \in Q$.
By reversing the orientations on $P_i$'s (and replacing $\Theta_i$ by $\Theta_i^{-1}$ accordingly) if necessary, we can choose the oriented normals of $P_1$ and $P_2$ to be $-m_1 = -\nu_{P_2}$ and $\ -m_2 = \nu_{P_1}.$ We orient the rank $2$ normal bundle $NQ$ by the ordered basis $\{-m_1, -m_2\}$.
By identifying an open neighborhood of the zero section in the normal bundle $NQ$ with a tubular neighborhood of $Q$ in $B_0$, we see that the two walls are dividing $V \cap NQ$ into $4$ quadrants; this can be visualized in the 2-dimensional slice $V \cap NQ_{q_0}$ in $V \cap NQ$, as shown in Figure \ref{fig:quadrant}

\begin{figure}[h]
\begin{center}
\includegraphics[scale=0.35]{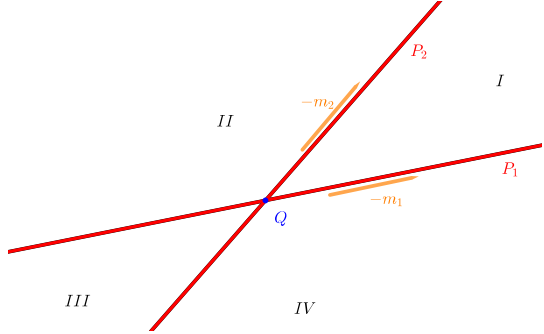}
\end{center}
\caption{The slice $V \cap NQ_{q_0}$ in $V \cap NQ$.}\label{fig:quadrant}
\end{figure}

Now we fix local affine coordinates in $V$ near $q_0$, and choose an affine flat metric $g_V$ with the property that $m_1$, $m_2$ and $TQ$ are perpendicular to each other.
Then we choose a point $x^0$ in the third quadrant in $NQ_{q_0}$ (see Figure \ref{fig:quadrant}) with $x^0 \notin (P_1 \cup P_2)$ and a ball $U \subset V$ (defined using the metric $g_V$) centered at $x^0$ which contains $q_0$.
We fix this neighborhood $U$ centered at $x^0$ and call it a {\em spherical neighborhood}; see Figure \ref{fig:spherical_nbh}. From this point on, we will work with a spherical neighborhood $U \subset B_0$ for the rest of this paper.

\begin{figure}[h]
\begin{center}
\includegraphics[scale=0.3]{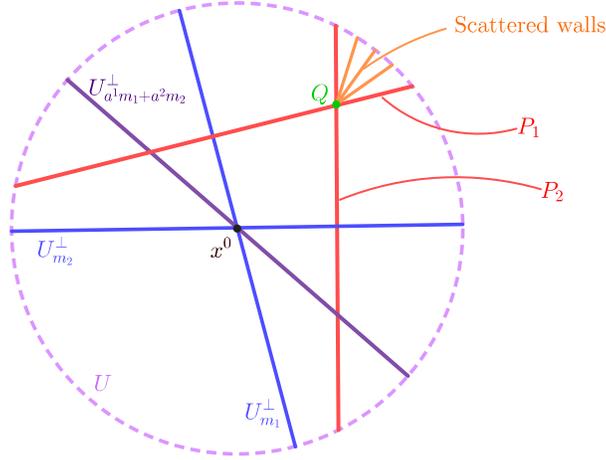}
\end{center}
\caption{The spherical neighborhood $U$.}\label{fig:spherical_nbh}
\end{figure}
\end{notation}


Since $TU = \Lambda_{B_0} \otimes_\inte \real$ and $\Gamma(U, \Lambda_{B_0}) \cong M$, we can identify an element $0 \neq m \in M$ with an affine vector field $m \in \Gamma(U, \Lambda_{B_0}) \subset \Gamma(U, TU)$ (with respect to the affine structure on $B_0$). We denote by $U_m^{\perp}$ the tropical hyperplane perpendicular to $m$ with respect to the metric $g_V$. Then $U_m^{\perp}$ divides $U$ into two half-spheres $U_m^+$ and $U_m^-$, which are named so that $-m$ is pointing into $U_m^+$. The property that $m_1$, $m_2$ and $TQ$ are perpendicular to each other then implies that
\begin{equation}\label{eqn:Q_intersect_U_a}
Q \cap U \subset U_{a^1m_1 + a^2m_2}^+
\end{equation}
for all $(a^1,a^2) \in (\inte_{\geq 0})^2 \setminus \{0\}$; see Figure \ref{fig:spherical_nbh}.

To define a homotopy operator on $\mathbf{g}^*_N(U)$, we will first define one on the direct sum $\bigoplus_m \filt^\infty_*(U)\bmc^m$ and extend it by taking tensor product.
For each $m \in M$, recall that we have $\pdb(\alpha \bmc^{m}) = d(\alpha) \bmc^{m}$, where $d$ is the deRham differential on $U$. So the cohomology
$H^*(\filt^\infty_*(U)\bmc^m, \pdb) = \{f:\real_{>0}\rightarrow \real \mid |f(\hp)| \leq C \hp^{-N} \text{ for some $C$ and $N$}\}$
is represented by functions depending only on $\hp$ and with polynomial growth in $\hp^{-1}$.

We are going to construct a homotopy operator $H_m$ on $\filt^\infty_*(U)\bmc^m$ which retracts to its cohomology $H^*(\filt^\infty_*(U)\bmc^m) = H^*(\filt^\infty_*(U))\bmc^m$.
The hyperplane $U_m^{\perp}$ we chose above is playing the role of the reference hyperplane $R$ when we define the operator $I$ in \eqref{general_integral_operator} in Section \ref{sec:asy_support}.
In the current situation, we need a family of reference hyperplanes $U_{a^1m_1 + a^2m_2}^{\perp}$, and the condition \eqref{eqn:Q_intersect_U_a} is to ensure that we can define $H_m$ in the same way as $I$ and apply Lemma \ref{integral_lemma} for each $m = a^1m_1 + a^2m_2$.

Now for $0 \neq m \in M$, as in Lemma \ref{integral_lemma}, we use flow lines along the affine vector field $-m$ to define a diffeomorphism
$\tau_m : W_m \rightarrow U,$
where $W_m \subset \real \times U^\perp_m$ is the maximal domain of definition of $\tau$.
Under the diffeomorphism $\tau_m$, we obtain affine coordinates $(t =: u_{m,1}, u_{m,2}, \dots, u_{m,n})$ on $U$ such that $x^0 = (0, \dots, 0)$. Note that these coordinates satisfy the condition in Notations \ref{orthonormalcoordinates} and we will set $\uperp := (u_{m,2}, \dots, u_{m,n})$.
For $m = 0 \in M$, we will choose an arbitrary set of local affine coordinates $(u_{0,1}, u_{0,2}, \dots, u_{0,n})$ in defining the homotopy operator $H$.

In the coordinates $(u_{m,1}, \dots, u_{m,n})$, we decompose a differential form $\alpha \in \filt^\infty_*(U)$ uniquely as
\begin{equation}\label{eqn:alpha_decomposition}
\alpha = \alpha_0 + du_{m,1} \wedge \alpha_1,
\end{equation}
where $\iota_{\dd{u_{m,1}}} \alpha_0 = \iota_{\dd{u_{m,1}}} \alpha_1 = 0$.
We define a contraction $\rhoperp : \real \times U^{\perp}_{m} \rightarrow U^{\perp}_{m}$ by $\rhoperp(r,\uperp) := r\uperp$.

\begin{definition}\label{MC_homotopy}
We define the homotopy operator $H_m : \filt^\infty_*(U) \bmc^m \rightarrow \filt^\infty_{*-1}(U) \bmc^m$ by
$H_m(\alpha \bmc^m) := \left(I_{m,er}(\alpha) + I_m(\alpha)\right)\bmc^{m},$
where we set
\begin{align*}
I_{m,er}(\alpha)(u_{m,1},\uperp) & := \int_{0}^1 \rhoperp^*(\alpha_0(0,\cdot )) = \int_{0}^1 \big( \iota_{\dd{r}}\rhoperp^*(\alpha_0|_{U_m^{\perp}}) \big) dr,  \\
I_m(\alpha)(u_{m,1},\uperp) & := \int_{0}^{u_{m,1}} \alpha_1(s,\uperp) ds
\end{align*}
using the decomposition of differential forms $\alpha \in \filt^\infty_*(U)$ specified in \eqref{eqn:alpha_decomposition}.

We define the projection $\mathcal{P}_m : \filt^\infty_*(U) \bmc^m \rightarrow H^*(\filt^\infty_*(U)) \bmc^m$ by
by
$$\mathcal{P}_m (\alpha \bmc^m) := \begin{cases} (\alpha|_{x^0})\bmc^{m} & \text{for $\alpha$ of degree $0$},\\  0 & \text{otherwise},\end{cases}$$
where $\alpha|_{x^0}$ is evaluation of $\alpha$ at the point $x^0$ and is to be treated as a constant function along $U$, and
the operator $\iota_m : H^*(\filt^\infty_*(U))\bmc^m \rightarrow \filt^\infty_*(U) \bmc^m$ by
$\iota_m (\alpha \bmc^m) := \iota(\alpha) \bmc^m,$ where $\iota: H^*(\filt^\infty_*(U)) \hookrightarrow \filt^\infty_*(U)$ is the embedding of constant functions over $U$ at degree $0$ and $0$ otherwise.

We will abuse notations by treating $H_m$, $\mathcal{P}_m$ and $\iota_m$ as acting on the spaces $\filt^\infty_*(U)$ and $H^*(\filt^\infty_*(U))$.
\end{definition}

\begin{prop}\label{homotopy_prop}
The operator $H_{m}$ is a homotopy retract of $\filt^\infty_*(U) \bmc^m$ onto its cohomology $ H^*(\filt^\infty_*(U)) \bmc^m$, i.e. we have
$
\text{Id} - \iota_m \mathcal{P}_m = \pdb H_m + H_m \pdb.
$
\end{prop}

The integral operator $H_m$ preserves $\filt^{-\infty}_*(U)$ because the path integrals preserve terms with exponential decay in $\hp$ as one can see from Definition \ref{exponential_decay}. Also, we have the natural identifications
\begin{align*}
H^*(\filt^\infty_*(U),d) & = \{f: \real_{>0} \rightarrow \real \mid |f(\hp)| \leq C \hp^{-N} \text{ for some $C$ and $N$}\},\\
H^*(\filt^{-\infty}_*(U),d) & = \{f: \real_{>0} \rightarrow \real \mid |f(\hp)| \leq C e^{-c/\hp} \text{ for some $c$ and $C$}\},
\end{align*}
under which we see that the operators $\mathcal{P}_m$ and $\iota_m$ can be descended to the quotient:
\begin{eqnarray*}
\mathcal{P}_m &: \filt^{\infty}_*(U)/ \filt^{-\infty}_*(U) \rightarrow H^*(\filt^{\infty}_*(U),d)/H^*(\filt^{-\infty}_*(U),d),\\
\iota_m &: H^*(\filt^{\infty}_*(U),d)/H^*(\filt^{-\infty}_*(U),d) \rightarrow \filt^{\infty}_*(U)/ \filt^{-\infty}_*(U),
\end{eqnarray*}
again in view of Definition \ref{exponential_decay}. Therefore, Proposition \ref{homotopy_prop} holds in the quotient $\filt^{\infty}_*(U)/ \filt^{-\infty}_*(U)$ as well.

\begin{definition}\label{pathspacehomotopy}
We define the operators
$H := \bigoplus H_m$, $\mathcal{P} := \bigoplus \mathcal{P}_m$ and $\iota := \bigoplus \iota_m$
acting on the direct sum $\bigoplus_m \filt^\infty_*(U)\bmc^m$ and its cohomology.
These operators extend naturally to the tensor product $\mathbf{g}^*_N(U) = \left(\bigoplus_m \filt^\infty_*(U)\bmc^m\right) \otimes_\inte \Lambda^\vee_{B_0}(U) \otimes_\comp (R/\mathbf{m}^{N+1})$ and descend to the quotient $\mathbf{g}^*_N(U)/\mathcal{E}^*_N(U)$ (since $\mathcal{E}^*_N(U)$ is a dg ideal of $\mathbf{g}^*_N(U)$).
We then take the inverse limit to define the operators $H$, $\mathcal{P}$ and $\iota$ acting on $\widehat{\mathbf{g}^*/\mathcal{E}^*}(U)$.
\end{definition}

\begin{remark}
We remark that the homotopy operators defined above depend on the choices of $U$, the affine coordinates, etc, and so does the MC solution $\Phi$ that we are going to construct. However, the scattering diagram $\mathscr{D}(\Phi)$ associated to $\Phi$ is independent of these choices.
\end{remark}

\begin{proof}[Proof of Lemma \ref{cohomology_lemma}]
We prove this lemma for each direct summand $(\filt^{\infty}_*(U)/\filt^{-\infty}_*(U)) \bmc^m$, which will be identified with $(\filt^{\infty}_*(U)/\filt^{-\infty}_*(U))$
so that the Witten differential $\pdb$ becomes the usual deRham differential $d$.
We have the following operators
\begin{align*}
H_m: & \filt^{\infty}_*(U)/\filt^{-\infty}_*(U)\rightarrow \filt^{\infty}_{*-1}(U)/\filt^{-\infty}_{*-1}(U)\\
\mathcal{P}_m: & \filt^{\infty}_*(U)/\filt^{-\infty}_*(U) \rightarrow
\frac{\{f: \real_{>0} \rightarrow \real \mid |f(\hp)| \leq C \hp^{-N} \text{ for some $C$ and $N$}\}}{\{f: \real_{>0} \rightarrow \real \mid |f(\hp)| \leq C e^{-c/\hp} \text{ for some $c$ and $C$}\}}\\
\iota_m: & \frac{\{f:\real_{>0}\rightarrow \real \mid |f(\hp)| \leq C \hp^{-N} \text{ for some $C$ and $N$}\}}{  \{f:\real_{>0}\rightarrow \real \mid |f(\hp)| \leq C e^{-c/\hp} \text{ for some $c$ and $C$}\}}\rightarrow \filt^{\infty}_*(U)/\filt^{-\infty}_*(U)
\end{align*}
defined in Definition \ref{pathspacehomotopy}. By descending the formula in Proposition \ref{homotopy_prop} to the quotient by $\filt^{-\infty}_*(U)$, we have
$
I-\iota_m \circ \mathcal{P}_m = \pdb H_m + H_m \pdb.
$
The result follows by a natural extension of this homotopy equation to $\mathbf{g}^*_N(U)$.
\end{proof}



\subsubsection{Asymptotic analysis of Maurer-Cartan solutions}\label{sec:leading_order_MC}

Going back to the two given non-parallel walls
$\mathbf{w}_1 = (m_1, P_1, \Theta_1)$, $\mathbf{w}_2 = (m_2, P_2, \Theta_2).$
Recall that each wall crossing factor $\Theta_i$ is of the form
$
\text{Log}(\Theta_i) = \sum_{j, k \geq 1} a^{(i)}_{jk} w^{km_i} \check{\partial}_{n_i} t^j,
$
where $n_i \in \Lambda_{B_0}^\vee(U) \cong N$ is the unique primitive element satisfying $n_i \in (TP_i)^\perp$ and $(\nu_{P_i}, n_i) < 0$, and $\nu_{P_i} \in TU$ a normal to $P_i$ so that the orientation on $TP_i \oplus \real \cdot \nu_{P_i}$ agrees with that on $U$ (see Definition \ref{wall}). As in Remark \ref{general_input} in Section \ref{onewall}, we assume that the two inputs $\incoming^{(1)}, \incoming^{(2)}$ associated to the walls $\mathbf{w}_1, \mathbf{w}_2$ are of the following form:
\begin{assum}\label{input_assumption}
We assume that there are two MC solutions $\incoming^{(1)}, \incoming^{(2)}$ of $\widehat{\mathbf{g}^*/\mathcal{E}^*}(U)$,\footnote{Obviously MC solutions of the dg-Lie subalgebra $\hat{\mathbf{g}}^*(U) \subset KS_{\check{X}_0}(U)[[t]]$ descend to the quotient to give MC solutions of $\widehat{\mathbf{g}^*/\mathcal{E}^*}(U)$.} which can be represented by elements in $\hat{\mathbf{g}}^*(U)$ of the form
\begin{equation}\label{two_wall_incoming}
\incoming^{(i)} \in -\sum_{j, k \geq 1} a^{(i)}_{jk}(\delta_{jk}^{(i)}+\filt^{0}_{P_i}(U))\bmc^{km_i} \check{\partial}_{n_i} t^j
\end{equation}
with $a^{(i)}_{jk} \neq 0$ only for finitely many $k$'s for each fixed $j$, and each $\delta_{jk}^{(i)} \in \filt^1_{P_i}(U)$ can be written as $\delta_{jk}^{(i)} = (\pi\hp)^{-1/2} e^{-\frac{\eta^2}{\hp}} d\eta$ in some neighborhood $W_{P_i}$ of $P_i$ for some affine linear function $\eta$ on $W_{P_i}$ such that $P_i$ is defined by $\eta=0$ locally and $\iota_{\nu_{P_i}} d\eta > 0$.
\end{assum}

For convenience, we will abuse notations and use $\incoming^{(i)} \in \hat{\mathbf{g}}^*(U)$ to denote its class in $\widehat{\mathbf{g}^*/\mathcal{E}^*}(U)$ as well.
We will also denote by
\begin{equation}\label{essential_incoming}
\eincoming^{(i)} := -\sum_{j, k \geq 1} a^{(i)}_{jk}\delta_{jk}^{(i)}\bmc^{km_i} \check{\partial}_{n_i } t^j
\end{equation}
the leading order term of the input $\incoming^{(i)}$ for $i = 1, 2$. Then we have
\begin{equation}\label{incoming_asy_support}
\eincoming^{(i)}\in \left(\bigoplus_{k\geq 1} \filt^{1}_{P_i}(U) \bmc^{km_i} \check{\partial}_{n_i}\right)[[t]]; \quad
\incoming^{(i)} - \eincoming^{(i)}  \in \left(\bigoplus_{k\geq 1} \filt^{0}_{P_i}(U) \bmc^{km_i} \check{\partial}_{n_i}\right)[[t]].
\end{equation}

We now solve the MC equation of the dgLa $\widehat{\mathbf{g}^*/\mathcal{E}^*}(U)$ by solving the equation \eqref{pseudoMC} with the input data
$
\incoming := \incoming^{(1)} + \incoming^{(2)}.
$
Using the homotopy operator $H$ and applying the sum over trees formula \eqref{solve_sum_over_trees}, we obtain an element
\begin{equation}\label{eqn:MC_sol_Phi}
\Phi := \sum_{k\geq 1 } \frac{1}{k!}\mathfrak{I}_{k}(\incoming, \ldots, \incoming) = \sum_{k\geq 1} \mathfrak{l}_k(\incoming, \ldots, \incoming),
\end{equation}
in $\hat{\mathbf{g}}^*(U)$, whose class in $\widehat{\mathbf{g}^*/\mathcal{E}^*}(U)$ will also be denoted by $\Phi$.

\begin{lemma}\label{approximated_MC}
The solution $\Phi$ constructed from the input $\incoming = \incoming^{(1)} + \incoming^{(2)}$ using the equation \eqref{pseudoMC} and the homotopy operator $H$ defined in Definition \ref{pathspacehomotopy} is a MC solution in $\widehat{\mathbf{g}^*/\mathcal{E}^*}(U)$, i.e. we have $\mathcal{P}[\Phi,\Phi] = 0$ in $\widehat{\mathbf{g}^*/\mathcal{E}^*}(U)$.
\end{lemma}
We postpone the proof of Lemma \ref{approximated_MC} to Section \ref{sec:leading_order_MC}. From Lemmas \ref{approximated_MC} and \ref{cohomology_lemma}, we obtain a unique element $\varphi \in \widehat{\mathbf{g}^*/\mathcal{E}^*}(U)$ satisfying $e^{\varphi}* 0 = \Phi$ and $\mathcal{P} (\varphi) = 0$ in $\widehat{\mathbf{g}^*/\mathcal{E}^*}(U)$ using Lemma \ref{gauge_fixing_lemma}.
To start the asymptotic analysis of $\Phi$ and $\varphi$, we first decompose the Lie bracket $[\cdot,\cdot]$ on $\widehat{\mathbf{g}^*/\mathcal{E}^*}(U)$ into three types of operators:

\begin{definition}\label{lie_bracket_decomp}
For $\alpha =  f \bmc^{m} \check{\partial}_{n}$ and $\beta = g \bmc^{m'} \check{\partial}_{n'}$ where $f, g \in \filt^{\infty}_*(U)$, we decompose the Lie bracket $[\cdot, \cdot]$ into three operators $\natural$, $\sharp$ and $\flat$ defined by
\begin{align*}
\natural(\alpha, \beta) & := (-1)^{\bar{f}(\bar{g}+1)}f g [\bmc^{m} \check{\partial}_{n}, \bmc^{m'} \check{\partial}_{n'}],\\
\sharp(\alpha,\beta) & :=  (-1)^{\bar{f}(\bar{g}+1)}f (\nabla_{\partial_{n}} g) \bmc^{m+m'} \check{\partial}_{n'},\\
\flat(\alpha,\beta) & := (-1)^{(\bar{f}+1)\bar{g}} g (\nabla_{\partial_{n'}} f) \bmc^{m+m'} \check{\partial}_{n};
\end{align*}
here $\bar{f}$ and $\bar{g}$ denote the degrees of $f$ and $g$ respectively.
These operators extend by linearity to $\mathbf{g}^*_N(U)$ by treating a general element as a polynomial on the basis $\{\bmc^m t^j\}$, and descend to the quotient $\mathbf{g}^*_N(U)/\mathcal{E}^*_N(U)$, and can be further extended to $\hat{\mathbf{g}}^*(U)$ and $\widehat{\mathbf{g}^*/\mathcal{E}^*}(U)$ by taking inverse limits.
\end{definition}

Next we will decompose the operation $\mathfrak{l}_{k}$ defined in Definition \ref{r_tree_operation} according to the above decomposition of the Lie bracket $[\cdot,\cdot]$, the powers of the formal variable $t$ and the Fourier modes $\{\bmc^m\}$. For this purpose, we need to introduce the notion of a {\em labeled $k$-tree}:

\begin{definition}\label{label_tree_def}
A {\em labeled ribbon $k$-tree} is a ribbon $k$-tree $\lrtr$ together with
\begin{itemize}
\item
a labeling of each trivalent vertex $v \in \lrtr^{[0]}$ by $\natural, \sharp$ or $\flat$, and
\item
a labeling of each incoming edge $e \in \partial_{in}^{-1}(\lrtr^{[0]}_{in})$ by a pair $(m_e, j_e)$,
where $m_e = k m_i$ (for some $k>0$ and $i = 1, 2$) specifies the Fourier mode and $j_e \in \inte_{>0}$ specifies the order of the formal variable $t$ (corresponding to the input term $\bmc^{km_i} t^{j_e}$ in $\hat{\mathbf{g}}^*(U)$).
\end{itemize}
Similarly, we define a {\em labeled $k$-tree} as a $k$-tree $\ltr$ together with a labeling of the trivalent vertices $\ltr^{[0]}$ by $\natural$ or $\sharp+\flat$ (as only symmetric operations are allowed if there is no ribbon structure) and the same labeling of the incoming edges $\partial^{-1}_{in}(\ltr^{[0]}_{in})$ as above.
We use $\underline{\lrtr}$ to denote the underlying labeled $k$-tree of a labeled ribbon $k$-tree $\lrtr$.

Two labeled ribbon $k$-trees $\lrtr_1$ and $\lrtr_2$ (resp. two labeled $k$-trees $\ltr_1$ and $\ltr_2$) are said to be {\em isomorphic} if they are isomorphic as ribbon $k$-trees (resp. $k$-trees) and the isomorphism preserves the labeling. The set of isomorphism classes of labeled ribbon $k$-trees (resp. labeled $k$-trees) will be denoted by $\lrtree{k}$ (resp. $\ltree{k}$).
As before, we will abuse notations by using $\lrtr$ (resp. $\ltr$) to stand for an isomorphism class of labeled ribbon $k$-trees (resp. labeled $k$-trees).
\end{definition}

\begin{notation}\label{not:edge_labeling}
For a labeled ribbon $k$-tree $\lrtr$ (resp. labeled $k$-tree $\ltr$), there is an induced labeling of all the edges in $\lrtr^{[1]}$ (resp. $\ltr^{[1]}$) by the rule that at any trivalent vertex $v \in \lrtr^{[0]}$ (resp. $v \in \ltr^{[0]}$) with two incoming edges $e_1, e_2$ and one outgoing edge $e_3$, we set
$(m_{e_3},j_{e_3}) := (m_{e_1},j_{e_1}) + (m_{e_2},j_{e_2}).$
We also write $(m_{\lrtr}, j_{\lrtr})$ (resp. $(m_{\ltr}, j_{\ltr})$) for the labeling of the unique edge $e_o$ attached to the outgoing vertex $v_o$.
\end{notation}

\begin{definition}\label{color_tree_operation}
Given a labeled ribbon $k$-tree $\lrtr$, we label the incoming vertices by $v_1,\dots,v_k$ according to its cyclic ordering.
We define the operator (similar to Definition \ref{r_tree_operation})
$
\mathfrak{l}_{k,\lrtr}: (\hat{\mathbf{g}}^*(U)[1])^{\otimes k} \rightarrow \hat{\mathbf{g}}^*(U)[1],
$
for inputs $\zeta_1, \dots, \zeta_k$ by
\begin{enumerate}
\item
extracting the coefficient of the term $\bmc^{m_{e_i}} t^{j_{e_i}}$ in $\zeta_i$ and aligning it as the input at $v_i$,
\item
applying the operators $\natural$, $\sharp$ or $\flat$ to each trivalent vertex $v \in \lrtr^{[0]}$ according to the labeling,
\item
and applying the homotopy operator $-H$ to each edge in $\lrtr^{[1]}$.
\end{enumerate}
The operator $\mathfrak{l}_{k,\lrtr}$ descends to $\widehat{\mathbf{g}^*/\mathcal{E}^*}(U)$ and will be denoted by the same notation.
\end{definition}

\begin{notation}\label{two_sets_of_trees}
We decompose the set $\lrtree{k}$ of isomorphism classes of labeled ribbon $k$-trees into two parts:
$\lrtree{k} = \lrtree{k}_0 \sqcup \lrtree{k}_1,$
where $\lrtree{k}_0$ consists of trees whose trivalent vertices are all labeled by $\natural$ and $\lrtree{k}_1 := \lrtree{k} \setminus \lrtree{k}_0$. We then consider the operators
\begin{align*}
\mathfrak{l}_{k,0} := \sum_{\lrtr \in \lrtree{k}_0} \frac{1}{2^{k-1}}\mathfrak{l}_{k,\lrtr},\
\mathfrak{l}_{k,1} := \sum_{\lrtr \in \lrtree{k}_1} \frac{1}{2^{k-1}}\mathfrak{l}_{k,\lrtr}.
\end{align*}
\end{notation}

It is easy to see that for each labeled ribbon $k$-tree $\lrtr$, the labeling $m_{\lrtr}$ associated to the unique outgoing edge is of the form
$m_{\lrtr} = l (a_1 m_1 + a_2 m_2)$
for some $l \in \inte_{>0}$ and $(a_1, a_2) \in (\inte_{\geq 0})^2_{\text{prim}}$, where $(\inte_{\geq 0})^2_{\text{prim}}$ denotes the set of all primitive elements in $(\inte_{\geq 0})^2 \setminus \{0\}$, so the solution $\Phi$ can be decomposed as a sum of Fourier modes parametrized by $(\inte_{\geq 0})^2_{\text{prim}}$.

\begin{notation}\label{two_wall_mode}
We let $m_a := a_1 m_1 + a_2 m_2$ for $a = (a_1, a_2) \in (\inte_{\geq 0})^2_{\text{prim}}$. Note that we have $m_{(1,0)} = m_1$ and $m_{(0,1)} = m_2$.
For each $a \in (\inte_{\geq 0})^2_{\text{prim}}$, we let $P_a$ be the tropical half-hyperplane
$P_a = Q - \real_{\geq 0} \cdot m_a.$
We equip each $P_a$ with a normal $\nu_{P_a}$ such that $\{-m_a,\nu_{P_a}\}$ agrees with orientation given by $\{-m_1,-m_2\}$ on $NQ$, and this gives an orientation on $P_a$ such that the orientation of $TP_a \oplus \real\cdot \nu_{P_a}$ agrees with that of $B_0$ (so that $\nu_{P_a}$ satisfies the condition in Definition \ref{wall} as well).
\end{notation}

\begin{definition}\label{solving_MC_tree}
Given the input $\incoming = \incoming^{(1)} + \incoming^{(2)} \in \hat{\mathbf{g}}^*(U)$,
we put
$
\emc := \sum_{k\geq 1 } \mathfrak{l}_{k,0}(\eincoming, \ldots, \eincoming) = \sum_{a \in (\inte_{\geq 0 })^2_{\text{prim}}} \emc^{(a)},
$
where
$$\emc^{(a)} \in \bigoplus_{\substack{k\geq 1\\ 1\leq j\leq N}}\sum_{n \in \Lambda^\vee_{B_0}(U): n \perp m_a} \filt^{\infty}_*(U) \cdot  \bmc^{km_a} \check{\partial}_{n} t^j \ (\text{mod $\mathbf{m}^{N+1}$})$$
for each $a \in (\inte_{\geq 0 })^2_{\text{prim}}$, and
$
\digamma := \sum_{k\geq 1 }( \mathfrak{l}_{k}(\incoming, \ldots, \incoming)-\mathfrak{l}_{k,0}(\eincoming, \ldots, \eincoming)  )  =\sum_{a \in (\inte_{\geq 0 })^2_{\text{prim}}} \digamma^{(a)},
$
where
$$\digamma^{(a)} \in \bigoplus_{\substack{k\geq 1\\ 1\leq j\leq N}} \sum_{n \in \Lambda^\vee_{B_0}(U)} \filt^{\infty}_*(U) \cdot \bmc^{km_a} \check{\partial}_{n} t^j \ (\text{mod $\mathbf{m}^{N+1}$})$$
for each $a \in (\inte_{\geq 0 })^2_{\text{prim}}$.
\end{definition}

Then we have
$\Phi = \emc + \digamma,$
where $\emc$ are the leading $\hp$ order terms and $\digamma$ are the error terms, as $\hp \rightarrow 0$.
We also put $\Phi^{(a)} := \emc^{(a)} + \digamma^{(a)}$.
The key result on the asymptotic analysis of $\Phi$ is the following:

\begin{theorem}\label{asy_support_theorem}
For each $a \in (\inte_{>0})^2_{\text{prim}}$, we have
$$
\emc^{(a)}  \in  \left(\bigoplus_{k\geq 1} \filt^{1}_{P_a}(U) \bmc^{km_a} \check{\partial}_{n_a}\right)[[t]]; \quad
\digamma^{(a)}  \in  \left(\bigoplus_{k\geq 1} \sum_{n \in \Lambda_{B_0}^\vee(U)} \filt^{0}_{P_a}(U) \bmc^{km_a} \check{\partial}_{n}\right)[[t]],
$$
where $n_a \in \Lambda_{B_0}^\vee(U)$ is the unique primitive normal to $P_a$ such that $(\nu_{P_a},n_a) < 0$.
\end{theorem}
\begin{proof}
According to the definitions of $\emc^{(a)}$ and $\digamma^{(a)}$ in Definition \ref{solving_MC_tree}, this theorem is equivalent to the following statements:
\begin{equation*}
\left\{
\begin{array}{rc}
\mathfrak{l}_{k,\lrtr}(\eincoming,\dots,\eincoming) \in \filt^{1}_{P_a}(U) \bmc^{m_\lrtr} \check{\partial}_{n_a} t^{j_{\lrtr}} & \text{if $\lrtr \in \lrtree{k}_0$},\\
\mathfrak{l}_{k,\lrtr}(\incoming,\dots,\incoming) - \mathfrak{l}_{k,\lrtr}(\eincoming,\dots,\eincoming) \in \sum_{n}  \filt^{0}_{P_a}(U) \bmc^{\lrtr} \check{\partial}_{n} t^{j_{\lrtr}} & \text{if $\lrtr \in \lrtree{k}_0$},\\
\mathfrak{l}_{k,\lrtr}(\incoming,\dots,\incoming) \in \sum_{n} \filt^{0}_{P_a}(U) \bmc^{\lrtr} \check{\partial}_{n} t^{j_{\lrtr}} & \text{if $\lrtr \in \lrtree{k}_1$}.
\end{array}
\right.
\end{equation*}
The condition $n_a \perp P_a$ in the first statement follows from a simple induction argument using the formula \eqref{vertex_lie_algebra}; see the proof of Lemma \ref{lem:semi_classical_integral} for more details.
All other statements follow from Lemma \ref{tree_lemma} below.
\end{proof}

To state Lemma \ref{tree_lemma}, we fix a labeled ribbon $k$-tree $\lrtr$ whose incoming edges are $e_1, \dots, e_k$ with labeling $(m_{e_1},j_{e_1}), \dots, (m_{e_k},j_{e_k})$ respectively.

We then consider the operation
$\mathfrak{l}_{k,\lrtr}((\alpha_1 \check{\partial}_{n_1}) \bmc^{m_{e_1}} t^{j_{e_1}},\dots,(\alpha_k \check{\partial}_{n_k}) \bmc^{m_{e_k}}t^{j_{e_k}})$
for given $\alpha_1, \dots, \alpha_k \in \filt^{\infty}_*(U)$ and $n_1, \dots, n_k \in \Gamma(U,\Lambda^\vee_{B_0})$, defined in exactly the same way as in Definition \ref{color_tree_operation},
and treat it as an operation on $\alpha_1 \check{\partial}_{n_1}, \dots, \alpha_k \check{\partial}_{n_k}$ via the formula
\begin{equation}\label{label_tree_operation}
\mathfrak{l}_{k,\lrtr}(\alpha_1 \check{\partial}_{n_1},\dots,\alpha_k \check{\partial}_{n_k})\bmc^{m_{\lrtr}}t^{j_{\lrtr}} := \mathfrak{l}_{k,\lrtr}((\alpha_1 \check{\partial}_{n_1}) \bmc^{m_{e_1}} t^{j_{e_1}},\dots,(\alpha_k \check{\partial}_{n_k}) \bmc^{m_{e_k}}t^{j_{e_k}}),
\end{equation}
which will further be abbreviated as
$\mathfrak{l}_{k,\lrtr}(\vec{\alpha},\vec{n}) := \mathfrak{l}_{k,\lrtr}(\alpha_1 \check{\partial}_{n_1},\dots,\alpha_k \check{\partial}_{n_k}),$
where we put $\vec{\alpha} := (\alpha_1, \dots, \alpha_k)$ and $\vec{n} := (n_1, \dots, n_k)$.

\begin{notation}\label{P_edge}
Given a labeled ribbon $k$-tree $\lrtr$ and suppose that for each incoming edge $e_i$, we have assigned a closed codimension 1 tropical polyhedral subset $P_{e_i}$, which is either one of the two initial hyperplanes $P_1, P_2$ or one of the half-hyperplanes $P_a$'s introduced in Notation \ref{two_wall_mode}.
We then inductively assign a (possibly empty) tropical hyperplane or half-hyperplane $P_e$ to each edge $e \in \ltr^{[1]}$ as follows:

If $\acute{e}_1$ and $\acute{e}_2$ are two incoming edges meeting at a vertex $v$ with an outgoing edge $\acute{e}_3$ for which $P_{\acute{e}_1}$ and $P_{\acute{e}_2}$ are defined beforehand, we set
$P_{\acute{e}_3} := (Q - \real_{\geq 0} m_{\acute{e}_3}) \cap U$
if both $P_{\acute{e}_1}$ and $P_{\acute{e}_1}$ are non-empty and they intersect transversally at $Q := P_{\acute{e}_1} \cap P_{\acute{e}_1}$ and
$P_{\acute{e}_3} := \emptyset$
otherwise (recall that {\em transversal intersection} between two closed tropical polyhedral subsets, including the case when they have nonempty boundaries, was defined right before the proof of Lemma \ref{support_product}).

We denote the hyperplane or half-hyperplane associated to the unique outgoing edge $e_o$ by $P_{\lrtr}$. Note that if $P_\lrtr \neq \emptyset$, then $P_\lrtr = P_a$ for some $a \in (\inte_{\geq 0})^2_{\text{prim}}$.
\end{notation}

\begin{lemma}\label{tree_lemma}
Given a labeled ribbon $k$-tree $\lrtr$, each of whose incoming edges $e_i$ is assigned with a closed codimension 1 tropical polyhedral subset $P_{e_i}$, which is either one of the two initial hyperplanes $P_1, P_2$ or one of the half-hyperplanes $P_a$'s introduced in Notation \ref{two_wall_mode}. Also given $\alpha_1, \dots, \alpha_k \in \filt^{\infty}_*(U)$ and $n_1, \dots, n_k \in \Gamma(U,\Lambda^\vee_{B_0})$ and suppose that $\alpha_i$ has asymptotic support (Definition \ref{asypmtotic_support_def}) on $P_{e_i}$ with either $\alpha_i \in \filt^{1}_{P_{e_i}}(U)$ or $\alpha_i \in \filt^{0}_{P_{e_i}}(U)$ for each $i = 1, \dots, k$, then we have
\begin{equation*}
\left\{\begin{array}{ll}
\mathfrak{l}_{k,\lrtr}(\vec{\alpha},\vec{n}) \in \filt^{1}_{P_{\lrtr}}(U)\otimes_\inte \Lambda^\vee_{B_0}(U) & \text{if $\lrtr \in \lrtree{k}_0$, $P_{\lrtr} \neq \emptyset$ and $\alpha_i\in \filt^{1}_{P_{e_i}}(U)$ for all $i$},\\
\mathfrak{l}_{k,\lrtr}(\vec{\alpha},\vec{n}) \in \filt^{0}_{P_{\lrtr}}(U)\otimes_\inte \Lambda^\vee_{B_0}(U) & \text{if $\lrtr \in \lrtree{k}_0$, $P_{\lrtr} \neq \emptyset$ and $\exists i$ such that $\alpha_i\in \filt^{0}_{P_{e_i}}(U)$},\\
\mathfrak{l}_{k,\lrtr}(\vec{\alpha},\vec{n}) \in \filt^{0}_{P_{\lrtr}}(U)\otimes_\inte \Lambda^\vee_{B_0}(U) & \text{if $\lrtr \in \lrtree{k}_1$ and $P_{\lrtr} \neq \emptyset$},\\
\mathfrak{l}_{k,\lrtr}(\vec{\alpha},\vec{n}) \in \filt^{-\infty}_{1}(U)\otimes_\inte \Lambda^\vee_{B_0}(U) & \text{if $P_{\lrtr} = \emptyset$}.
\end{array}\right.
\end{equation*}
\end{lemma}

For the purpose of the induction argument used to prove Lemma \ref{tree_lemma}, we will temporarily relax the condition that $m_{e_i} = k m_i$ for $i = 1, 2$ on the labeling of the incoming edges $e_i$'s in Definition \ref{label_tree_def} and replace it by the condition that $m_{e_i} = k m_a$ for some $k > 0$ and some $a \in (\inte_{\geq 0})^2_{\text{prim}}$.

\begin{proof}
We prove by induction on the number of vertices of a labeled ribbon $k$-tree $\lrtr$. The initial step is trivial because $\lrtree{1} = \lrtree{1}_0$ and $\mathfrak{l}_{1,0}$ is the identity.

We illustrate the induction step by considering the simplest non-trivial case, namely, when we have a labeled ribbon $2$-tree $\lrtr$ with only one trivalent vertex $v$, two incoming edges $e_1, e_2$ and one outgoing edge $e_o$ meeting $v$. Suppose that the incoming edges $e_1, e_2$ are assigned labeling $(m_{e_1},j_{e_1}), (m_{e_2},j_{e_2})$ and inputs $\alpha_1 \check{\partial}_{n_1}\bmc^{e_1} t^{j_{e_1}}, \alpha_2 \check{\partial}_{n_2}\bmc^{e_2} t^{j_{e_2}}$ respectively.

If $\lrtr \in \lrtree{k}_0$ (i.e. with labeling $\natural$ at every $v \in \lrtr^{[0]}$) and $P_{\lrtr} \neq \emptyset$, then we have
\begin{equation*}
\mathfrak{l}_{k,\lrtr}(\vec{\alpha},\vec{n}) = \mathfrak{l}_{k,\lrtr}(\alpha_1 \check{\partial}_{n_1},\alpha_2 \check{\partial}_{n_2}) = -H_{m_{\lrtr}}(\alpha_1\wedge \alpha_2)  \check{\partial}_{n_{\lrtr}},
\end{equation*}
where $n_{\lrtr} = ( m_{e_2},n_1 ) n_2 - ( m_{e_1} ,n_2 ) n_1$ is given by the formula \eqref{vertex_lie_algebra} (here we are viewing $H_{m_{\lrtr}}$ as an operator on $\filt^\infty_*(U)$ as in Definition \ref{MC_homotopy}).

The first case is when $\alpha_1, \alpha_2 \in \filt^{1}_{P_{e_i}}(U)$.
Since $P_{\lrtr} \neq \emptyset$, the walls $P_{e_1}$ and $P_{e_2}$ are intersecting transversally at $Q$, so we have $\alpha_1 \wedge \alpha_2 \in \filt^{2}_{P_{e_1}\cap P_{e_2}}(U) = \filt^{2}_{Q}(U)$ by Lemma \ref{support_product}.
Recall the decomposition $H_{m_{\lrtr}} = I_{m_{\lrtr}} + I_{m_\lrtr,er}$ in Definition \ref{MC_homotopy}. For the second integral $I_{m_\lrtr,er}$, its domain of integration lies inside the hyperplane $U^\perp_{m_{\lrtr}}$ which does not intersect $Q$ by our choice of the spherical neighborhood $U$ in Notations \ref{choice_of_U} (see \eqref{eqn:Q_intersect_U_a} and Figure \ref{fig:spherical_nbh}), so it produces terms in $\filt^{-\infty}_*(U)$.
For the first integral $I_{m_\lrtr}$, applying Lemma \ref{integral_lemma} gives $-I_{m_\lrtr}(\alpha_1\wedge \alpha_2) \in \filt^{1}_{P_{\lrtr}}(U)$, where $P_{\lrtr} = (Q - \real_{\geq 0 }m_{\lrtr}) \cap U$ as described in notation \ref{P_edge}. This proves the first case.

For the second case, either $\alpha_1 \in \filt^0_{P_{e_1}}(U)$ or $\alpha_2 \in \filt^0_{P_{e_1}}(U)$, so we have $\alpha_1 \wedge \alpha_2 \in \filt^{1}_Q(U)$ by Lemma \ref{support_product}. The rest of the argument is the same as in the first case.

In the third case, we have $\lrtr \in \lrtree{k}_1$ and $P_{\lrtr} \neq \emptyset$. This means that either $\sharp$ or $\flat$ is applied at $v$; we will only give the proof for the case when $\sharp$ is applied because the other case is similar. In such a case, we have
$
\mathfrak{l}_{k,\lrtr}(\alpha_1 \check{\partial}_{n_1},\alpha_2 \check{\partial}_{n_2}) = -H_{m_{\lrtr}}(\alpha_1 \wedge (\nabla_{\partial_{n_1}} \alpha_2)) \check{\partial}_{n_2}.
$
Now $\nabla_{\partial_{n_1}} (\alpha_2) \in \filt^{0}_{P_{e_2}}(U)$ by \eqref{asy_support_basic_property} and \eqref{affine_vector_field}, so we get $\alpha_1 \wedge (\nabla_{\partial_{n_1}} \alpha_2) \in \filt^{1}_{Q}(U)$ and the rest of the proof is same as in the first case.

Finally, for the fourth case we have $P_{\lrtr} = \emptyset$, meaning that $P_{e_1}$ and $P_{e_2}$ not intersecting transversally. Then we have $\alpha_1 \wedge \alpha_2 \in \filt^{-\infty}_2(U)$ by Lemma \ref{support_product} and the integral operator $H_{m_{\lrtr}}$ preserves $\filt^{-\infty}_*(U)$ by its definition in Definition \ref{exponential_decay}. This completes the proof of labeled ribbon $2$-tree.

Next, suppose that we have a general labeled ribbon $k$-tree $\lrtr$, and $v_r \in \lrtr^{[0]}$ is the unique trivalent vertex adjacent to the unique outgoing edge $e_o$.
Assuming that $\acute{e}_1$ and $\acute{e}_2$ are the incoming edges connecting to $v_r$ so that the edges $\acute{e}_1, \acute{e}_2, e_o$ are arranged in clockwise orientation.
We split $\lrtr$ at $v_r$ to obtain two trees $\lrtr_1, \lrtr_2$ with outgoing edges $\acute{e}_1, \acute{e}_2$ and $k_1, k_2$ incoming edges respectively such that $k = k_1 + k_2$
We split the inputs $(\vec{\alpha},\vec{n})$ into two accordingly as $\vec{\alpha}_1 = (\alpha_1, \dots, \alpha_{k_1})$, $\vec{n}_1 = (n_1, \dots, n_{k_1})$ and $\vec{\alpha}_2 = (\alpha_{k_1+1}, \dots, \alpha_{k})$, $\vec{n}_2 = (n_{k_1+1}, \dots, n_{k})$. We then consider the operation $\mathfrak{l}_{k_i,\lrtr_i}(\vec{\alpha}_i,\vec{n}_i)$ associated to each $\lrtr_i$.

If one of $P_{\lrtr_i}$'s is empty, say, if $P_{\lrtr_1} = \emptyset$, then $\mathfrak{l}_{k_1,\lrtr_1}(\vec{\alpha}_1,\vec{n}_1) \in \filt^{-\infty}_1(U)\otimes_{\inte} \Lambda_{B_0}^\vee(U)$ by the induction hypothesis. Hence we also have $\mathfrak{l}_{k,\lrtr}(\vec{\alpha},\vec{n}) \in \filt^{-\infty}_{1}(U)\otimes_{\inte} \Lambda_{B_0}^\vee(U)$ since $\filt^{-\infty}_*(U)$ is a dg-Lie ideal of $\filt^{\infty}_*(U)$ and $H_{m}$ preserves $\filt^{-\infty}_*(U)$.

So it remains to consider the case when $P_{\lrtr_i} \neq \emptyset$ for $i = 1, 2$. Note that
$\lrtr \in \lrtree{k}_0$ if and only if $\lrtr_i \in \lrtree{k_i}_0$ for both $i = 1, 2$ and the labeling of the root vertex $v_r$ is also $\natural$, and
$P_{\lrtr} \neq \emptyset$ if and only if $P_{\acute{e}_1}$ intersects $P_{\acute{e}_2}$ transversally at $Q$.
The induction step is completed by replacing $\alpha_i \check{\partial}_{n_i}$ with $\mathfrak{l}_{k_i,\lrtr_i}(\vec{\alpha}_i,\vec{n}_i)$ for $i = 1, 2$ and using the same argument as in the case for labeled ribbon $2$-tree.

\end{proof}

\begin{proof}[Proof of Lemma \ref{approximated_MC}]
Recall that the spherical neighborhood $U$ was chosen so that $x^0 \not\in P_a$ for any $a \in (\inte_{\geq 0 })^2_{\text{prim}}$, where $x^0$ is the center of $U$ (see Notations \ref{choice_of_U}, \eqref{eqn:Q_intersect_U_a} and Figure \ref{fig:spherical_nbh}). By Theorem \ref{asy_support_theorem}, we have a neighborhood $V \subset U$ of $x^0$ such that $\overline{V}$ is compact and $\Phi|_{V} \in \mathcal{E}^*_N(V)$ (mod $\mathbf{m}^{N+1}$) for every $N \in \inte_{>0}$. This implies that $[\Phi,\Phi]|_{V} \in \mathcal{E}^*_N(V)$ (mod $\mathbf{m}^{N+1}$) for every $N \in \inte_{>0}$ since $\mathcal{E}^*_N(V)$ is closed under the bracket $[\cdot,\cdot]$. As the operator $\mathcal{P}$ preserves $\mathcal{E}^*_N(V)$, we have $\mathcal{P}[\Phi,\Phi]|_{V} \in \mathcal{E}^*_N(V)$ (mod $\mathbf{m}^{N+1}$) for every $N \in \inte_{>0}$. But then this means that $\mathcal{P}[\Phi,\Phi] = 0$ in $\widehat{\mathbf{g}^*/\mathcal{E}^*}(U)$ since $\mathcal{P}$ is the evaluation at $x^0 \in V \subset U$.
\end{proof}

\subsubsection{Leading $\hp$ order terms of a Maurer-Cartan solution}\label{sec:tropical_leading_order}
Recall from Section \ref{sec:leading_order_MC} that the leading order term $\emc$ (constructed in Definition \ref{solving_MC_tree}) is a sum over labeled ribbon trees in $\lrtree{k}_0$ (i.e. those with only $\natural$ labeling on trivalent vertices; see Notation \ref{two_sets_of_trees}) with inputs $\eincoming^{(i)}$ (defined in \eqref{essential_incoming}). This operation is closely related to the tropical vertex group as well as tropical counting. We are going to discuss the precise correspondence in this subsection.

For this purpose, it would conceptually be more appropriate to use labeled trees rather than labeled ribbon trees, because tropical trees are not equipped with ribbon structures. As in the case of labeled ribbon $k$-trees, we split the set of isomorphism classes of labeled $k$-trees into two components
$\ltree{k} = \ltree{k}_0 \sqcup \ltree{k}_1,$
where $\ltree{k}_0$ consists of those whose trivalent vertices are all labeled by $\natural$ and $\ltree{k}_1 := \ltree{k} \setminus \ltree{k}_0$.
Then given a labeled $k$-tree $\ltr \in \ltree{k}_0$ and taking an arbitrary labeled ribbon $k$-tree $\lrtr$ with $\underline{\lrtr} = \ltr$, we can define the operation $\mathfrak{I}_{k,\ltr}$ by
$
\mathfrak{I}_{k,\ltr}(\zeta_1, \dots, \zeta_k) := \sum_{\sigma \in \Sigma_k} (-1)^{\chi(\sigma,\vec{\zeta})} \mathfrak{l}_{k,\lrtr}(\zeta_{\sigma(1)},\dots,\zeta_{\sigma(k)}),
$
as in Definition \ref{def:tree_operation}.

Nonetheless, we prefer to work with labeled ribbon $k$-trees $\lrtr$ instead to simplify the formulas.
There is a combinatorial relation
$
\frac{1}{k!|\text{Aut}(\ltr)|}\mathfrak{I}_{k,\ltr}(\eincoming, \dots, \eincoming) = \sum_{\underline{\lrtr} = \ltr} \frac{1}{2^{k-1}} \mathfrak{l}_{k,\lrtr}(\eincoming,\dots,\eincoming),
$
where the sum is over all labeled ribbon $k$-trees $\lrtr$ with underlying labeled $k$-tree $\underline{\lrtr} = \ltr$ and $|\text{Aut}(\ltr)|$ is the order of the automorphism group of $\ltr$.
Furthermore, since $\eincoming \in \widehat{\mathbf{g}^*/\mathcal{E}^*}(U)$ has cohomological degree $1$, we have
\begin{equation}\label{eqn:relating_ribbon_or_not}
\frac{1}{k!|Aut(\ltr)|}\mathfrak{I}_{k,\ltr}(\eincoming,\dots,\eincoming) = \frac{\#\{\underline{\lrtr} = \ltr\}}{2^{k-1}} \mathfrak{l}_{k,\lrtr_0}(\eincoming,\dots,\eincoming),
\end{equation}
for an {\em arbitrary} labeled ribbon $k$-tree $\lrtr_0$ with underlying labeled $k$-tree $\underline{\lrtr_0} = \ltr$.
We will obtain results for $\mathfrak{I}_{k,\ltr}$ by working with $\mathfrak{l}_{k,\lrtr}$ through equation \eqref{eqn:relating_ribbon_or_not} in this section.

Given $\lrtr \in \lrtree{k}_0$ with $P_{\lrtr} \neq \emptyset$ (recall from Notations \ref{P_edge} that $P_{\lrtr}$ is the wall attached to the unique outgoing edge $e_o$), we have an alternative way to describe the operation $\mathfrak{l}_{k,\lrtr}$. Recall that $\lrtr^{[1]}$ is the set of edges excluding the incoming edges (but including the outgoing edge) by Definition \ref{k_tree_def}. We let $\tau^e_s$ be the flow of the affine vector field $-m_e$ for time $s$ ($s \in \real_{\leq 0}$ so it is flowing backward in time), where $(m_e, j_e)$ is the labeling of the edge $e \in \lrtr^{[0]}$.

\begin{definition}\label{sequence_flow_def}
Given a sequence of edges $\mathfrak{e} = (e_0, e_{1}, \ldots, e_{l})$, as a path which starts from $e_0$ and ends at $e_l$ following the direction of the tree $\lrtr$, we define a map
$
\tau^{\mathfrak{e}}: W_{\mathfrak{e}} \rightarrow U,
$
by
$
\tau^{\mathfrak{e}} (\vec{s},x) = \tau_{s_0}^{e_0} \circ \tau_{s_1}^{e_1} \circ \cdots \circ \tau_{s_l}^{e_l} (x),
$
where $s_{j}$ is the time coordinate for the flow of $-m_{e_{j}}$, $\lrtr^{[1]}_{\mathfrak{e}}$ is the subset $\{e_0, e_1, \ldots, e_l\} \subset \lrtr^{[1]}$ and $W_{\mathfrak{e}}\subset \real_{\leq 0}^{|\lrtr^{[1]}_{\mathfrak{e}}|} \times U$ is the maximal domain such that the image of the flow $\tau^{\mathfrak{e}}$ lies in $U$. It can also be extended naturally to a map
$
\hat{\tau}^\mathfrak{e}: \hat{W}_{\mathfrak{e}} \rightarrow \real_{\leq 0}^{|\lrtr^{[1]} \setminus \lrtr^{[1]}_{\mathfrak{e}}|} \times U,
$
where $\hat{W}_{\mathfrak{e}} := \real_{\leq 0}^{|\lrtr^{[1]} \setminus \lrtr^{[1]}_{\mathfrak{e}}|} \times W_{\mathfrak{e}}$, by taking direct product with $\real_{\leq 0}^{|\lrtr^{[1]} \setminus \lrtr^{[1]}_{\mathfrak{e}}|}$.
Notice that this definition does not depend on the ribbon structure on $\lrtr$, so it can be regarded as a definition for a labeled $k$-tree $\ltr := \underline{\lrtr}$.
\end{definition}


\begin{definition}\label{def:trop_tree_orientation}
We attach a differential form $\nu_e$ on $\real_{\leq 0}^{|\lrtr^{[1]}|}$ to each $e \in \bar{\lrtr}^{[1]}$ recursively by letting $\nu_e := 1$ for each incoming edge $e$, and $\nu_{e_3} =  (-1)^{\bar{\nu}_{e_2}} \nu_{e_1} \wedge \nu_{e_2} \wedge ds_{e_{3}}$ (here $\bar{\nu}_{e_2}$ is the cohomological degree of $\nu_{e_2}$) if $v$ is an internal vertex with incoming edges $e_1,e_2 \in \lrtr_0$ and outgoing edge $e_3$ such that $e_1, e_2, e_3$ is clockwise oriented. We let $\nu_{\lrtr}$ be the differential form attached to the unique outgoing edge $e_o \in \lrtr^{[1]}$, which defines a volume form or orientation on $\real_{\leq 0}^{|\lrtr^{[1]}|}$.
\end{definition}

As usual, we let $v_1,\dots,v_k$ be the clockwise ordered incoming vertices of $\lrtr$ and $e_1,\dots,e_k$ the incoming edges respectively. We associate to each $e_i$ a unique sequence $\mathfrak{e}_i$ of edges in $\lrtr^{[1]}$ (excluding the incoming edge $e_i$ itself) joining $e_i$ to the outgoing edge $e_o$ (including the outgoing edge $e_o$) along the direction of $\lrtr$.

\begin{remark}\label{rem:tropical_moduli}
The subset
$W_{\lrtr}  = \big(\bigcap_{i=1}^k \hat{W}_{\mathfrak{e}_i}\big) \subset \real_{\leq 0}^{|\lrtr^{[1]}|} \times U,$
can be viewed as a moduli space of tropical trees in $U$, denoted by $\mathfrak{M}_{\underline{\lrtr}}(U)$, with prescribed slope data $\{m_{e}\}_{e \in \lrtr^{[1]}}$ (as in Notations \ref{not:edge_labeling}) as follows (this will not be necessary for the rest of the paper): A point $(\vec{s},x) \in W_{\lrtr} \subset \real_{\leq 0}^{|\lrtr^{[1]}|} \times U$ will precribe the location of the vertices $\lrtr^{[0]}\sqcup \{v_o\}$ of $\lrtr$. First, $x \in U$ is the image of the outgoing vertex $v_o$.
For any trivalent vertex $v \in \lrtr^{[0]}$, there is a unique sequence of edges $\mathfrak{e} = (e_0, e_1, \ldots, e_l)$ connecting $v$ to $v_o$ and $\tau^{\mathfrak{e}} (\vec{s},x)$ is the image of $v$. The images of these vertices are allowed to overlap with each other as $\vec{s}$ is taken from $\real_{\leq 0 }^{\lrtr^{[1]}}$. Figure \ref{fig:tropical_moduli} illustrates the generic situation.

\begin{figure}[h]
\centering
\includegraphics[scale=0.4]{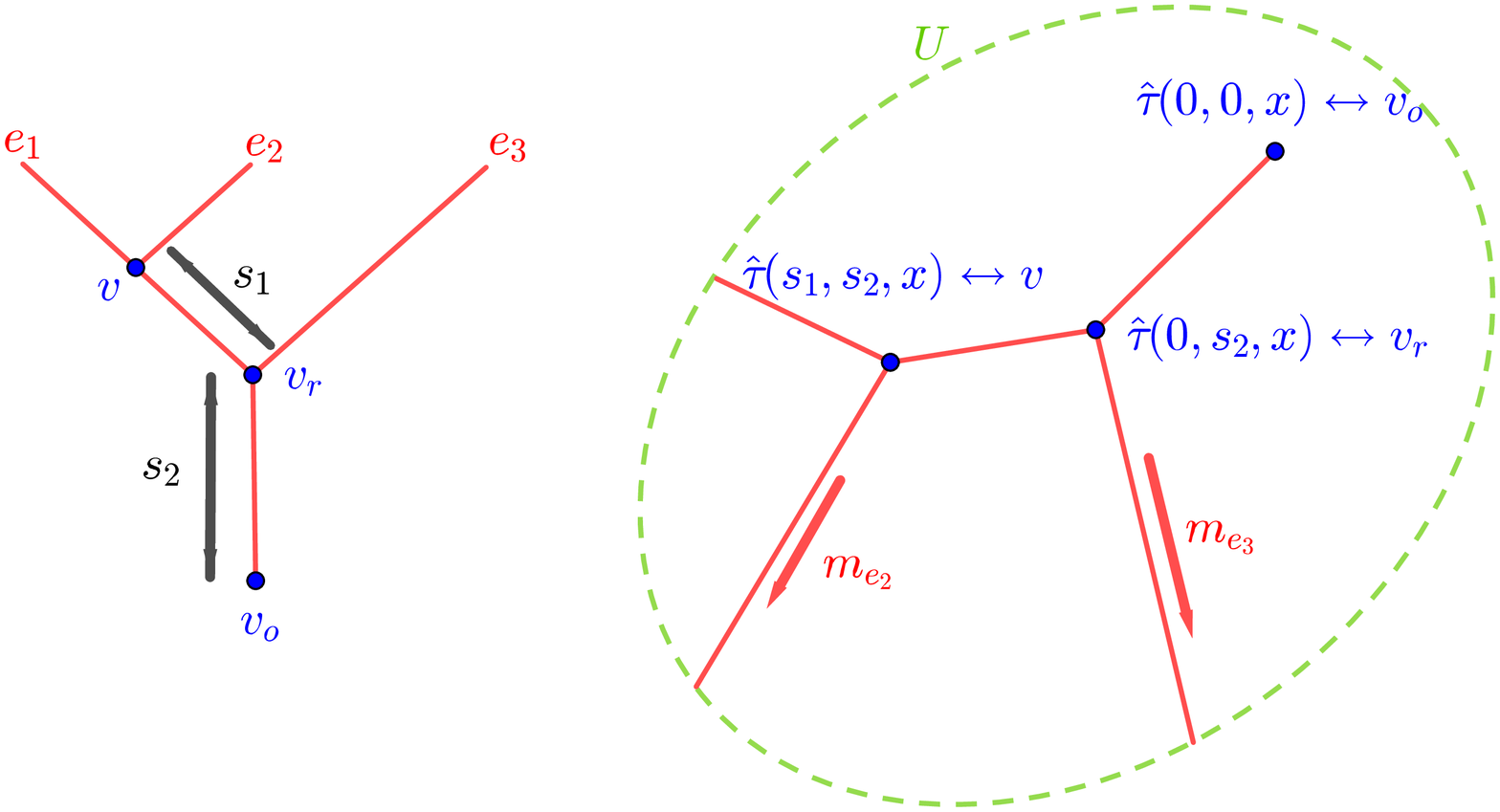}
\caption{$W_{\lrtr}$ parametrizing tropical trees in $U$}
\label{fig:tropical_moduli}
\end{figure}
\end{remark}

Recall that Definition \ref{color_tree_operation}, which defines the operator $\mathfrak{l}_{k,\lrtr}$, uses the labeling $(m_{e_i},j_{e_i})$ for each incoming edge $e \in \partial^{-1}_{in}(\lrtr^{[0]}_{in})$ to extract the coefficient of $\bmc^{m_{e_i}} t^{j_{e_i}}$ in $\eincoming$ and then treat it as the input at $v_i$. For the input $\alpha_{i}\check{\partial}_{n_{i}}$, we have $n_{i} \in TP_{e_i}^\perp$ and $(\nu_{P_i},n_i)>0$, and $\alpha_{i} = -\delta_{jk}^{(1)}$ or $-\delta_{jk}^{(2)}$ (see \eqref{essential_incoming}) so that $\alpha_i \in \filt^1_{P_{e_i}}(U)$. We decompose the output $\mathfrak{l}_{k,\lrtr}(\vec{\alpha},\vec{n})$ (defined in equation \eqref{label_tree_operation}) into a differential form part $\alpha_{\lrtr} \in \filt^{1}_{P_{\lrtr}}(U)$ and a vector field part $\check{\partial}_{n_{\lrtr}}$:
\begin{equation}\label{alpha_out_def}
\mathfrak{l}_{k,\lrtr} (\vec{\alpha},\vec{n}) = \alpha_{\lrtr} \check{\partial}_{n_{\lrtr}},
\end{equation}
where $\check{\partial}_{n_\lrtr}:= \mathfrak{l}_{k,\lrtr} (\check{\partial}_{n_1}, \ldots, \check{\partial}_{n_k})$; note that $n_\lrtr \in \inte \cdot n_a$ if we write $P_\lrtr = P_a$ for some $a \in (\inte_{\geq 0})^2_{\text{prim}}$, where $n_a \in \Lambda_{B_0}^\vee(U)$ is the unique primitive normal to $P_\lrtr = P_a$ such that $(\nu_{P_a},n_a) < 0$.
The following lemma shows how $\alpha_{\lrtr}(x)$ can be expressed as an integral over the space
$\mathcal{I}_x := \left(\real_{\leq 0}^{|\lrtr^{[1]}|} \times \{x\}\right) \cap \left(\bigcap_{i=1}^k \hat{W}_{\mathfrak{e}_i}\right),$
for any $x \in U$ up to error terms of exponential order in $\hp^{-1}$.

\begin{lemma}\label{iteratedintegral}
We have the identity
$$
\alpha_{\lrtr}(x) = (-1)^{k-1}\int_{\mathcal{I}_x} (\tau^{\mathfrak{e}_1})^*(\alpha_1) \wedge \cdots \wedge (\tau^{\mathfrak{e}_k})^*(\alpha_k)
$$
in $\filt^{\infty}_1(U)/\filt^{-\infty}_1(U)$, where we use the volume form $\nu_{\lrtr}$ on $\real_{\leq 0}^{|\lrtr^{[1]}|}$ for the integration on the right hand side.
\end{lemma}

\begin{proof}
We prove the lemma by induction on the number of vertices of a labeled ribbon $k$-tree $\lrtr$ (as in the proof of Lemma \ref{tree_lemma}). In the initial step, $\mathcal{I}_x$ is just the point $\{x\}$ and the right hand side is nothing but evaluation at $x$, so the result follows from the fact that $\mathfrak{l}_{k,\lrtr}$ is the identity.

As in the proof of Lemma \ref{tree_lemma}, we illustrate the induction step by considering the simplest non-trivial case when we are given a labeled ribbon $2$-tree $\lrtr$ with only one trivalent vertex $v$, two incoming edges $e_1, e_2$ and one outgoing edge $e_o$ meeting $v$.
Suppose that the incoming edges $e_1, e_2$ are assigned labeling $(m_{e_1},j_{e_1}), (m_{e_2},j_{e_2})$ and inputs $\alpha_1 \check{\partial}_{n_1}\bmc^{e_1} t^{j_{e_1}}, \alpha_2 \check{\partial}_{n_2}\bmc^{e_2} t^{j_{e_2}}$ respectively.
The operator $\mathfrak{l}_{k,\lrtr}$ associated to $\lrtr$ is explicitly expressed as
\begin{equation}
\alpha_{\lrtr}\check{\partial}_{n_{\lrtr}} = \mathfrak{l}_{k,\lrtr}(\alpha_1 \check{\partial}_{n_1},\alpha_2 \check{\partial}_{n_2}) = -H_{m_{\lrtr}}(\alpha_1\wedge \alpha_2)  \check{\partial}_{n_{\lrtr}},
\end{equation}
with $n_{\lrtr} = ( m_{e_2},n_1 ) n_2 - ( m_{e_1} ,n_2 ) n_1$ given by the formula \eqref{vertex_lie_algebra}.

There are two cases depending on whether $P_{e_1}$ and $P_{e_2}$ are intersecting transversally or not, as in proof of Lemma \ref{tree_lemma}. In both cases we can treat $\alpha_1\wedge \alpha_2 \in \filt^{2}_Q(U)$ (because if intersection is not transversal, then we have $\alpha_1\wedge \alpha_2 \in \filt^{-\infty}_2(U) \subset \filt^2_Q(U)$).
From Definition \ref{MC_homotopy}, we have the decomposition $H_{m_{\lrtr}} = I_{m_\lrtr} + I_{m_\lrtr,er}$. By our choice of the spherical neighborhood $U$ in Notations \ref{choice_of_U}, the domain of integration for the second integral $I_{m_\lrtr,er}$ is supported away from $Q$, so it gives a term in $\filt^{-\infty}_1(U)$. Thus we have $\alpha_\lrtr \in -I_{m_\lrtr}(\alpha_1 \wedge \alpha_2) + \filt^{-\infty}_1(U)$.

In the current case, $\lrtr^{[1]}$ consists of the unique outgoing edge $e_o$, so $\mathfrak{e}_1 = \mathfrak{e}_2 = (e_o)$ and hence the map $\tau^{\mathfrak{e}_1} = \tau^{\mathfrak{e}_2} : W_{\mathfrak{e}_i} \rightarrow U$ is simply given by the flow associated to $-m_{\lrtr} = -m_{e_o}$. Now we have the half space $U_{m_{\lrtr}}^+$ containing $Q$, as shown in Figure \ref{fig:spherical_nbh}. Using the coordinates $(t,u^\perp)$ on $U$ where $u^\perp$ are local affine coordinates on $U_{m_\lrtr}^\perp$, we obtain a maximal interval $(a_{u^\perp}, b_{u^\perp})$ for each point ${u^\perp}$ on $U_{m_\lrtr}^\perp$ such that the interval $(a_{u^\perp}, b_{u^\perp}) \times \{u^\perp\}$ lies in $U$. Then the two integrals that we want to compare are $I_m(\alpha_1 \wedge \alpha_2)(t,u^\perp)  = \int_{0}^t \iota_{\dd{s}}(\alpha_1 \wedge \alpha_2)(s,u^\perp) ds$ and $\int_{\mathcal{I}_{(t,u^\perp)}} (\tau^{\mathfrak{e}_1})^*(\alpha_1 \wedge \alpha_2)  = \int_{a_{u^\perp}}^t \iota_{\dd{s}}(\alpha_1 \wedge \alpha_2)(s,u^\perp) ds$,
the difference between which is given by $\int_{a_{u^\perp}}^0 \iota_{\dd{s}}(\alpha_1 \wedge \alpha_2)(s,u_{m_\lrtr}^\perp) ds$, which produces a term in $\filt^{-\infty}_1(U)$ because it misses the asymptotic support $Q$ of $\alpha_1 \wedge \alpha_2$. This proves the statement for the current case.

Next we consider the induction step. We will adapt the same notations as the induction step in the proof of Lemma \ref{tree_lemma} (see Notations \ref{P_edge}). So we take a general labeled ribbon $k$-tree $\lrtr \in \lrtree{k}_0$, and then split it at the unique vertex $v_r \in \lrtr^{[0]}$ adjacent to the unique outgoing edge $e_o$ to obtain two trees $\lrtr_1, \lrtr_2$ with incoming edges $e_1,\dots,e_{k_1}$ and $e_{k_1+1},\dots,e_{k}$ respectively.
We denote by $\tilde{\mathfrak{e}}_1 ,\dots,\tilde{\mathfrak{e}}_{k_1}$ (resp. $\tilde{\mathfrak{e}}_{k_1+1} ,\dots,\tilde{\mathfrak{e}}_{k}$) the sequences of edges in $\lrtr_1$ (resp. $\lrtr_2$) associated to the incoming edges $e_1,\dots,e_{k_1}$ (resp. $e_{k_1+1},\dots,e_{k}$) obtained respectively from the sequences $\mathfrak{e}_1,\dots,\mathfrak{e}_{k_1}$ (resp. $\mathfrak{e}_{k_1+1},\dots,\mathfrak{e}_{k}$) of edges in $\lrtr$ by removing the unique outgoing edge $e_o$.

By the induction hypothesis, we have
$
(-1)^{k_1-1}\alpha_{\lrtr_1}(x) = \int_{\mathcal{I}_{1,x}} (\tau^{\tilde{\mathfrak{e}}_1})^*(\alpha_1) \wedge \cdots \wedge (\tau^{\tilde{\mathfrak{e}}_{k_1}})^*(\alpha_{k_1})
$
modulo $\filt^{-\infty}_1(U)$, where $\mathcal{I}_{1,x} = \left(\real_{\leq 0}^{|\lrtr_1^{[1]}|} \times \{x\}\right) \cap \left( \bigcap_{i=1}^{k_1} \hat{W}_{\mathfrak{e}_i}^{(1)} \right)$, and
$
(-1)^{k_2-1}\alpha_{\lrtr_2}(x) = \int_{\mathcal{I}_{2,x}} (\tau^{\tilde{\mathfrak{e}}_{k_1+1}})^*(\alpha_{k_1+1}) \wedge \cdots \wedge (\tau^{\tilde{\mathfrak{e}}_{k}})^*(\alpha_k)
$
where $\mathcal{I}_{2,x} = \left(\real_{\leq 0}^{|\lrtr_2^{[1]}|} \times \{x\}\right) \cap \left( \bigcap_{i=k_1+1}^k \hat{W}_{\mathfrak{e}_i}^{(2)} \right)$; here $\hat{W}_{\mathfrak{e}_i}^{(1)} \subset \real_{\leq 0}^{|\lrtr_1^{[1]}|} \times U$ is the domain associated to $\tilde{\mathfrak{e}}_i$ for the tree $\lrtr_1$ as in Definition \ref{sequence_flow_def}.

Fixing a point $x\in U$, we consider the flow $\tau^{e_o}: W_{e_o} (\subset \real \times U) \rightarrow U$ by $-m_{e_o} = -m_{\lrtr}$ and let $\mathcal{I}^{e_o}_x := (\real_{\leq 0} \times \{x\} )\cap W_{e_o}$. From its definition, we have, for $x \in U$,
$$
\mathcal{I}_x = \bigcup_{s \in \mathcal{I}^{e_o}_x} \mathcal{I}_{1,\tau^{e_o}_s(x)} \times \mathcal{I}_{2,\tau^{e_o}_s (x)} \times \{s\} \subset \real_{\leq 0}^{|\lrtr_1^{[1]}|} \times \real_{\leq 0}^{|\lrtr_2^{[1]}|} \times \real_{\leq 0}.
$$
Using the same reasoning as in the $2$-tree case, together with the fact that $\alpha_{\lrtr_1} \wedge \alpha_{\lrtr_2}$ is again having asymptotic support on $Q$, we have
\begin{align*}
&(-1)^{k-1}\alpha_{\lrtr}(x)
=  (-1)^{k-2}\int_{\mathcal{I}_{x}^{e_o}} (\tau^{e_o})^*(\alpha_{\lrtr_1} \wedge \alpha_{\lrtr_2})\\
= & \int_{\mathcal{I}_{x}^{e_o}} (\tau^{e_o})^* \Big( \int_{\mathcal{I}_{1,\tau^{e_o}_s(x)}} (\tau^{\tilde{\mathfrak{e}}_1})^*(\alpha_1) \cdots (\tau^{\tilde{\mathfrak{e}}_{k_1}})^*(\alpha_{k_1})
   \wedge \int_{\mathcal{I}_{2,\tau^{e_o}_s (x)}} (\tau^{\tilde{\mathfrak{e}}_{k_1+1}})^*(\alpha_{k_1+1}) \cdots (\tau^{\tilde{\mathfrak{e}}_{k}})^*(\alpha_k) \Big)\\
= &  \int_{\bigcup_{s \in \mathcal{I}^{e_o}_x} \mathcal{I}_{1,\tau^{e_o}_s(x)} \times \mathcal{I}_{2,\tau^{e_o}_s (x)} \times \{s\}} (\tau^{e_o})^* \left( (\tau^{\tilde{\mathfrak{e}}_1})^*(\alpha_1) \cdots (\tau^{\tilde{\mathfrak{e}}_{k}})^*(\alpha_k) \right)
=  \int_{\mathcal{I}_x} (\tau^{\mathfrak{e}_1})^*(\alpha_1) \cdots (\tau^{\mathfrak{e}_k})^*(\alpha_k),
\end{align*}
modulo terms in $\filt^{-\infty}_1(U)$.
This completes the proof of the lemma.
\end{proof}

\begin{remark}
Geometrically, Lemma \ref{iteratedintegral} means that terms of the form $\alpha_{\lrtr}\check{\partial}_{n_{\lrtr}} \bmc^{m_\lrtr} t^{j_{\lrtr}}$ for $\lrtr \in \lrtree{k}_0$, which appear in the leading order contribution of the Maurer-Cartan solution $\Phi$ (introduced in Definition \ref{solving_MC_tree}), can be expressed as integrals over the moduli space $\mathfrak{M}_{\underline{\lrtr}}(U)$ of tropical trees in $U$.
\end{remark}

We will see in Section \ref{sec:key_lemmas} that what we essentially care about is the path integral
\begin{equation}\label{tropical_integral}
\int_{\varrho} \alpha_{\lrtr} = (-1)^{k-1}\int_{\mathcal{I}_\varrho} (\tau^{\mathfrak{e}_1})^*(\alpha_1) \cdots (\tau^{\mathfrak{e}_k})^*(\alpha_k) + O(e^{-c_{\varrho}/\hp})
\end{equation}
along an embedded affine path $\varrho: (a,b) \rightarrow U$ that crosses the wall $P_{\lrtr}$ transversally and positively (meaning that $TP_{\lrtr} \oplus \real \cdot \varrho'$ agrees with the orientation of $B_0$),
where we let
\begin{equation}\label{eqn:varrho_tropical_space}
\mathcal{I}_\varrho := \bigcup_{t \in (a,b)} \mathcal{I}_{\varrho(t)}.
\end{equation}
Here $\mathcal{I}_\varrho $ is equipped with coordinates $(\{s_e\}_{e\in \lrtr^{[1]}},t)$, where $(s_e)_e \in \real_{\leq 0 }^{|\lrtr^{[1]}|}$ and $t \in (a,b)$. We are going to calculate the integral in \eqref{tropical_integral} explicitly.

Recall that $\alpha_i = -\delta_{jk}^{(1)}$ or $-\delta_{jk}^{(2)}$, and each $\alpha_i$ has asymptotic support on $P_{e_i}$ which is either $P_1$ or $P_2$ (supports of the two initial walls). For each $i \in \{1, \dots, k\}$, we take an affine coordinate $\eta_i$ associated to the corresponding $\delta^{(i)}_{jk}$ as in Definition \ref{input_assumption}, namely, such that $\{\eta_i = 0\} = P_{e_i}$ and $\iota_{\nu_{P_{e_i}}} d\eta_i >0$, so that we can write $\alpha_i =- ( \pi\hp)^{-1/2} e^{-\frac{\eta_i^2}{\hp}} d\eta_i$ locally near $P_{e_i}$.

The flow $\tau^{e_o}$ corresponding to the outgoing edge $e_o$ is an affine map
$\tau^{e_o}|_{\varrho}: \bigcup_{t \in (a,b)} \mathcal{I}_{\varrho(t)}^{e_o} \rightarrow U,$
where $\mathcal{I}_{\varrho(t)}^{e_o} \subset \real_{\leq 0}$ is the maximal domain of backward flow associated to $-m_{e_o}$ starting from the point $\varrho(t)$. The property that $\varrho \pitchfork P_{\lrtr} $ is equivalent to the condition that the image of $\tau^{e_o}|_{\varrho}$ intersects transversally with $Q$ at a point $q_1$. For each $i \in \{1, \dots, k\}$, we let $N_i\subset U$ be the affine line through the point $q_1 \in Q$ transversal to $P_{e_i}$. Then we consider the affine space $\prod_{i=1}^k N_i$ with local affine coordinates $\eta_1, \dots, \eta_k$. See Figure \ref{fig:high_dimensional_tropical_intersection}.

\begin{figure}[h]
\centering
\includegraphics[scale=0.3]{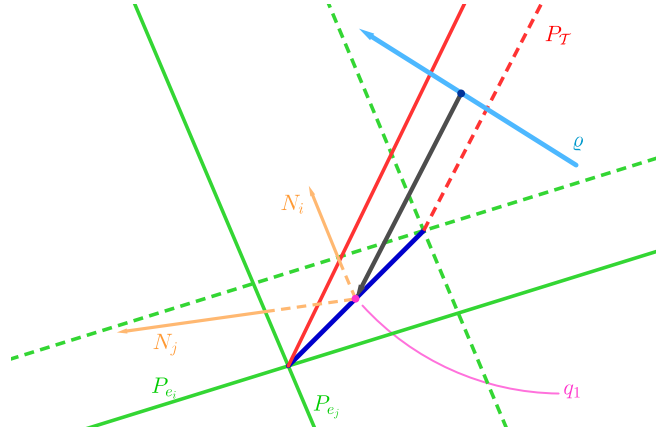}
\caption{The lines $N_i$'s through $q_1 \in Q$}
\label{fig:high_dimensional_tropical_intersection}
\end{figure}

We define the affine map
\begin{equation}\label{tau_map}
\vec{\tau} : \mathcal{I}_\varrho \rightarrow \prod_{i=1}^k N_i
\end{equation}
by requiring that $(\vec{\tau})^*(\eta_{i}) = \eta_{i} (\tau^{\mathfrak{e}_i}(\vec{s},x))$, where $\tau^{\mathfrak{e}_i}$ is defined in Definition \ref{sequence_flow_def}.

\begin{lemma}\label{lem:tau_map_iso}
There exists some constant $c>0$ such that
$
(\vec{\tau})^*(d\eta_{1} \wedge \cdots \wedge d\eta_{k}) = c (-1)^{\chi(\lrtr)} \nu_{\lrtr} \wedge dt,
$
where we set $(-1)^{\chi(\lrtr)} := \prod_{v \in \lrtr^{[0]}} (-1)^{\chi(\lrtr,v)}$ (with the convention that $(-1)^{\chi(\lrtr)} = 1$ if $\lrtr^{[0]} = \emptyset$) and $(-1)^{\chi(\lrtr,v)}$ is defined for each trivalent vertex $v$ (attached to two incoming edges $e_1, e_2$ and one outgoing edge $e_3$ so that $e_1,e_2,e_3$ are arranged in the clockwise orientation) by comparing the orientation of the ordered basis $\{-m_{e_1}, -m_{e_2}\}$ with that of $\{-m_1, -m_2\}$ of $NQ$ (cf. Notations \ref{choice_of_U}). In particular, $\vec{\tau}$ is an affine isomorphism onto its image $C(\vec{\tau}) \subset  \prod_{i=1}^k N_i$.
\end{lemma}

\begin{proof}
Once again, we will prove by induction on the number of vertices of the labeled ribbon tree $\lrtr$. The initial step concerning labeled ribbon $1$-trees is trivial because in this case $\mathcal{I}_\varrho = (a, b)$ and $\vec{\tau} = \varrho$.

As before, we will consider the next step, or the simplest nontrivial case, namely, when we are given a labeled ribbon $2$-tree $\lrtr$ with only one trivalent vertex $v$, two incoming edges $e_1, e_2$ and one outgoing edge $e_o$ meeting $v$, to illustrate the induction step.

We first assume that the orientation of $\{-m_{e_1},-m_{e_2}\}$ agrees with that of $\{-m_1,-m_2\}$. We can treat $\eta_1, \eta_2$ as oriented affine linear coordinates on the fiber $NQ_{q_1}$ of the normal bundle $NQ$. We will use $(s,t)$ for the coordinates of $\mathcal{I}_{\varrho} = \mathcal{I}_{\varrho}^{e_o}$ defined in  \eqref{eqn:varrho_tropical_space}.
Let $x_0 := \vec{\tau}(0,t_0) \in \varrho\cap P_\lrtr$ be the unique intersection point between $\varrho$ and $P_\lrtr$. Then there exists a unique $ s_0 \leq 0$ such that $q_1 = \vec{\tau}(s_0,t_0) \in P_{e_1} \cap P_{e_2}$. We see that $\{(d\vec{\tau})_{(s_0,t_0)}(\dd{s}) = -m_\lrtr \neq 0, (d\vec{\tau})_{(s_0,t_0)}(\dd{t})\}$ is an oriented basis of $NQ_{q_1}$ by the assumption that $\varrho$ intersects positively with $P_\lrtr$; in other words, $(d\vec{\tau})_{(s_0,t_0)}$ is a linear isomorphism from $T_{(s_0,t_0)} \mathcal{I}_\varrho$ onto $NQ_{q_1}$ and we have $(\vec{\tau})^*(d\eta_1\wedge d \eta_2) = c ds \wedge dt$ for some $c>0$.

In the opposite case when the orientation of $\{-m_{e_1},-m_{e_2}\}$ disagrees with that of $\{-m_1,-m_2\}$, $\eta_2,\eta_1$ are oriented coordinates of $NQ_{q_1}$. So we get $(\vec{\tau})^*(d\eta_1\wedge d \eta_2) = - c ds \wedge dt = c(-1)^{\chi(\lrtr)} ds \wedge dt$ for some $c>0$, because $\chi(\lrtr) = 1$ in this case.

For the induction step, we again split a general $k$-tree $\lrtr \in \lrtree{k}_0$ at the root vertex $v_r$ to get two trees $\lrtr_1$ and $\lrtr_2$, as in the proof of Lemma \ref{tree_lemma}. Since $P_\lrtr \neq \emptyset$, both $P_{\lrtr_1}$ and $P_{\lrtr_2}$ are non-empty and they intersect transversally.
We take two embedded paths $\varrho_1$ and $\varrho_2$ intersecting positively with $P_{\lrtr_1}$ and $P_{\lrtr_2}$ at $q_1$ with coordinates $\eta_{\lrtr_1}$ and $\eta_{\lrtr_2}$ respectively. By the induction hypothesis, the forms
$(\tau^{\tilde{\mathfrak{e}}_1})^*(d\eta_{1}) \wedge \cdots \wedge (\tau^{\tilde{\mathfrak{e}}_{k_1}})^*(d\eta_{k_1}) = c_1 (-1)^{\chi(\lrtr_1)} \nu_{\lrtr_1} \wedge d\eta_{\lrtr_1}$
and
$(\tau^{\tilde{\mathfrak{e}}_{k_1+1}})^*(d\eta_{k_1+1}) \wedge \cdots \wedge (\tau^{\tilde{\mathfrak{e}}_k})^*(d\eta_{k}) = c_2 (-1)^{\chi(\lrtr_2)}\nu_{\lrtr_2} \wedge d\eta_{\lrtr_2}$
are non-degenerate on $\mathcal{I}_{\varrho_1}$ and $\mathcal{I}_{\varrho_2}$ respectively. Therefore we have a nontrivial product
$(\tau^{\tilde{\mathfrak{e}}_1})^*(d\eta_{1}) \wedge \cdots \wedge (\tau^{\tilde{\mathfrak{e}}_{k}})^*(d\eta_{k}) = (-1)^{\chi(\lrtr_1)+\chi(\lrtr_2)+\bar{\nu}_{\lrtr_2}} \nu_{\lrtr_1}\wedge\nu_{\lrtr_2} \wedge d\eta_{\lrtr_1}\wedge d\eta_{\lrtr_2}$
on $\mathcal{I}_{\varrho_1} \times \mathcal{I}_{\varrho_2}$, where $\bar{\nu}_{\lrtr_2}$ denotes the degree of the differential form $\nu_{\lrtr_2}$.

Assuming that the orientation of $\{-m_{P_{\lrtr_1}},-m_{P_{\lrtr_1}}\}$ agrees with that of $\{-m_1,-m_2\}$ on $NQ_{q_1}$, we can treat $\{\eta_{\lrtr_1},\eta_{\lrtr_2}\}$ as an oriented basis for $NQ_{q_1}$. Using the same argument as in the $2$-tree case, we use $(s,t)$ as coordinates for $\bigcup_{t\in (a,b)}\mathcal{I}^{e_o}_{\varrho(t)}$ and obtain the relation $(\tau^{e_o}|_{\varrho})^*(d\eta_{\lrtr_1} \wedge d \eta_{\lrtr_2}) = c ds \wedge dt$ for some $c>0$. Combining with the induction hypothesis, we get
\begin{align*}
&(\tau^{\mathfrak{e}_1})^*(d\eta_{1}) \wedge \cdots \wedge (\tau^{\mathfrak{e}_k})^*(d\eta_{k})
 = (\tau^{e_o}|_{\varrho})^* \left( (\tau^{\tilde{\mathfrak{e}}_1})^*(d\eta_1) \wedge \cdots \wedge  (\tau^{\tilde{\mathfrak{e}}_{k}})^*(d\eta_k) \right)\\
=& (-1)^{\chi(\lrtr_1)+\chi(\lrtr_2)+\bar{\nu}_{\lrtr_2}} (\tau^{e_o}|_{\varrho})^*\big( \nu_{\lrtr_1}\wedge\nu_{\lrtr_2} \wedge d\eta_{\lrtr_1}\wedge d\eta_{\lrtr_2}\big)\\
 = &(-1)^{\chi(\lrtr_1)+\chi(\lrtr_2) +\bar{\nu}_{\lrtr_2}} (\tau^{e_o}|_{\varrho})^*\big( \nu_{\lrtr_1}\wedge\nu_{\lrtr_2} \big) \wedge ds \wedge dt
 = (-1)^{\chi(\lrtr)} \nu_{\lrtr}\wedge dt;
\end{align*}
here we have $\chi(\lrtr,v_r) = 0$ because we assume that the orientation of $\{-m_{P_{\lrtr_1}},-m_{P_{\lrtr_1}}\}$ agrees with that of $\{-m_{P_{\lrtr_1}},-m_{P_{\lrtr_1}}\}$.

Reversing the orientation condition, we will have $\chi(\lrtr,v_r) = 1$, while at the same time we get an extra $(-1)$ in the above formula because $ (\tau^{e_o}|_{\varrho})^*(d\eta_{\lrtr_1}\wedge d\eta_{\lrtr_2}) = -c ds \wedge dt$, exactly as in the $2$-tree case. This completes the proof.
\end{proof}

Now $\mathcal{I}_\varrho$ is an open neighborhood of $\vec{0} \times \text{Im}(\varrho)$ in the cone $\real_{\leq 0}^{|\lrtr^{[1]}|} \times \text{Im}(\varrho)$, and we let $C(\vec{\tau}) \subset \prod_{i=1}^k N_i$ be its image under the map $\vec{\tau}$. The local diffeomorphism $\vec{\tau}$ allows us to transform the integral in \eqref{tropical_integral} to an integral over $C(\vec{\tau})$, so we have the identity
\begin{align*}
&\int_{x \in \varrho} \alpha_{\lrtr}
 = (-1)^{k-1}\int_{C(\vec{\tau})} (\alpha_1\wedge \cdots\wedge \alpha_k ) + O(e^{-c_{\varrho}/\hp})\\
 =& (-1)^{\chi(\lrtr)+1}( \pi\hp)^{-\frac{k}{2}}\int_{C(\vec{\tau})} e^{-\frac{\sum_{i=1}^k \eta_{i}^2}{\hp}} d\eta_{1} \cdots d\eta_{k} + O(e^{-c_{\varrho}/\hp})
=  (-1)^{\chi(\lrtr)+1} \lim_{\epsilon \rightarrow 0}\frac{\text{vol}(C(\vec{\tau}) \cap B_\epsilon)}{\text{vol}(B_\epsilon)} + O(\hp^{1/2}),
\end{align*}
which computes the leading order contribution of $\int_{x \in \varrho} \alpha_{\lrtr}$; here $B_\epsilon$ is the $\epsilon$-ball in $\prod_{i=1}^k N_i$ and $\text{vol}$ is the volume with respect to the standard metric $\sum_{i=1}^k d\eta_i^2$.

\begin{remark}
The meaning of the above equation is that the leading order contribution of $\int_{x \in \varrho} \alpha_{\lrtr}$ (corresponding to the effect of crossing the new wall $P_{\lrtr}$) depends on how the image $C(\vec{\tau})$ of the locus $\mathcal{I}_\varrho$ in the moduli space $\mathfrak{M}_{\underline{\lrtr}}(U)$ of tropical trees in $U$ intersects with the normals of the initial walls $P_1, P_2$. Figure \ref{fig:tropical_integral} illustrates the situation for a tree with only two incoming edges $e_1, e_2$ and one outgoing edge, where the ansatz for the initial walls $P_1, P_2$ are drawn as in Figure \ref{fig:delta_function}.
\begin{figure}[h]
\centering
\includegraphics[scale=0.4]{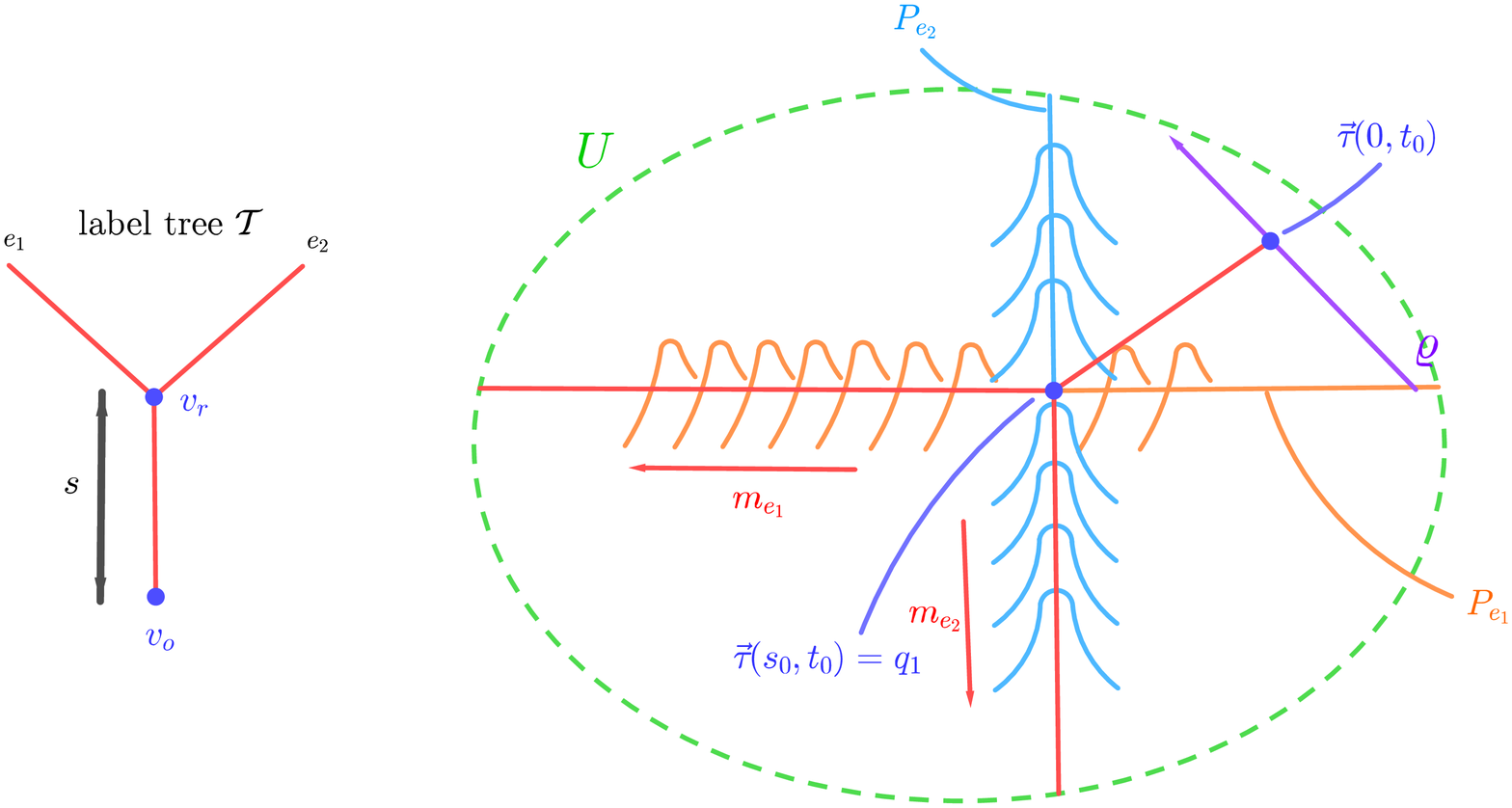}
\caption{Intersection $C(\vec{\tau})$ with $P_i$'s giving tropical counting}
\label{fig:tropical_integral}
\end{figure}
\end{remark}

The following lemma summarizes the results of this subsection:
\begin{lemma}\label{lem:semi_classical_integral}
For a labeled ribbon $k$-tree $\lrtr \in \lrtree{k}_0$ with $P_\lrtr \neq \emptyset$, we write $\alpha_\lrtr \check{\partial}_{n_\lrtr} = \mathfrak{l}_{k,\lrtr}(\vec{\alpha},\vec{n})$ as in \eqref{alpha_out_def}.
Then, for any embedded affine line $\varrho :(a,b)\rightarrow U$ intersecting transversally and positively with $P_\lrtr$, we have
$$
\int_{\varrho} \alpha_\lrtr = (-1)^{\chi(\lrtr)+1} \lim_{\epsilon \rightarrow 0}\frac{\text{vol}(C(\vec{\tau}) \cap B_\epsilon)}{\text{vol}(B_\epsilon)} + O(\hp^{1/2}),
$$
where $\text{vol}$ is the volume with respect to the standard metric $\sum_{i=1}^k d\eta_i^2$ on $\prod_{i=1}^k N_i$, and $\chi(\lrtr)$ is defined as in Lemma \ref{lem:tau_map_iso}.
Moreover, we have $n_\lrtr \in (TP_\lrtr)^\perp$ and $(-1)^{\chi(\lrtr)}(\nu_{P_\lrtr},n_\lrtr)< 0$.
\end{lemma}
\begin{proof}
It remains to prove the last statement, which is yet another induction on the number of vertices of the labeled ribbon tree $\lrtr$. The initial step is trivially true.

For the simplest non-trivial case, we look at a labeled ribbon $2$-tree $\lrtr$ with only one trivalent vertex $v$, two incoming edges $e_1, e_2$ and one outgoing edge $e_o$ meeting $v$, as before. Since $n_i \in (TP_i)^\perp$ and $(\nu_{P_{e_i}},n_i)<0$ for $i = 1, 2$, we have, by the formula \eqref{vertex_lie_algebra},
$
n_\lrtr = ( m_{e_2},n_1 ) n_2 - ( m_{e_1} ,n_2 ) n_1 \in (TP_\lrtr)^\perp
$
since $P_\lrtr = Q-\real_{\geq} \cdot m_\lrtr$.

When the orientation of $\{-m_{e_1},-m_{e_2}\}$ agrees with that of $\{-m_1,-m_2\}$, we can choose $-m_{e_2}$ as an oriented normal to $P_\lrtr$ (here we only care about orientation so there are many different choices), so we have $(-m_{e_2},n_{\lrtr}) = ( -m_{e_1} ,n_2 )(-m_{e_2}, n_1) < 0$ which means that $(\nu_{P_\lrtr},n_\lrtr) < 0$ in this case.

When the orientation of $\{-m_{e_1},-m_{e_2}\}$ disagrees with that of $\{-m_1,-m_2\}$, we have $(\nu_{P_\lrtr},n_\lrtr)>0$ from the above argument as now $-m_{e_1}$ is chosen as an oriented normal to $P_\lrtr$. This completes the proof of the $2$-tree case.

By (once again) splitting a $k$-tree $\lrtr \in \lrtree{k}_0$ at the root vertex $v_r$ into two trees $\lrtr_1$ and $\lrtr_2$, we can prove the induction step using exactly the same argument as above with $n_1, n_2$ replaced by $n_{\lrtr_1}, n_{\lrtr_2}$ respectively.
\end{proof}

\begin{remark}
The integral in Lemma \ref{iteratedintegral} depends on the ribbon structure on $\lrtr$ because the order in taking the wedge product and the orientation of $\mathcal{I}_x$ given by Definition \ref{def:trop_tree_orientation} depend on it. Nevertheless, as Lemmas \ref{lem:tau_map_iso} and \ref{lem:semi_classical_integral} show, the whole expression $\alpha_{\lrtr} \check{\partial}_{\lrtr}$ is independent of the ribbon structure, as only the sign of $\check{\partial}_{\lrtr}$ depends on the ribbon structure and so does $\alpha_\lrtr$, and this dependence cancels out with each other. This matches our earlier observation in equation \eqref{eqn:relating_ribbon_or_not} that the term $\mathfrak{l}_{k,\lrtr}(\eincoming,\dots,\eincoming)$ is independent of the ribbon structure.
\end{remark}


\subsection{Consistent scattering diagrams from Maurer-Cartan solutions}\label{sec:key_lemmas}

In this section, we apply the results we obtained in Sections \ref{sec:leading_order_MC} and \ref{sec:tropical_leading_order} to prove Theorems \ref{theorem1} and \ref{theorem2} in the Introduction.

\subsubsection{The scattering diagram associated to the MC solution $\Phi$}\label{sec:diagram_associated}

Recall that the Maurer-Cartan solution $\Phi$ constructed in \eqref{eqn:MC_sol_Phi} is decomposed as a sum of Fourier modes $\Phi^{(a)} = \emc^{(a)} + \digamma^{(a)}$ (see Definition \ref{solving_MC_tree}). The asymptotic behavior of each $\Phi^{(a)}$ is described by Theorem \ref{asy_support_theorem} and a precise expression for the leading order terms $\emc^{(a)}$ is obtained in Lemma \ref{lem:semi_classical_integral}.
Applying these results, we are going to associate a scattering diagram $\mathscr{D}(\Phi)$ to $\Phi$.

We will first construct a finite diagram $\mathscr{D}(\Phi)_N$ for each fixed $N \in \inte_{>0}$, producing a sequence $\{\mathscr{D}(\Phi)_N\}_{N \in \inte_{>0}}$ such that $\mathscr{D}(\Phi)_{N+1}$ is extension of $\mathscr{D}(\Phi)_N$ in the sense that there is an inclusion $\mathscr{D}(\Phi)_{N} \subset \mathscr{D}(\Phi)_{N+1}$ identifying walls $\text{(mod\ $\mathbf{m}^{N+1}$)}$ and each $\mathbf{w} \in \mathscr{D}(\Phi)_{N+1} \setminus \mathscr{D}(\Phi)_{N}$ has a trivial wall crossing factor $\text{(mod\ $\mathbf{m}^{N+1}$)}$.
Then we define $\mathscr{D}(\Phi)$ as the limit of this sequence.

The order $N$ scattering diagram $\mathscr{D}(\Phi)_N$ will be constructed by adding to the initial diagram $\mathscr{D}(\Phi)_1 = \{\mathbf{w}_1,\mathbf{w}_2\}$ new walls $\mathbf{w}_a$ parametrized by a finite set of $a \in (\inte_{> 0 })^2_{\text{prim}}$, where each $\mathbf{w}_a$ is supported on the half-hyperplane $P_a = Q -\real_{\geq 0 } \cdot m_a$ and equipped with a wall crossing factor $\Theta_a$ (which could be trivial) determined from the leading order term $\emc^{(a)}$ in the asymptotic expansion of $\Phi^{(a)}$.
In order to parametrize the old and new walls by the same parameter space, we introduce the following notations:
\begin{notation}\label{Rays_notation}
We set
$a\in \widetilde{(\inte_{\geq 0 })}^2_{\text{prim}} := (\inte_{\geq 0 })^2_{\text{prim}} \cup \{(-1,0),(0,-1)\},$
and use the (rather unusual) convention that
$m_{(-1,0)} = m_{(1,0)} = m_1 \text{ and } m_{(0,-1)} = m_{(0,1)} = m_2$
for the Fourier modes corresponding to the two initial walls $\mathbf{w}_1$ and $\mathbf{w}_2$. We use $\widetilde{(\inte_{\geq 0 })}^2_{\text{prim}}$ to parametrize the set of half-hyperplanes $P_a$ emanating from $Q$ with slope $-m_a = -(a_1 m_1 + a_2 m_2)$ for $a = (a_1, a_2)$, where we are regarding each initial wall $\mathbf{w}_i$ as a union of two half-hyperplanes.

For a fixed $N \in \inte_{>0}$, there will only be finitely many Fourier modes involved in the expression for the MC solution $\Phi$ in Definition \ref{solving_MC_tree}. For this purpose, we use
$$
\mathbb{W}(N):= \{a \in \widetilde{(\inte_{\geq 0})}^2_{\text{prim}} \mid l m_a = m_\lrtr \ \text{for some $l \geq 1$ and $\lrtr \in \lrtree{k}$ with $1 \leq j_\lrtr \leq N$} \},
$$
where $(m_\lrtr, j_\lrtr)$ is the labeling of the unique outgoing edge $e_o$ attached to the outgoing vertex $v_o$ in $\lrtr$ (see Definition \ref{label_tree_def} and Notations \ref{not:edge_labeling}), to parametrize the possible walls involved in $\Phi$ $\text{(mod $\mathbf{m}^{N+1}$)}$.
\end{notation}

It makes sense to regard each of the two initial walls $\mathbf{w}_1, \mathbf{w}_2$ as a union of two half-hyperplanes in Notations \ref{Rays_notation} because of the following construction:
\begin{definition}\label{cut_off_input}
Given an input term $\incoming^{(i)}$ in the form of \eqref{two_wall_incoming} having asymptotic support on $P_i$ for $i = 1, 2$, we take an affine coordinate function $u_{m_i,1}$ along $-m_i$ which assumes the value $0$ along $Q$. Then the functions
$$
\chi_{i,+}(u_{m_i,1})  := \left( \frac{1}{\hp \pi} \right)^{\half} \int_{-\infty}^{u_{m_i,1}}e^{-\frac{s^2}{\hp}} ds; \quad
\chi_{i,-}(u_{m_i,1})  := 1 - \chi_{1,+}(u_{m_i,1}) = \left( \frac{1}{\hp \pi} \right)^{\half} \int_{u_{m_i,1}}^{\infty}e^{-\frac{s^2}{\hp}} ds
$$
have asymptotic support on $\{u_{m_i,1} \geq 0 \} \cap U$ and $\{u_{m_i,1} \leq 0\} \cap U$ respectively, which implies that the cut-offs $
\Phi^{(1,0)} := \chi_{1,+} \incoming^{(1)}$, and $\Phi^{(-1,0)} := \chi_{1,-} \incoming^{(1)}$
have asymptotic support on $\{u_{m_1,1}\geq 0 \} \cap P_1$ and $\{u_{m_1,1}\leq 0 \} \cap P_1$ respectively as well;
the cut-offs $\Phi^{(0,\pm 1)}$ can be defined similarly using $\chi_{2,\pm}$ and they have asymptotic support on $\{u_{m_2,1}\geq 0 \} \cap P_2$ and $\{u_{m_2,1}\leq 0 \} \cap P_2$ respectively.
\end{definition}
From this construction, we see that both $\pdb \Phi^{(1,0)}, \pdb \Phi^{(-1,0)}$ (resp. $\pdb \Phi^{(0,1)}, \pdb \Phi^{(0,-1)}$) have asymptotic support on $\{u_{m_1,1}=0 \} \cap P_1 = Q$ (resp. $\{u_{m_2,1}=0 \} \cap P_2 = Q$).

To prove that $\mathscr{D}(\Phi)_N$ is consistent $\text{(mod\ $\mathbf{m}^{N+1}$)}$, we will remove $Q = P_1 \cap P_2 = \text{Sing}(\mathscr{D})$ from the spherical neighborhood $U$ and apply a monodromy argument on the annulus $A := U \setminus Q$ by considering the universal cover $\mathtt{p}: \tilde{A} \rightarrow A$, which is endowed with the pullback affine structure from $A$.
We use polar coordinates $(r,\ang)$ on a fiber of the normal bundle $NQ$ (identified with a slice of a tubular neighborhood around $Q$) together with a set of affine coordinates $b := (b_3,\dots,b_n)$ on $Q$ to get the coordinates
$\hat{b} := (b_1 = r, b_2 = \ang, b_3, \dots, b_n)$
on $\tilde{A}$.\footnote{Note that the polar coordinates $(r,\ang)$ are {\em not} affine coordinates.}

We fix, once and for all, an angle $\ang_0$ (chosen up to multiples of $2\pi$) such that the half-hyperplane $\mathscr{R}_{\ang_0}$ with slope $\ang_0$ through $Q$ contains the center $x^0$ of the spherical neighborhood $U$ (recall that the initial walls $\mathbf{w}_1, \mathbf{w}_2$ are dividing $U \cap NQ$ into 4 quadrants and $x^0 $ lies in the 3rd quadrant; see Figures \ref{fig:quadrant} and \ref{fig:spherical_nbh}) and also a base point $\hat{b}^0 = (r_0, \ang_0, b^0) \in \mathscr{R}_{\ang_0}$ such that $\mathtt{p}(\hat{b}_0) = x^0$.

\begin{notation}\label{order_N_neighborhood}
For each $a\in \mathbb{W}(N)$, we associate to the wall $P_a$ an angle $\ang_a$ in the branch $\{(r, \ang, \hat{b}) \mid \ang_0 < \ang < \ang_0 + 2\pi\}$ to parametrize the lifting of $P_a \cap A$ in $\tilde{A}$. We identify $P_a \cap A$ with its lift in $\tilde{A}$, and will denote it again by $P_a$ by abusing notations.
\end{notation}

We choose a sufficiently small $\epsilon_0$ and set $\mathbb{V} := \left\{(r, \ang, b) \mid \ang_0 - \epsilon_{0} + 2\pi < \ang < \ang_0 + 2\pi \right\}$
so that the open subset
$\mathbb{V} - 2\pi = \left\{(r, \ang, b) \mid \ang_0 - \epsilon_0 < \ang < \ang_0 \right\}$ stays away from all the possible walls $\{\mathbf{w}_a\}$, as shown in Figure \ref{fig:raytheta}.
\begin{figure}[h]
\begin{center}
\includegraphics[scale=0.25]{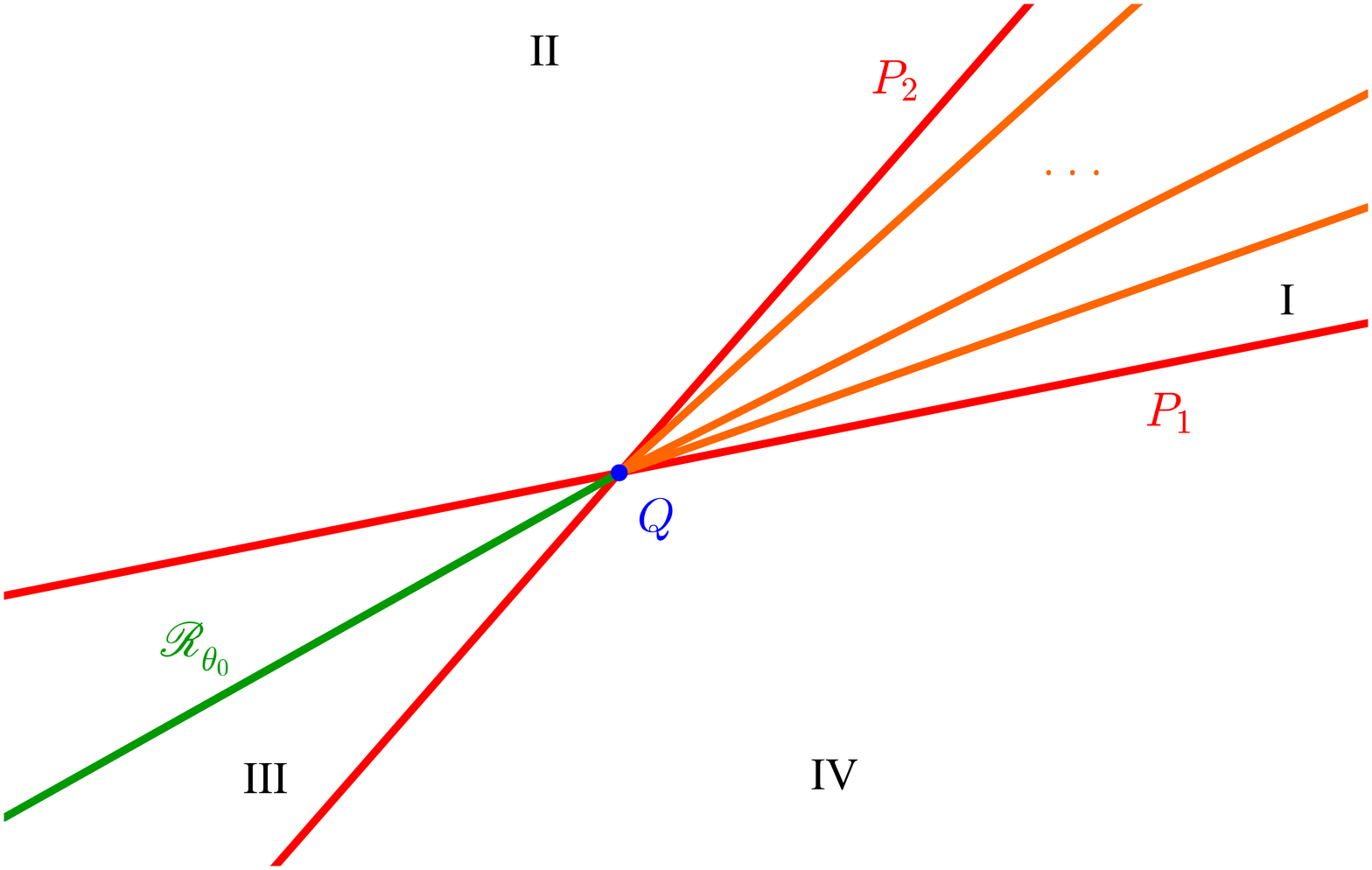}
\end{center}
\caption{}\label{fig:raytheta}
\end{figure}

For computation of some monodromy around $Q$, we consider the open subset $\tilde{A}_{0} := \{(r, \ang, b) \mid \ang_0 - \epsilon_0 < \ang < \ang_0 + 2\pi \} \subset \tilde{A}.$ Through the covering map $\mathtt{p} : \tilde{A}_0 \rightarrow A$, we pull back the dgLa's $\mathbf{g}^*_N(A)$, $\mathcal{E}^*_N(A)$ and $\mathbf{g}^*_N(A)/\mathcal{E}^*_N(A)$ to $\tilde{A}_0$, and consider $\mathbf{g}^*_N(\tilde{A}_0)$, $\mathcal{E}^*_N(\tilde{A}_0)$ and $\mathbf{g}^*_N(\tilde{A}_0)/\mathcal{E}^*_N(\tilde{A}_0)$.



We write $\tilde{\Phi}^{(a)} := \mathtt{p}^*(\Phi^{(a)})$, $\tilde{\emc}^{(a)} := \mathtt{p}^*(\emc^{(a)})$, $\tilde{\digamma}^{(a)} := \mathtt{p}^*(\digamma^{(a)})$, and
$
\mathtt{p}^*(\Phi) := \sum_{a \in \mathbb{W}(N)} \mathtt{p}^*(\Phi^{(a)}) = \sum_{a \in \mathbb{W}(N)} \tilde{\emc}^{(a)} + \tilde{\digamma}^{(a)} \ (\text{mod $\mathbf{m}^{N+1}$}),
$
for the pullbacks to $\tilde{A}_0$. We then have the following lemma.

\begin{lemma}\label{split_MC}
For each $a \in \widetilde{(\inte_{\geq 0 })}^2_{\text{prim}}$, the Fourier mode $\tilde{\Phi}^{(a)}$ is itself a solution of the Maurer-Cartan equation \eqref{MCequation} of the dgLa $\widehat{\mathbf{g}^*/\mathcal{E}^*}(\tilde{A}_0)$ which further satisfies $\pdb \tilde{\Phi}^{(a)} = [\tilde{\Phi}^{(a)}, \tilde{\Phi}^{(a)}] = 0$.
\end{lemma}

\begin{proof}
For any fixed $N \in \inte_{>0}$, $\tilde{\Phi} = \sum_{a \in \mathbb{W}(N)} \tilde{\Phi}^{(a)} \ \text{(mod $\mathbf{m}^{N+1}$)}$ by Definition \ref{solving_MC_tree} and $\mathbb{W}(N)$ is a finite set.
Now for two different $a, a'\in \mathbb{W}(N)$, we have $(P_{a} \cap P_{a'}) \cap A = \emptyset$, and hence $[\tilde{\Phi}^{(a)},\tilde{\Phi}^{(a')}] \in \mathcal{E}^2_N(\tilde{A}_0)$ which means that $[\tilde{\Phi}^{(a)},\tilde{\Phi}^{(a')}]= 0 $ in $\mathbf{g}_N^2(\tilde{A}_0)/\mathcal{E}_N^2(\tilde{A}_0)$. Therefore each $\tilde{\Phi}^{(a)}$ is itself a MC solution in $\mathbf{g}^*_N(\tilde{A}_0)/\mathcal{E}^*_N(\tilde{A}_0)$. Taking inverse limit shows that $\tilde{\Phi}^{(a)}$ is a MC solution in $\widehat{\mathbf{g}^*/\mathcal{E}^*}(\tilde{A}_0)$. Furthermore, as $\tilde{\Phi}^{(a)}$ is having asymptotic support on $P_a$, we have $[\tilde{\Phi}^{(a)},\tilde{\Phi}^{(a)}] = 0$ in $\widehat{\mathbf{g}^*/\mathcal{E}^*}(\tilde{A}_0)$ by Lemma \ref{filtrationlemma} (because $P_a$ intersects itself non-transversally), so $\pdb \tilde{\Phi}^{(a)} =\pdb \tilde{\Phi}^{(a)} + \half [\tilde{\Phi}^{(a)}, \tilde{\Phi}^{(a)}] = 0$.
\end{proof}

Lemma \ref{split_MC} says that
$\tilde{\Phi}^{(a)} = \tilde{\emc}^{(a)} + \tilde{\digamma}^{(a)}$
is a Maurer-Cartan solution in $\widehat{\mathbf{g}^*/\mathcal{E}^*}(\tilde{A}_0)$ with support concentrated along the wall $P_a$. Using a similar argument as the proof of Lemma \ref{cohomology_lemma}, we note that the higher cohomologies of the complex $\mathbf{g}^*_N(\tilde{A}_0)/\mathcal{E}^*_N(\tilde{A}_0)$ and $\widehat{\mathbf{g}^*/\mathcal{E}^*}(\tilde{A}_0)$ are all trivial. So there are no non-trivial deformations of the dgLa $\widehat{\mathbf{g}^*/\mathcal{E}^*}(\tilde{A}_0)$. In particular, the MC solution $\tilde{\Phi}^{(a)}$ is gauge equivalent to $0$, i.e. there exists $\facs_a \in \widehat{\mathbf{g}^1/\mathcal{E}^1}(\tilde{A}_0)$ such that
\begin{equation}\label{eqn:gauge_facs}
e^{\facs_a} * 0 = \tilde{\emc}^{(a)} + \tilde{\digamma}^{(a)}
\end{equation}
on $\tilde{A}_0$. As in the single wall case (Section \ref{sec:relating_wall_crossing_fac}), we need to define a homotopy operator $\hat{\mathcal{H}}$ on $\widehat{\mathbf{g}^*/\mathcal{E}^*}(\tilde{A}_0)$ in order to fix the choice of $\facs_a$.

We take a smooth homotopy $h : [0,1] \times \tilde{A}_0 \rightarrow \tilde{A}_0$ contracting $\tilde{A}_0$ to the fixed point $\hat{b}^0$ with the property that $h(1,\hat{b}) = \hat{b}$ and $h(0,\hat{b}) = \hat{b}^0$. We define the homotopy operator $\hat{\mathcal{H}} : F^{\infty}_*(\tilde{A}_0) \rightarrow F^{\infty}_{*-1}(\tilde{A}_0)$ by
\begin{equation}\label{homotopy_equation}
\hat{\mathcal{H}}(\alpha) := \int_{0}^1 h^*(\alpha)
\end{equation}
for $\alpha \in F^{\infty}_*(\tilde{A}_0)$, and we also define $\hat{\mathcal{P}}$ by the evaluation at the base point $\hat{b}^0$ and $\hat{\iota}$ by the embedding of constant functions on $\tilde{A}_0$, as before (cf. Section \ref{sec:relating_wall_crossing_fac}); one can see that these operators descend to the quotient $F^\infty_*(\tilde{A}_0) / F^{-\infty}_*(\tilde{A}_0)$ using the same argument as in the explanation for Definition \ref{pathspacehomotopy}. We extend the above operators to the complex $\mathbf{g}^*_N(\tilde{A}_0) / \mathcal{E}^*_N(\tilde{A}_0)$ as follows.

\begin{definition}\label{polarrealhomotopy}
We define the operators $\hat{\mathcal{H}}$, $\hat{\mathcal{P}}$ and $\hat{\iota}$ by extending linearly the formulas $\hat{\mathcal{H}}(\alpha \bmc^{m} \check{\partial}_n t^j)  := \hat{\mathcal{H}}(\alpha) \bmc^m \check{\partial}_n t^j$, $\hat{\mathcal{P}}(\alpha \bmc^{m} \check{\partial}_n t^j)  := \hat{\mathcal{P}}(\alpha) \bmc^m \check{\partial}_n t^j$ and $\hat{\iota}(\alpha \bmc^{m} \check{\partial}_n t^j)  := \hat{\iota}(\alpha) \bmc^m \check{\partial}_n t^j$, descending to the quotient and taking inverse limit.
\end{definition}

\begin{definition}\label{facs_a_definition}
Similar to the deduction of \eqref{vartheta_s_definition} from \eqref{gaugeequation} \cite{manetti2005differential}, we solve the equation \eqref{eqn:gauge_facs} in $\widehat{\mathbf{g}^*/\mathcal{E}^*}(\tilde{A}_0)$ iteratively to obtain the gauge:
$
\facs_{a} := - \hat{\mathcal{H}}\left(ad_{\facs_a} / (e^{ad_{\facs_a}} - \text{Id})\right) (\tilde{\emc}^{(a)} + \tilde{\digamma}^{(a)})
$
associated to $\tilde{\Phi}^{(a)} = \tilde{\emc}^{(a)}+\tilde{\digamma}^{(a)}$ which satisfies the gauge fixing condition $\hat{\mathcal{P}}(\facs_a) = 0$; this gauge is unique by Lemma \ref{gauge_fixing_lemma}.
\end{definition}

We now apply asymptotic analysis to the gauge $\facs_a$, similar to what we have done in the single wall case.
First of all, as in Section \ref{sec:analysis_one_wall} (see Remark \ref{pathindependent} and the setup before Lemma \ref{integral_lemma}), we shall replace $\hat{\mathcal{H}}$ by another operator $\hat{\mathcal{I}}$, which is defined using an integral over affine lines transversal to the wall $P_a$.
For this purpose, we consider the half-space
$\hat{\mathbb{H}}(P_a) := \{(r,\ang,b) \in \tilde{A}_0 \mid \ang \geq \ang_a\}$
in $\tilde{A}_0$, on which $\facs_a$ is possibly having asymptotic support.
Note that $\ang$ is not an affine coordinate but we can always express $\hat{\mathbb{H}}(P_a)$ as a tropical half-space in $\tilde{A}_0$ (by pulling back an affine linear function defining $P_a$ to $\tilde{A}_0$ and parallel transporting to hyperplanes parallel to $P_a$).

We write
$\tilde{\emc}^{(a)} = \sum_{s = 1}^\infty \tilde{\emc}^{(a)}_s$, $\tilde{\digamma}^{(a)} = \sum_{s = 1}^\infty \tilde{\digamma}^{(a)}_s$ and $\facs_{a} = \sum_{s = 1}^\infty \facs_{a,s}$ according to powers of the formal variable $t$. We also set $\facs_a^s  := \facs_{a,1}+\cdots+\facs_{a,s}$, $\tilde{\emc}^{(a),s}  := \tilde{\emc}^{(a)}_1 +\cdots+\tilde{\emc}^{(a)}_s$ and $\tilde{\digamma}^{(a),s}  := \tilde{\digamma}^{(a)}_1 + \cdots + \tilde{\digamma}^{(a)}_s$.
Then we have the following lemma, which is parallel to Lemma \ref{leadingorderlemma} in Section \ref{onewall}.

\begin{lemma}\label{lem:loc_constant_coeff}
The gauge $\facs_a$ has asymptotic support on the (codimension 0) tropical half-space $\hat{\mathbb{H}}(P_a) \subset \tilde{A}_0$, and we have
\begin{align*}
\facs_{a,s} &\in \bigoplus_{k\geq 1} \sum_{n \in \Lambda^\vee_{B_0}(U)} F^{0}_{\hat{\mathbb{H}}(P_a)}(\tilde{A}_0)\cdot \bmc^{km_a} \check{\partial}_{n} t^s ,\\
\facs_{a,s} + \hat{\mathcal{H}}(\tilde{\emc}^{(a)}_{s}) &\in \bigoplus_{k\geq 1 } \sum_{n \in \Lambda^\vee_{B_0}(U)} F^{-1}_{\hat{\mathbb{H}}(P_a)}(\tilde{A}_0)\cdot \bmc^{km_a} \check{\partial}_{n} t^s ,\\
ad_{\facs_a^s}^l (\pdb \facs_a^s) &\in  \bigoplus_{\substack{k\geq 1\\ 1 \leq j \leq s(l+1)}} \sum_{n \in \Lambda^\vee_{B_0}(U)} F^{0}_{P_a}(\tilde{A}_0)\cdot \bmc^{km_a} \check{\partial}_{n} t^j 
\end{align*}
for all $s \geq 1$ and $l \geq 1$.
\end{lemma}

\begin{proof}
We prove by induction on $s$ (the power of the formal variable $t$). In the initial case, the equation defining $\facs_{a,1}$ is
$\facs_{a,1} = -\hat{\mathcal{H}}( \tilde{\emc}^{(a)}_1+ \tilde{\digamma}^{(a)}_1)$. We want define an integral operator to replace $\hat{\mathcal{H}}$ in order to apply Lemma \ref{integral_lemma}.
Since all the assertions that we need to prove are local properties, we will work locally around any given point in $\tilde{A}_0$.

So we fix a point $\hat{b}^1 \in \tilde{A}_0$ and choose a sufficiently small pre-compact open neighborhood $K \subset \tilde{A}_0$ of $\hat{b}_1$, and then try to prove the initial case in $K$.
We will also need to choose a family of piecewise affine lines, as in the single wall case.
There are two scenarios:
\begin{enumerate}
\item
If $\hat{b}^1 \in \hat{\mathbb{H}}(P_a)$, we choose a sufficiently small pre-compact open neighborhood $K$ of $\hat{b}^1$ and
a family of paths $\varrho_K : [0,1] \times K \rightarrow \tilde{A}_0$ such that $\varrho_K(0,\hat{b}) = \hat{b}^0$, $\varrho_K(1,\hat{b}) = \hat{b}$ and there exists a partition $0 = t_0 < \cdots < t_{l-1}<t_{l} = 1 $ so that only one of the intervals $[t_{i_0-1},t_{i_0}]$ has its image possibly intersecting with $P_a$ and that $\varrho_{K}|_{[t_{i_0-1},t_{i_0}]}$ is a flow line of the affine vector field $v_K$ pointing into $\hat{\mathbb{H}}(P_a)$.
\item
If $\hat{b}^1 \not\in \hat{\mathbb{H}}(P_a)$, we choose a sufficiently small pre-compact open neighborhood $K$ of $\hat{b}^1$ with $K \cap \hat{\mathbb{H}}(P_a) = \emptyset$ and
a family of paths $\varrho_K : [0,1] \times K \rightarrow \tilde{A}_0$ such that $\varrho_K(0,\hat{b}) = \hat{b}^0$, $\varrho_K(1,\hat{b}) = \hat{b}$ and $\text{Im}(\varrho_K) \cap P_a = \emptyset$.
\end{enumerate}
Such a family always exists when $K$ is sufficiently small. Figure \ref{fig:different_path} illustrates the difference between the integral operators $\hat{\mathcal{H}}$ and $\hat{\mathcal{I}}_{a,K}$.
\begin{figure}[h]
\begin{center}
\includegraphics[scale=0.3]{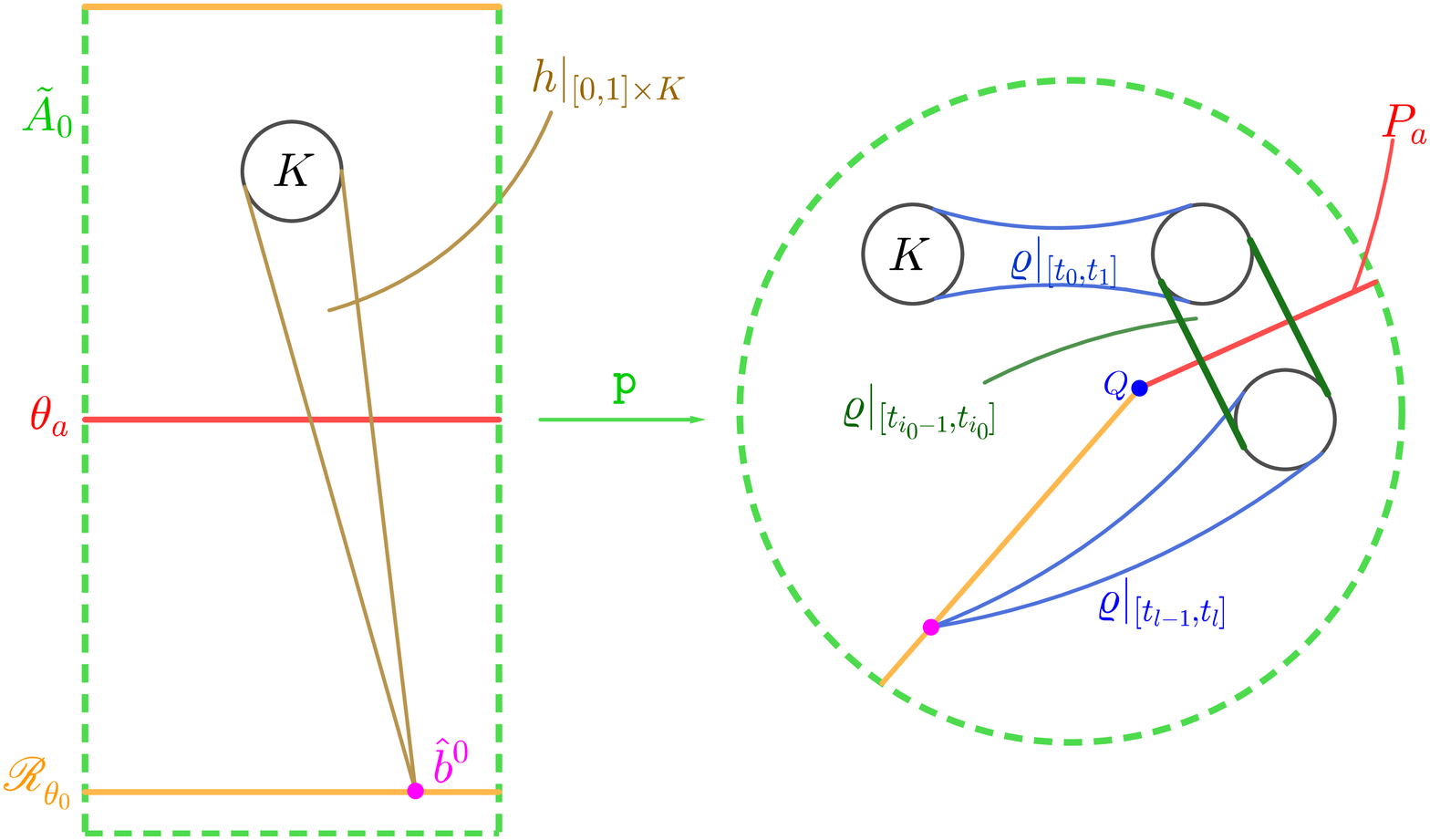}
\end{center}
\caption{}\label{fig:different_path}
\end{figure}
Then we set
$
\hat{\mathcal{I}}_{a, K}(\alpha)(\hat{b}) := \int_0^1 \varrho^*(\alpha)(s,\hat{b}).
$ Applying Lemma \ref{integral_lemma} to the piece $\varrho_K|_{[t_{i_0-1},t_{i_0}]}$ gives
$$
-\hat{\mathcal{H}}(\tilde{\emc}^{(a)}_1)  \in F^{0}_{\hat{\mathbb{H}}(P_a)}(K)\cdot (\bmc^{km_a} \check{\partial}_{n_a}) t^1;\quad
-\hat{\mathcal{H}}(\tilde{\digamma}^{(a)}_1)  \in \sum_{n \in \Lambda^\vee_{B_0}(U)} F^{-1}_{\hat{\mathbb{H}}(P_a)}(K)\cdot (\bmc^{km_a} \check{\partial}_{n}) t^1,
$$
which proves the first two assertions in the initial case.

For the third assertion, we have
$$\pdb \facs_{a,1} = -\pdb \hat{\mathcal{H}}(\tilde{\emc}^{(a)}_1 + \tilde{\digamma}^{(a)}_1) = -\tilde{\emc}^{(a)}_1-\tilde{\digamma}^{(a)}_1 \in \bigoplus_{k\geq 1} \sum_{n \in \Lambda^\vee_{B_0}(U)} F^{1}_{P_a}(\tilde{A}_0)\cdot \bmc^{km_a} \check{\partial}_{n} t^1$$
from the gauge fixing condition in Definition \ref{facs_a_definition}. Upon repeated applications of Lemma \ref{filtrationlemma}, we have
$
ad_{\facs_a^1}^l (\tilde{\digamma}^{(a)}_1) \in \bigoplus_{\substack{k\geq 1\\ 1 \leq j \leq l+1}} \sum_{n \in \Lambda^\vee_{B_0}(U)} F^{0}_{P_a}(\tilde{A}_0)\cdot \bmc^{km_a} \check{\partial}_{n} t^j, 
$
so we only need to take care of the term $ad_{\facs_a^1}^l(\tilde{\emc}^{(a)}_1 )$. Writing $\facs_{a,1} = -\hat{\mathcal{H}}(\tilde{\emc}^{(a)}_1 + \tilde{\digamma}^{(a)}_1)$ and applying Lemma \ref{filtrationlemma} again, we see that the only term we have to consider is $ad_{-\hat{\mathcal{H}}(\tilde{\emc}^{(a)}_1)}^l(\tilde{\emc}^{(a)}_1 )$ in the expression $ad_{-\hat{\mathcal{H}}(\tilde{\emc}^{(a)}_1+ \tilde{\digamma}^{(a)}_1)}^l(\tilde{\emc}^{(a)}_1 )$, because the appearance of any of $ad_{-\hat{\mathcal{H}}(\tilde{\digamma}^{(a)}_1)}$ in the above expression will result in a term in $\bigoplus_{j, k\geq 1} \sum_{n \in \Lambda^\vee_{B_0}(U)} F^{0}_{P_a}(\tilde{A}_0)\cdot (\bmc^{km_a} \check{\partial}_{n}) t^j$. Concerning the term $ad_{-\hat{\mathcal{H}}(\tilde{\emc}^{(a)}_1)}^l(\tilde{\emc}^{(a)}_1 )$, Theorem \ref{asy_support_theorem} says that that $n_a \perp P_a$ in the expressions
$$
\hat{\mathcal{H}}(\tilde{\emc}^{(a)}_1)  \in \bigoplus_{k\geq 1}  F^{0}_{\hat{\mathbb{H}}(P_a)}(\tilde{A}_0)\cdot \bmc^{km_a} \check{\partial}_{n_a} t^1;\quad
\tilde{\emc}^{(a)}_1  \in \bigoplus_{k\geq 1} F^{1}_{P_a}(\tilde{A}_0)\cdot \bmc^{km_a} \check{\partial}_{n_a} t^1,
$$
meaning that the leading order term of $ad_{-\hat{\mathcal{H}}(\tilde{\emc}^{(a)})}(\tilde{\emc}^{(a)}_1)$ given by Lemma \ref{filtrationlemma} vanishes. Hence the third assertion follows.


Now we assume that the assertions hold for $s' \leq s$. We consider the equation
$
\facs_{a,s+1} =  -\hat{\mathcal{H}} \big(\tilde{\emc}^{(a)}_{s+1} + \tilde{\digamma}^{(a)}_{s+1} + \sum_{k\geq 0 }\frac{ad_{\facs_a^s}^k}{(k+1)!} \pdb\facs_a^s \big)_{s+1}
$
which determines $\facs_{a,s+1}$ iteratively. From the induction hypothesis, we have
$$
\tilde{\digamma}^{(a)}_{s+1}+\left(\sum_{k\geq 0 }\frac{ad_{\facs_a^s}^k}{(k+1)!} \pdb\facs_a^s \right)_{s+1} \in \bigoplus_{k\geq 1 } \sum_{n \in \Lambda^\vee_{B_0}(U)} F^{0}_{P_a}(\tilde{A}_0)\cdot \bmc^{km_a} \check{\partial}_{n} t^{s+1}.
$$
Applying $\hat{\mathcal{H}}$ (replacing $\hat{\mathcal{H}}$ by $\hat{\mathcal{I}}_{a,K}$ again) to this expression give the first two assertions of the induction step by Lemma \ref{integral_lemma}.

For the third assertion, we have
$\pdb \facs_{a,s+1} = -\big(\tilde{\emc}^{(a)}_{s+1} + \tilde{\digamma}^{(a)}_{s+1} + \sum_{k\geq 0 }\frac{ad_{\facs_a^s}^k}{(k+1)!} \pdb\facs_a^s \big)_{s+1}$
in $\widehat{\mathbf{g}^*/\mathcal{E}^*}(\tilde{A}_0)$, again from the gauge fixing condition in Definition \ref{facs_a_definition}. Applying Lemma \ref{filtrationlemma} as in the proof of the initial step, we see that the essential term to be considered is $ad_{-\hat{\mathcal{H}}(\tilde{\emc}^{(a),s+1})}^l(\tilde{\emc}^{(a),s+1})$. By Theorem \ref{asy_support_theorem} again, we have $n \perp P_a$ in the following expressions
$$
\hat{\mathcal{H}}(\tilde{\emc}^{(a),s+1})  \in \bigoplus_{\substack{k \geq 1 \\ 1 \leq j \leq s+1}}  F^{0}_{\hat{\mathbb{H}}(P_a)}(\tilde{A}_0)\cdot \bmc^{km_a} \check{\partial}_{n_a} t^{j}; \quad
\tilde{\emc}^{(a),s+1}  \in \bigoplus_{\substack{k \geq 1 \\ 1 \leq j \leq s+1}} F^{1}_{P_a}(\tilde{A}_0)\cdot \bmc^{km_a} \check{\partial}_{n_a} t^{j}.
$$
Then we can conclude that
$
ad_{-\hat{\mathcal{H}}(\tilde{\emc}^{(a),s+1})}^l(\tilde{\emc}^{(a),s+1}) \in \bigoplus_{\substack{k\geq 1\\1 \leq j\leq (s+1)(l+1)}}  F^{0}_{P_a}(\tilde{A}_0)\cdot \bmc^{km_a} \check{\partial}_{n_a} t^j
$
because the leading order term given by Lemma \ref{filtrationlemma} vanishes, as in the proof of the initial step. This finishes the proof of the induction step.
\end{proof}

The following lemma is parallel to Proposition \ref{prop:MC_sol_one_wall} in Section \ref{onewall}.

\begin{lemma}\label{asymptoticexpansion}
Over the half-space $\hat{\mathbb{H}}(P_a) \setminus P_a$, we have
$$
\displaystyle \facs_a \in \psi_a + \left(\bigoplus_{k\geq 1} \sum_{n \in \Lambda^\vee_{B_0}(U)} F^{-1}_{\hat{\mathbb{H}}(P_a) \setminus P_a}(\hat{\mathbb{H}}(P_a) \setminus P_a) \cdot \bmc^{km_a} \check{\partial}_{n} \right)[[t]]\cdot t,
$$
where $\psi_a = \text{Log}(\Theta_a)$ for some element $\Theta_a$ of the tropical vertex group of the form
$
\psi_a  = \sum_{j, k \geq 1} b^{(a)}_{jk} \cdot \bmc^{km_a} \check{\partial}_{n_a} t^j,
$
where $b^{(a)}_{jk}$'s are constants independent of $\hp$ with $b^{(a)}_{jk} \neq 0$ only for finitely many $k$'s for each fixed $j$
and $n_a$ is the unique primitive normal to $P_a$ satisfying $(\nu_{P_a}, n_a) < 0$;
while over the other half-space $\tilde{A}_0 \setminus \hat{\mathbb{H}}(P_a)$, we have
$
\facs_a = 0.
$
\end{lemma}

\begin{proof}
We first consider $\facs_a$ over $\hat{\mathbb{H}}(P_a) \setminus P_a$.
From the proof of Lemma \ref{lem:loc_constant_coeff}, we see that
$$\pdb \facs_{a,s} = -\left(\tilde{\emc}^{(a)}_{s} + \tilde{\digamma}^{(a)}_{s} + \sum_{k\geq 0 }\frac{ad_{\facs_a^{s-1}}^k}{(k+1)!} \pdb\facs_a^{s-1} \right)_{s}  \in \bigoplus_{k\geq 1 } \sum_{n \in \Lambda^\vee_{B_0}(U)} F^{0}_{P_a}(\tilde{A}_0)\cdot \bmc^{km_a} \check{\partial}_{n} t^{s}$$
for every $s \geq 1$. In particular, we have $\pdb \facs_{a,s} = 0$ in $\widehat{\mathbf{g}^*/\mathcal{E}^*}(\hat{\mathbb{H}}(P_a) \setminus P_a)$.

Applying Lemma \ref{cohomology_lemma} to $\widehat{\mathbf{g}^*/\mathcal{E}^*}(\hat{\mathbb{H}}(P_a) \setminus P_a)$, we can write $\facs_{a,s} = (\hat{\iota}_1 \circ \hat{\mathcal{P}}_1)(\facs_{a,s})$ where $\hat{\mathcal{P}}_1$ is the projection operator defined by evaluating at a point $\hat{b}^1 \in \hat{\mathbb{H}}(P_a) \setminus P_a$ and $\hat{\iota}_1$ is the corresponding embedding operator, constructed similarly as $\hat{\mathcal{P}}$ in Definition \ref{polarrealhomotopy}.\footnote{Note that $\hat{\mathcal{P}}$ and $\hat{\mathcal{P}}_1$ are defined by evaluation at two {\em different} points $\hat{b}_0$ and $\hat{b}_1$ respectively.}
By Lemma \ref{lem:loc_constant_coeff} and the above discussion, it remains to show that the leading order term of the asymptotic expansion of
$-\hat{\mathcal{P}}_1 \big(\hat{\mathcal{H}}(\tilde{\emc}^{(a)}_{s} + \tilde{\digamma}^{(a)}_{s} + \sum_{k\geq 0 }\frac{ad_{\facs_a^{s-1}}^k}{(k+1)!} \pdb\facs_a^{s-1} )_{s}\big)$
is exactly of the form $\psi_a$ over $\hat{\mathbb{H}}(P_a) \setminus P_a$.

Choose a neighborhood $K \subset \tilde{A}_0$ of $\hat{b}_1$ and a family of paths $\varrho_K$, and using the operator $\hat{\mathcal{I}}_{a,K}$ as defined in the proof of Lemma \ref{lem:loc_constant_coeff},
we see that
$$\hat{\mathcal{P}}_1 \left(\hat{\mathcal{H}}\left(\tilde{\emc}^{(a)}_{s} + \tilde{\digamma}^{(a)}_{s} + \sum_{k\geq 0 }\frac{ad_{\facs_a^{s-1}}^k}{(k+1)!} \pdb\facs_a^{s-1} \right)_{s} \right) = \int_{\varrho_K(\cdot,\hat{b}^1)} \left(\tilde{\emc}^{(a)}_{s} + \tilde{\digamma}^{(a)}_{s} + \sum_{k\geq 0 }\frac{ad_{\facs_a^{s-1}}^k}{(k+1)!} \pdb\facs_a^{s-1} \right)_{s}.$$
Also we have
$$\hat{\iota}_1 \int_{\varrho_K(\cdot,\hat{b}^1)} \left(\tilde{\digamma}^{(a)}_{s} + \sum_{k\geq 0 }\frac{ad_{\facs_a^{s-1}}^k}{(k+1)!} \pdb\facs_a^{s-1} \right)_{s}\in \bigoplus_{k \geq 1} \sum_{n \in \Lambda^\vee_{B_0}(U)} F^{-1}_{\hat{\mathbb{H}}(P_a)\setminus P_a}(\hat{\mathbb{H}}(P_a) \setminus P_a) \cdot \bmc^{km_a} \check{\partial}_{n} t^s.$$
So it remains to compute $\int_{\varrho(\cdot,\hat{b}^1)}(\tilde{\emc}^{(a)}_{s})$.

Lemma \ref{tree_lemma} together with Lemma \ref{lem:semi_classical_integral} allow us to compute the leading order term of the integral $-\int_{\varrho(\cdot,\hat{b}^1)}  \tilde{\emc}^{(a)}$ explicitly as
$
-\int_{\varrho_K(\cdot,\hat{b}^1)} \tilde{\emc}^{(a)} = \sum_k \frac{1}{2^{k-1}}\sum_{\lrtr \in \lrtree{k}_0: P_{\lrtr} \neq \emptyset, m_\lrtr \parallel m_a} \int_{\varrho_K(\cdot,\hat{b}^1)} \mathfrak{l}_{k,\lrtr}(\eincoming, \dots, \eincoming).
$
From the discussion in Section \ref{sec:tropical_leading_order}, we learn that $\mathfrak{l}_{k,\lrtr}(\eincoming, \dots, \eincoming) = \alpha_\lrtr \check{\partial}_{n_\lrtr} \bmc^{m_\lrtr} t^{j_\lrtr}$ for each $\lrtr \in \lrtree{k}_0$. Therefore, restricting to the interval $[t_{i_0-1},t_{i_0}]$ of $\varrho_K(\cdot,\hat{b}^1)$ and applying Lemma \ref{lem:semi_classical_integral}, we find that the $\hp$ order expansion of $\int_{\varrho(\cdot,\hat{b}^1)}  \mathfrak{l}_{k,\lrtr}(\eincoming, \dots, \eincoming)$ is of the form
$
\int_{\varrho_K(\cdot,\hat{b}^1)}  \mathfrak{l}_{k,\lrtr}(\eincoming,\dots,\eincoming) \in  (b^{(a)}_{j_{\lrtr},k_{\lrtr}} + O(\hp^{1/2})) \bmc^{k_{\lrtr} m_a} \check{\partial}_{n_{\lrtr}} t^{j_{\lrtr}},
$
where $k_{\lrtr}m_a = m_{\lrtr}$ (here $m_{\lrtr}$, $j_{\lrtr}$ and $n_{\lrtr}$ are introduced in Definition \ref{label_tree_def} and the equation \eqref{alpha_out_def}). This proves the desired result over $\hat{\mathbb{H}}(P_a) \setminus P_a$.

Over the other half-space $\tilde{A}_0 \setminus \hat{\mathbb{H}}(P_a)$, the same reason yields $\pdb \facs_{a,s} = 0$ in $\widehat{\mathbf{g}^*/\mathcal{E}^*}(\tilde{A}_0 \setminus \hat{\mathbb{H}}(P_a))$. Therefore we have $\facs_{a,s} = (\hat{\iota} \circ \hat{\mathcal{P}})(\facs_{a,s}) = 0$
from the gauge fixing condition in Definition \ref{facs_a_definition}, where $\hat{\mathcal{P}}$ is treated as an operator acting on $\widehat{\mathbf{g}^*/\mathcal{E}^*}(\tilde{A}_0 \setminus \hat{\mathbb{H}}(P_a))$.
\end{proof}


Now we are ready to construct the order $N$ scattering diagram $\mathscr{D}(\Phi)_N$ for any fixed $N \in \inte_{>0}$. Given $a \in \mathbb{W}(N)$, Lemma \ref{asymptoticexpansion} says that the leading order term in the asymptotic expansion of the gauge $\facs_a$ produces the element
$
\psi_a = \sum_{\substack{k\geq 1 \\1 \leq j \leq N}} b^{(a)}_{jk} \cdot \bmc^{km_a} \check{\partial}_{n_a} t^j \ (\text{mod $\mathbf{m}^{N+1}$})
$
over $\hat{\mathbb{H}}(P_a) \setminus P_a$.

\begin{definition}\label{def:construct_diagram}
We define the order $N$ scattering diagram as
$\mathscr{D}(\Phi)_N := \{\mathbf{w}_a \mid a \in \mathbb{W}(N)\},$
where each newly added wall $\mathbf{w}_a$ is supported on the tropical half-hyperplane $P_a = Q - \real_{\geq 0 } m_a \subset U$ and equipped with the wall crossing factor $\Theta_a$ defined by
$
\text{Log}(\Theta_a) :=  \sum_{\substack{k\geq 1 \\1 \leq j \leq N}} b^{(a)}_{jk} \cdot \bmc^{km_a} \check{\check{\partial}}_{n_a} t^j \quad (\text{mod $\mathbf{m}^{N+1}$}).
$
The order $N+1$ diagram $\mathscr{D}(\Phi)_{N+1}$ is naturally an extension of the order $N$ diagram $\mathscr{D}(\Phi)_N$ because $\Phi^{(a)}$, $\emc^{(a)}$ and hence $\facs_a$ are defined for all orders of $t$. Hence this defines a scattering diagram $\mathscr{D}(\Phi)$ associated to $\Phi$.
\end{definition}

\subsubsection{Consistency of $\mathscr{D}(\Phi)$}\label{reformulatedtheorem}

We are now ready to prove Theorem \ref{theorem2}:

\begin{theorem}[=Theorem \ref{theorem2}]
\label{scatteringtheorem2}
For the Maurer-Cartan solution $\Phi$ constructed in \eqref{eqn:MC_sol_Phi}, the associated scattering diagram $\mathscr{D}(\Phi)$ defined in Definition \ref{def:construct_diagram} is consistent, i.e. we have the identity
$
\Theta_{\gamma,\mathscr{D}(\Phi)} = \prod^{\gamma}_{\mathbf{w}_a \in \mathscr{D}(\Phi)} \Theta_a = \text{Id}
$
for any embedded loop $\gamma$ in $U \setminus \text{Sing}(\mathscr{D}(\Phi))$ intersecting $\mathscr{D}(\Phi)$ generically.
\end{theorem}

\begin{proof}
Let us first recall that we are working over the open subset
$\tilde{A}_{0} = \left\{(r, \ang, b) \mid \ang_0 - \epsilon_0 < \ang < \ang_0 + 2\pi \right\},$
in the universal cover $\tilde{A}$ of $A = U \setminus Q$, where $Q = P_1 \cap P_2 = \text{Sing}(\mathscr{D}(\Phi))$.
We have also fixed a strip
$\mathbb{V} = \left\{(r, \ang, b) \mid \ang_0 - \epsilon_{0} + 2\pi < \ang < \ang_0 + 2\pi \right\}$
so that the strip $\mathbb{V} - 2\pi = \left\{(r, \ang, b) \mid \ang_0 - \epsilon_0 < \ang < \ang_0 \right\}$ stays away from all the possible walls in $\mathscr{D}(\Phi)$; see Figure \ref{fig:theorem_proof}.
\begin{figure}[h]
\begin{center}
\includegraphics[scale=0.3]{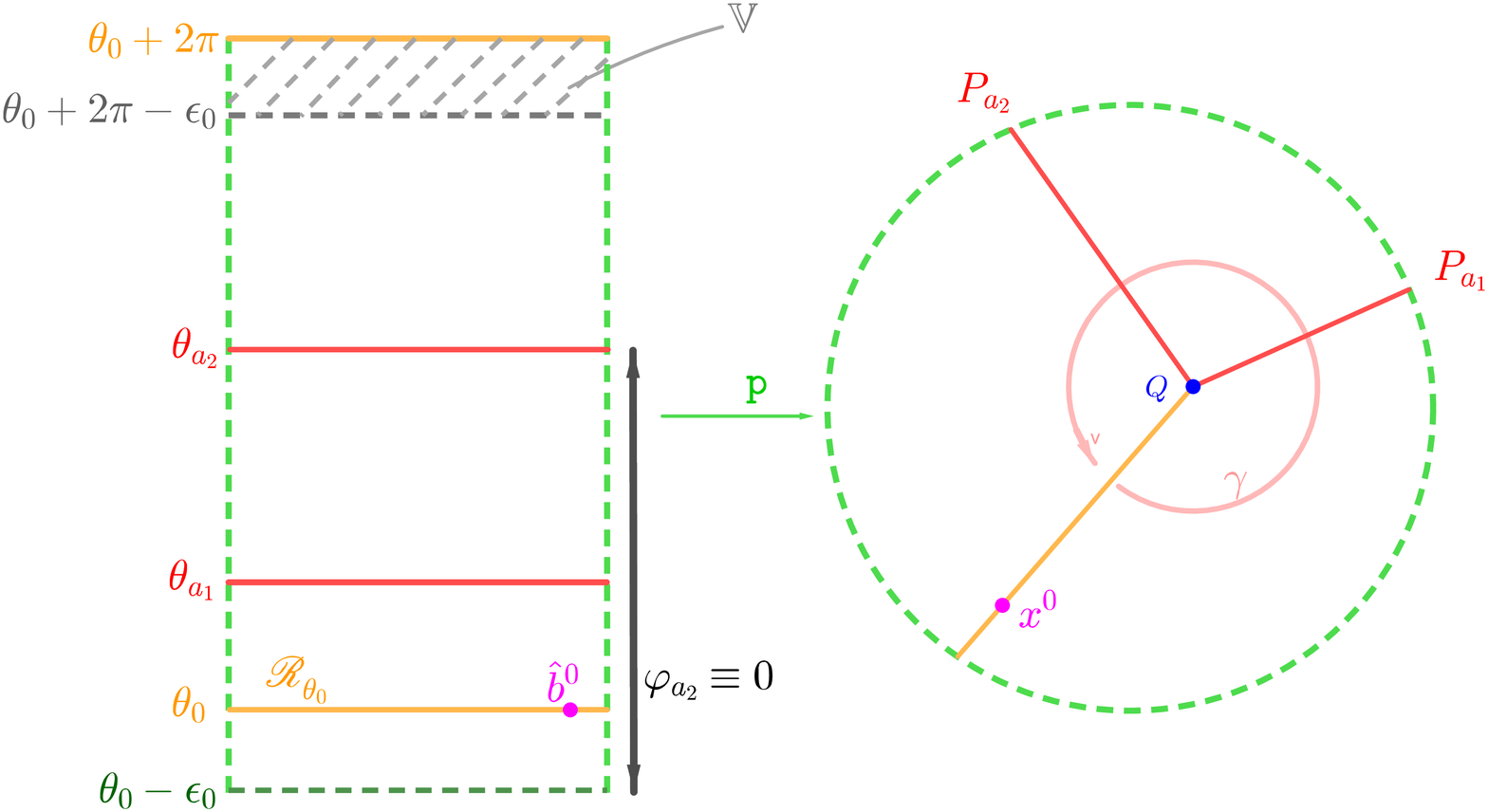}
\end{center}
\caption{}\label{fig:theorem_proof}
\end{figure}

It is enough to show that $\mathscr{D}(\Phi)_N$ is a consistent scattering diagram for each fixed $N \in \inte_{>0}$.
Recall from Definition \ref{facs_a_definition} that the gauge $\varphi_a$ is written as
$
\facs_a = - \hat{\mathcal{H}}\big(ad_{\facs_a}/(e^{ad_{\facs_a}} - \text{Id})\big) \left(\tilde{\emc}^{(a)}+\tilde{\digamma}^{(a)}\right) \ (\text{mod $\mathbf{m}^{N+1}$}),
$
and it satisfies the gauge fixing condition $\hat{\mathcal{P}}(\facs_a) = 0$ and solves the equation
$
e^{\facs_a}*0 = \tilde{\emc}^{(a)}+ \tilde{\digamma}^{(a)}
$
in $\mathbf{g}^*_N(\tilde{A}_0) / \mathcal{E}^*_N(\tilde{A}_0)$.
We first show that, given any embedded loop $\gamma$ in $U \setminus Q$ intersecting $\mathscr{D}(\Phi)$ generically (see Figure \ref{fig:theorem_proof}), we have
\begin{equation}\label{pre_consistent_equation}
\prod^{\gamma}_{a \in \mathbb{W}(N)} e^{\facs_a} = \text{Id} \quad (\text{mod $\mathbf{m}^{N+1}$}),
\end{equation}
over $\mathbb{V}$.

\begin{lemma}\label{lem:consistency}
Over $\tilde{A}_0$, we have
$
\left(\prod^{\gamma}_{a \in \mathbb{W}(N)} e^{\facs_a}\right) * 0 = \sum_{a \in \mathbb{W}(N)} \left(\tilde{\emc}^{(a)} + \tilde{\digamma}^{(a)}\right) \ (\text{mod $\mathbf{m}^{N+1}$}),
$
where the (finite) product on the left hand side is taken according to the orientation of $\gamma$.
\end{lemma}

\begin{proof}[Proof of Lemma \ref{lem:consistency}]
By Lemma \ref{asymptoticexpansion}, we have $\facs_a \equiv  0$ over the half-space $\tilde{A}_0 \setminus \hat{\mathbb{H}}(P_a) = \{(r, \ang, b) \in \tilde{A}_0 \mid \ang < \ang_a\}$ for any $a \in \mathbb{W}(N)$.
So $\text{supp}(\facs_{a'}) \cap P_{a} = \emptyset$ for any $a, a' \in \mathbb{W}(N)$ with $\ang_{a} < \ang_{a'}$ (see Figure \ref{fig:theorem_proof}).
As a result we have
$
\left[\facs_{a'} , \tilde{\emc}^{(a)} + \tilde{\digamma}^{(a)} \right] \equiv 0 \quad (\text{mod $\mathbf{m}^{N+1}$}),
$
and we get
\begin{align*}
e^{\facs_{a'}}*(\tilde{\emc}^{(a)} + \tilde{\digamma}^{(a)})
& = \left(\tilde{\emc}^{(a)} + \tilde{\digamma}^{(a)}\right) - \left(\frac{e^{ad_{\facs_{a'}}} - \text{Id}}{ad_{\facs_{a'}}}\right) \left(d\facs_{a'} + \{\tilde{\emc}^{(a)} + \tilde{\digamma}^{(a)}, \facs_{a'}\} \right)\\
& = \tilde{\emc}^{(a')} + \tilde{\digamma}^{(a')}+\tilde{\emc}^{(a)} + \tilde{\digamma}^{(a)} 
\end{align*}
in $\mathbf{g}^*_N(\tilde{A}_0) / \mathcal{E}^*_N(\tilde{A}_0)$.
The lemma follows by applying this argument repeatedly according to the anti-clockwise ordering (i.e. increasing values of $\theta_a$).
\end{proof}

On the other hand, since $\Phi$ is a Maurer-Cartan solution of $\mathbf{g}^*_N(U) / \mathcal{E}^*_N(U)$, whose deformations are all trivial in view of Lemma \ref{cohomology_lemma}, we can find a gauge $\varphi$ solving the equation $e^\facs * 0 = \Phi$ and satisfying the condition $\mathcal{P}(\facs) = 0$ over $U$.
Pulling back via $\mathtt{p}: \tilde{A}_0 \to A$, we get $\mathtt{p}^*(\facs)$ solving $e^{\mathtt{p}^*(\facs)} * 0 = \mathtt{p}^*(\Phi)$ and satisfying $\hat{\mathcal{P}}(\mathtt{p}^*(\facs))= 0$ (the latter using the fact that $\mathtt{p}(\hat{b}_0) = x^0$) over $\tilde{A}_0$.
Then uniqueness in Lemma \ref{gauge_fixing_lemma} and Lemma \ref{pre_consistent_equation} imply that
$e^{\mathtt{p}^*(\facs)} = \prod^{\gamma}_{a \in \mathbb{W}(N)}e^{\facs_a}$
in $\mathbf{g}^*_N(\tilde{A}_0) / \mathcal{E}^*_N(\tilde{A}_0)$.

But $\facs$ is defined over the whole spherical neighborhood $U$, instead of just over the annulus $A = U \setminus Q$, so in fact
$\prod^{\gamma}_{a \in \mathbb{W}(N)} e^{\facs_a} \in \mathbf{g}^*_N(\tilde{A}_0) / \mathcal{E}^*_N(\tilde{A}_0)$
is monodromy free. In particular, this tells us that
$$\left(\prod^{\gamma}_{a \in \mathbb{W}(N)} e^{\facs_a}\right)|_{\mathbb{V}-2\pi} = \left(\prod^{\gamma}_{a \in \mathbb{W}(N)} e^{\facs_a}\right)|_{\mathbb{V}}$$
modulo $\mathbf{m}^{N+1}$.
Note that $\mathbb{V}$ is chosen so that $\mathbb{V} - 2\pi$ stays away from $\bigcup_{a \in \mathbb{W}(N)}\hat{\mathbb{H}}(P_a)$. Thus we have $\left(\prod^{\gamma}_{a \in \mathbb{W}(N)} e^{\facs_a}\right)|_{\mathbb{V}-2\pi} = \text{Id} \ (\text{mod $\mathbf{m}^{N+1}$})$ by Lemma \ref{asymptoticexpansion}, so we obtain the identity
\begin{equation}\label{monodromy_free_equation_0}
\prod^{\gamma}_{a \in \mathbb{W}(N)} e^{\facs_a} = \text{Id} \ (\text{mod $\mathbf{m}^{N+1}$})
\end{equation}
over the strip $\mathbb{V}$.

Equation \eqref{monodromy_free_equation_0} is an identity in the Lie algebra $\bigoplus_{\substack{m \in \Lambda_{B_0}(U)\\1 \leq j \leq N}} \sum_{n \in \Lambda_{B_0}(U)} F^{0}_{\mathbb{V}}(\mathbb{V}) \cdot \bmc^{m} \check{\partial}_{n} t^j$.
Passing to the quotient by the ideal $\bigoplus_{\substack{m \in \Lambda_{B_0}(U)\\1 \leq j \leq N}} \sum_{n \in \Lambda_{B_0}(U)} F^{-1}_{\mathbb{V}}(\mathbb{V}) \cdot \bmc^{m} \check{\partial}_{n} t^j$
gives the identity
\begin{equation}\label{monodromy_free_equation}
\prod^{\gamma}_{a \in \mathbb{W}(N)}  e^{\psi_a} = \text{Id} \quad (\text{mod $\mathbf{m}^{N+1}$})
\end{equation}
in $
\bigoplus_{\substack{m \in \Lambda_{B_0}(U)\\1 \leq j \leq N}} \sum_{n \in \Lambda_{B_0}(U)} \left(F^{0}_{\mathbb{V}}(\mathbb{V})/F^{-1}_{\mathbb{V}}(\mathbb{V}) \right) \cdot \bmc^{m} \check{\partial}_{n} t^j
$ over $\mathbb{V}$.

The embedding
$\mathfrak{h}(\mathbb{V})\otimes_R (\mathbf{m}/\mathbf{m}^{N+1}) \hookrightarrow \bigoplus_{\substack{m \in \Lambda_{B_0}(U)\\1 \leq j \leq N}} \sum_{n \in \Lambda_{B_0}(U)} F^{0}_{\mathbb{V}}(\mathbb{V}) \cdot \bmc^{m} \check{\partial}_{n} t^j,$
whose image has trivial intersection with $\bigoplus_{\substack{m \in \Lambda_{B_0}(U)\\1 \leq j \leq N}} \sum_{n \in \Lambda_{B_0}(U)} F^{-1}_{\mathbb{V}}(\mathbb{V}) \cdot \bmc^{m} \check{\partial}_{n} t^j$ because coefficients of an element in the image are all constants independent of $\hp$,
so it descends to the quotient to give an embedding
$ \mathfrak{h}(\mathbb{V})\otimes_R (\mathbf{m}/\mathbf{m}^{N+1}) \hookrightarrow \bigoplus_{\substack{m \in \Lambda_{B_0}(U)\\1 \leq j \leq N}} \sum_{n \in \Lambda_{B_0}(U)} \left(F^{0}_{\mathbb{V}}(\mathbb{V})/F^{-1}_{\mathbb{V}}(\mathbb{V}) \right) \cdot \bmc^{m} \check{\partial}_{n} t^j.$
Therefore we obtain
$$
\prod^{\gamma}_{a \in \mathbb{W}(N)} \Theta_a = \text{Id} \quad (\text{mod $\mathbf{m}^{N+1}$})
$$
from \eqref{monodromy_free_equation} and completes the proof of the theorem.
\end{proof}


\subsubsection{Consistent scattering diagrams from more general Maurer-Cartan solutions}\label{general_consistency}
From the proof of Theorem \ref{scatteringtheorem2}, we observe a general relation between Maurer-Cartan solutions of the dgLa $\widehat{\mathbf{g}^*/\mathcal{E}^*}(U)$ with suitable asymptotic behavior and consistent scattering diagrams in $U$.

We work with a contractible open coordinate chart $U \subset B_0$ and the dgLa's $\mathbf{g}^*_N(U) / \mathcal{E}^*_N(U)$ as well as $\widehat{\mathbf{g}^*/\mathcal{E}^*}(U)$.
We also fix a codimension $2$ tropical subspace $Q \subset U$, which plays the role of the common boundary of the walls.\footnote{One can regard $Q$ as a joint in the Gross-Siebert program \cite{gross2011real} and we are indeed considering MC solutions near a joint $Q$ in $B_0$.}
In order to obtain a consistent scattering diagram from a Maurer-Cartan solution $\Phi$ of $\widehat{\mathbf{g}^*/\mathcal{E}^*}(U)$, we put two Assumptions \ref{asy_assumption_1} and \ref{asy_assumption_2} on the asymptotic behavior of $\Phi$, the first of which is the following.

\begin{assum}\label{asy_assumption_1}
We assume that $\Phi$ admits a (Fourier) decomposition of the form
\begin{equation}\label{MC_Fourier_decomp}
\Phi = \sum_{a \in \mathbb{W}} \Phi^{(a)},
\end{equation}
where we have a partition of the index set $\mathbb{W}$ into three subsets
$\mathbb{W} = \mathbb{W}_{\text{in}} \sqcup \mathbb{W}_{\text{out}} \sqcup \mathbb{W}_{\text{un}};$
here the subscripts stand for {\em incoming walls}, {\em outgoing walls} and {\em undirectional walls} respectively, following the notations in \cite{gross2011real}.
We further assume that there is an association
$a \in \mathbb{W} \mapsto m_a \in M \cong \Lambda_{B_0}(U)$
satisfying $m_a$ is not parallel to $Q$ if $a \in \mathbb{W}_{\text{in}} \sqcup \mathbb{W}_{\text{out}}$, and $m_a$ is parallel to $Q$ if $a \in \mathbb{W}_{\text{un}}$ and an association
$$a \in \mathbb{W} \mapsto \text{a tropical half-hyperplane $P_a$ containing $Q$ in $U$}$$
satisfying $P_a = Q - \real_{\leq 0} \cdot m_a$ if $a \in \mathbb{W}_{\text{in}}$, $P_a = Q - \real_{\geq 0} \cdot m_a$ if $a \in \mathbb{W}_{\text{out}}$, and $P_a \neq P_{a'}$ if $a \neq a'$ in $\mathbb{W}$,\footnote{Note that there is no restriction on $P_a$ if $a \in \mathbb{W}_{\text{un}}$, hence the name {\em undirectional walls}.}
such that the summand $\Phi^{(a)}$ has asymptotic support on $P_a$ and admits a decomposition $\Phi^{(a)} = \emc^{(a)} + \digamma^{(a)}$, where
$$
\emc^{(a)}  \in \left(\bigoplus_{k\geq 1} F^{1}_{P_a}(U) \bmc^{km_a} \check{\partial}_{n_a}\right)[[t]];\quad
\digamma^{(a)}  \in \left(\bigoplus_{k\geq 1} \sum_{n} F^{0}_{P_a}(U) \bmc^{km_a} \check{\partial}_{n}\right)[[t]];
$$
here $n_a$ is a primitive normal to $P_a$.\footnote{Note that we do not need to specify the sign of $(\nu_{P_a},n_a)$ in this assumption.}
\end{assum}

Under Assumption \ref{asy_assumption_1}, we can solve for the gauge $\facs_a$ by the same process as in Definition \ref{facs_a_definition} and prove the same statement as in Lemma \ref{lem:loc_constant_coeff} for each $\facs_a$ (because we have $(m_a, n_a) = 0$ even for undirectional walls).

Next, we consider the annulus $A := U \setminus Q$ and the universal cover $\mathtt{p}: \tilde{A} \rightarrow A$, as before.
We choose a reference half-hyperplane $\mathscr{R}_{\ang_0}$ of the form $\mathscr{R}_{\ang_0} = Q - \real_{\geq 0 } m_{\ang_0}$ with $m_{\ang_0} \in M_\real \setminus M$, so that $\mathscr{R}_{\ang_0}$ cannot overlap with any of the possible walls. Again we combine polar coordinates $(r,\ang)$ on a fiber of $NQ$ with affine coordinates $b := (b_3,\dots,b_n)$ on $Q$ to obtain coordinates $\hat{b} = (b_1 = r, b_2=\ang, b_3, \dots, b_n)$ on $\tilde{A}$.
We consider the branch $\{(r,\ang,b) \mid \ang_0 < \ang < \ang_0 + 2\pi \}$, where $\ang_0$ is a fixed angular coordinate for the half-hyperplane $\mathscr{R}_{\ang_0}$. For each $a \in \mathbb{W}$, we let $\ang_0 < \ang_a < \ang_0 + 2\pi$ be the angular coordinate of the half-hyperplane $P_a$ and set $\hat{\mathbb{H}}(P_a) := \left\{(r,\ang,b) \mid \ang_a \leq \ang < \ang_0 + 2\pi \right\}$.

\begin{assum}\label{asy_assumption_2}
We assume that there exists an element
$\psi_a = \sum_{j,k \geq 1}  b^{(a)}_{jk} \bmc^{km_a} \check{\partial}_{n_a} t^j,$
where $b^{(a)}_{jk}$'s are constants independent of $\hp$ with $b^{(a)}_{jk} \neq 0$ only for finitely many $k$'s for each fixed $j$
and $n_a$ is a primitive normal to $P_a$, such that
$$
\hat{\mathcal{H}}(\emc^{(a)})|_{\hat{\mathbb{H}}(P_a) \setminus P_a} \in \psi_a + \left(\bigoplus_{k\geq 1} F^{-1}_{\hat{\mathbb{H}}(P_a) \setminus P_a}(\hat{\mathbb{H}}(P_a) \setminus P_a)\cdot \bmc^{km_a} \check{\partial}_{n_a}\right)[[t]]\cdot t.
$$
\end{assum}

For each fixed $N\in \inte_{>0}$, we choose a sufficiently small $\epsilon_N > 0$ such that the subset $\mathbb{V}_N := \{(r,\ang,b) \mid \ang_0 - \epsilon_{N} + 2\pi < \ang < \ang_0 + 2\pi\}$ is disjoint from all the $P_a$'s. We then restrict our attention to $\tilde{A}_0 := \{ (r,\ang,b) \mid \ang_0 - \epsilon_N < \ang < \ang_0 + 2\pi\}$ in order to apply a monodromy argument as in Section \ref{sec:diagram_associated}. We also fix the homotopy operator $\hat{\mathcal{H}}$ as in Definition \ref{polarrealhomotopy}, together with $\hat{\mathcal{P}}$ and $\hat{\iota}$. Then we can prove the same statement as in Lemma \ref{asymptoticexpansion} under Assumption \ref{asy_assumption_2}.


So altogether, assuming both Assumptions \ref{asy_assumption_1} and \ref{asy_assumption_2}, we have Lemmas \ref{lem:loc_constant_coeff} and \ref{asymptoticexpansion}, and a scattering diagram $\mathscr{D}(\Phi)$ can be associated to the given Maurer-Cartan solution $\Phi$ in exactly the same way as in Definition \ref{def:construct_diagram}. Finally, the same proof as in Theorem \ref{scatteringtheorem2} gives the following:

\begin{theorem}[= Theorem \ref{theorem1}]\label{scatteringtheorem1}
Suppose that we have a Maurer-Cartan solution $\Phi$ of $\widehat{\mathbf{g}^*/\mathcal{E}^*}(U)$ satisfying both Assumptions \ref{asy_assumption_1} and \ref{asy_assumption_2}. Then the scattering diagram $\mathscr{D}(\Phi)$ associated to $\Phi$ is consistent, i.e. we have the following identity
$
\Theta_{\gamma, \mathscr{D}(\Phi)} = \prod^{\gamma}_{\mathbf{w}_a \in \mathscr{D}(\Phi)} \Theta_a = \text{Id}
$
along any embedded loop $\gamma$ in $U \setminus \text{Sing}(\mathscr{D}(\Phi))$ intersecting $\mathscr{D}(\Phi)$ generically.
\end{theorem}

\bibliographystyle{amsplain}
\bibliography{geometry}

\end{document}